\documentclass[a4paper, 11pt]{scrartcl}

\usepackage[numberTheoremsWithin=section, numberEquationsWithinTheorems]{mathEnv}
\usepackage[cmtip,arrow,all]{xy}
\CompileMatrices

\usepackage{coseoul}
\usepackage[bibencoding=ascii, style=alphabetic]{biblatex}
\setkomafont{subparagraph}{\normalfont\sffamily\itshape}

\usepackage[]{main_arxiv}%

\newrefformat{chp}{\hyperref[{#1}]{Chapter~\ref*{#1}}}
\newrefformat{sec}{\hyperref[{#1}]{Section~\ref*{#1}}}
\newrefformat{susec}{\hyperref[{#1}]{Subsection~\ref*{#1}}}
\newrefformat{sususec}{\hyperref[{#1}]{Subsection~\ref*{#1}}}
\newrefformat{app}{\hyperref[{#1}]{Appendix~\ref*{#1}}}

\newrefformat{def}{\hyperref[{#1}]{Definition~\ref*{#1}}}

\newrefformat{ivp}{\hyperref[{#1}]{initial value problem~\ref*{#1}}}
\newrefformat{est}{\hyperref[{#1}]{estimate~\ref*{#1}}}
\newrefformat{id}{\hyperref[{#1}]{identity~\ref*{#1}}}
\newrefformat{bed}{\hyperref[{#1}]{condition~\ref*{#1}}}
\newrefformat{char}{\hyperref[{#1}]{characterization~\ref*{#1}}}

\newrefformat{enum1}{\ref{#1}}

\setlist[enumerate,1]{label=(\alph*)}
\setlist[enumerate,2]{label=\roman*)}
\setlist[enumerate,3]{label=\arabic*.}
\setlist[enumerate,4]{label=\Alph*}

\hypersetup{%
	breaklinks,
	plainpages=false
}

\newcommand{\R}{\ensuremath{\mathbb{R}}}
\newcommand{\C}{\ensuremath{\mathbb{C}}}
\newcommand{\N}{\ensuremath{\mathbb{N}}}
\newcommand{\K}{\ensuremath{\mathbb{K}}}
\newcommand{\Z}{\ensuremath{\mathbb{Z}}}
\newcommand{\Q}{\ensuremath{\mathbb{Q}}}

\newcommand{\eps}{\ensuremath{\varepsilon}}

\DeclareMathOperator{\idSymb}{\mathrm{id}}
\newcommand{\id}[1]{\idSymb_{#1}}
\newcommand{\iso}{\ensuremath \cong}
\newcommand{\sub}{\ensuremath{\subseteq}}
\newcommand{\set}[2]{\ensuremath{\{#1 : #2\} }}
\newcommand{\sset}[1]{\ensuremath{\{#1\}}}
\newcommand{\ndef}{\colonequals}
\newcommand{\defn}{\equalscolon}
\newcommand{\rest}[3][]{\ensuremath{#2|_{#3}\ifthenelse{ \equal{#1}{} }{}{^{#1}}} }
\newcommand{\image}[1]{\ensuremath{\mathrm{im}\, #1 }}

\newcommand{\SkaPrd}[2]{\ensuremath{\langle #1, #2\rangle}}

\DeclareMathOperator{\distop}{\mathrm{dist}}
\newcommand{\dist}[2]{\ensuremath \distop(#1,#2)}
\newcommand{\abs}[1]{\ensuremath \lvert #1\rvert}
\newcommand{\norm}[1]{\ensuremath \lVert #1\rVert}
\newcommand{\MetaBall}[4][]
{
	\ensuremath
	{
	\ifthenelse{ \equal{#1}{} }
	{
		 #4_{#3}(#2)
	}
	{
		#4_{#1}(#2, #3)
	}
	}
}
\newcommand{\Ball}[3][]{\MetaBall[#1]{#2}{#3}{B}}
\newcommand{\clBall}[3][]{ \MetaBall[#1]{#2}{#3}{\cl{B}} }

\newcommand{\supp}[1]{\ensuremath{\mathrm{supp}(#1)}}
\newcommand{\neigh}[2][]{\ensuremath{ \mathcal{U}\ifthenelse{ \equal{#1}{} }{}{_{#1}}(#2) }}
\newcommand{\neighO}[2][]{\ensuremath{ \mathcal{U}\ifthenelse{ \equal{#1}{} }{}{_{#1}}^\circ(#2) }}
\newcommand{\cl}[1]{\overline{#1}}

\newcommand{\fami}[3]{\ensuremath{(#1)_{#2 \in #3}}}

\newcommand{\famiI}[1]{\fami{#1}{i}{I}}

\newcommand{\dotcup}{\ensuremath{\mathaccent\cdot\cup}}
\newcommand{\disjointU}[2][]{\dotcup\ifthenelse{\equal{#1}{}}{}{_{#1}}#2}
\newcommand{\disjointUI}[1]{\disjointU[i\in I]{#1}}

\newcommand{\Rint}[5][\int]{\ensuremath {#1}_{#2}^{#3} #4\, d #5}
\newcommand{\Mint}[3][\int]{\Rint[#1]{0}{1}{#2}{#3}}

\newcommand{\dA}[4][]{
	\ensuremath{
	d\ifthenelse{ \equal{#1}{} }{}{^{(#1)}} 
		#2 \ifthenelse{\equal{#3}{} \and \equal{#4}{} }{}{(#3;#4)}
		}
}
\newcommand{\FAbl}[2][]{
	\ensuremath{
	D\ifthenelse{ \equal{#1}{} }{}{^{(#1)}}#2
	}
}
\newcommand{\parFAbl}[3][]{ \ensuremath{ \FAbl[#1]{_{#2} #3} } }

\newcommand{\LLA}[2][]{ \ensuremath{\delta_{\ell}\ifthenelse{\equal{#2}{} }{}{(#2)} } }
\newcommand{\RLA}[2][]{ \ensuremath{\delta_{\rho}\ifthenelse{\equal{#2}{} }{}{(#2)} } }

\newcommand{\Tang}[2][]{
\ensuremath{\mathbf{T}\ifthenelse{ \equal{#1}{} }{}{_{#1} }
	\ifthenelse{ \equal{#2}{} }%
	{}%
	{#2}%
	}
}
\newcommand{\VecFields}[2][]{\ensuremath{\mathfrak{X}^{#1}(#2)}}

\DeclareMathOperator{\DiffSym}{\mathrm{Diff}}

\newcommand{\ConDiff}[4][]
{
	\ensuremath{
	\mathcal{C}^{#4}_{#1}
	\ifthenelse{ \equal{#2}{} \and \equal{#3}{} }
	{}
	{(#2, #3)}
	}
}

\newcommand{\FC}[4][]
{
	\ensuremath{
	\mathcal{FC}^{#4}
	\ifthenelse{ \equal{#2}{} \and \equal{#3}{} }
	{}
	{(#2, #3)}
	}
}

\newcommand{\GewFunk}{\cW}
\newcommand{\hn}[3]{\ensuremath\norm{#1}_{#2, #3}}

\newcommand{\CF}[4]{\ensuremath\mathcal{C}^{#4}_{#3}(#1, #2)}
\newcommand{\CFo}[4]{\CF{#1}{#2}{#3}{\partial, #4}}

\newcommand{\CcF}[3]{\CF{#1}{#2}{\GewFunk}{#3}}

\newcommand{\CcFo}[3]{\CFo{#1}{#2}{\GewFunk}{#3}}

\newcommand{\CFpro}[6]{\CF{#1}{#2}{#3}{#4}_{#5 \in #6}}

\newcommand{\CcFproI}[3]{\CFpro{#1}{#2}{\GewFunk}{#3}{i}{I}}
\newcommand{\CFfoPRO}[7]{\CFpro{#1}{#2}{#3}{{#5}_\partial, #4}{#6}{#7}}

\newcommand{\CcFfoPROi}[4]{\CFfoPRO{#1}{#2}{\GewFunk}{#3}{#4}{i}{I}}

\newcommand{\CFoPRO}[6]{\CFo{#1}{#2}{#3}{#4}_{#5 \in #6}}
\newcommand{\CcFoPRO}[5]{\CcFo{#1}{#2}{#3}_{#4 \in #5}}
\newcommand{\CcFoPROi}[3]{\CcFoPRO{#1}{#2}{#3}{i}{I}}

\newcommand{\BC}[3]{\ensuremath{\mathcal{BC}^{#3}(#1, #2)}}

\newcommand{\DCc}[3]{\CF{#1}{#2}{c}{#3}}
\newcommand{\DCcInf}[2]{\DCc{#1}{#2}{\infty}}

\newcommand{\hnM}[4]{\hn{ #1}{#4, #2}{#3}}
\newcommand{\CFM}[4]{\CF{#1}{\Tang{#1}}{#2}{#3}_{#4}}
\newcommand{\CcFM}[3]{\CFM{#1}{\GewFunk}{#2}{#3}}

\newcommand{\DCcInfM}[1]{\DCcInf{#1}{\Tang{#1}}}

\newcommand{\CcFproAtK}[5]{\CFpro{#1_\kappa}{#2}{\WeightsAtl{\GewFunk}{#3}}{#4}{\kappa}{#5}}
\newcommand{\CcFfoPROatK}[6]{\CFfoPRO{#1_\kappa}{#2}{\WeightsAtl{\GewFunk}{#3} }{#4}{#5}{\kappa}{#6}}
\newcommand{\CcFoPROatK}[5]{\CFoPRO{#1_\kappa}{#2}{\WeightsAtl{\GewFunk}{#3} }{#4}{\kappa}{#5}}

\newcommand{\embMWeightsAtl}[2]{\ensuremath{\iota_{#1}^{#2}}}
\newcommand{\embMcWatl}[1]{\embMWeightsAtl{\GewFunk}{#1}}

\newcommand{\Diff}[3]{
	\ensuremath{\DiffSym^{#2}_{#3}(#1)}
}
\newcommand{\DiffW}{\Diff{\SX}{}{\GewFunk}}

\newcommand{\DiffcM}{\Diff{M}{}{c}}

\newcommand{\compIdDiffKLWeights}[4]{\mathfrak{c}_{#4, #3}^{#1, #2}}

\newcommand{\compIdWeights}[2]{\mathfrak{c}_{#2}^{#1}}

\newcommand{\compIdDiffKLcW}[3]{\compIdDiffKLWeights{#1}{#2}{#3}{\GewFunk}}

\newcommand{\InvIdWeights}[2]{I^{#1}_{#2}}
\newcommand{\InvIdcW}[1]{\InvIdWeights{#1}{\GewFunk}}

\newcommand{\cW}{\ensuremath{\mathcal{W}}}

\newcommand{\ExtWeights}[1]{\ensuremath{#1_{\mathrm{max}}}}

\newcommand{\noma}[1]{\norm{#1}_\infty}
\newcommand{\Opnorm}[1]{\norm{#1}_{op}}

\DeclareMathOperator{\Id}{\mathrm{Id}}
\newcommand{\idco}{\ensuremath{\Id}}
\newcommand{\MaMu}{\ensuremath{\cdot}}
\newcommand{\eval}{\ensuremath{\cdot}}

\newcommand{\Lin}[3][]{
\ensuremath{
\ifthenelse{\equal{#1}{}}
	{
		\ifthenelse{\equal{#2}{#3}}
			{\mathrm{L}(#2)}
			{\mathrm{L}(#2,#3)}
	}
	{
		\mathrm{L}^{#1}(#2, #3)
	}
}
}
\newcommand{\EmbA}[3][]{\ensuremath{\mathcal E}_{#2,#3}\ifthenelse{\equal{#1}{}}{}{^{#1}}}

\newcommand{\Evol}[3]{\ensuremath{\mathrm{Evol}_{#1}^{#3}\ifthenelse{\equal{#2}{}}{}{(#2)}}}
\newcommand{\evol}[3]{\ensuremath{\mathrm{evol}_{#1}^{#3}\ifthenelse{\equal{#2}{}}{}{(#2)}}}

\newcommand{\SX}{\ensuremath X}
\newcommand{\SY}{\ensuremath Y}

\newcommand{\UF}{\ensuremath U}
\newcommand{\VF}{\ensuremath V}

\newcommand{\G}{G}

\newcommand{\Rexp}[1]{\ensuremath{\exp_{#1}}}
\newcommand{\RexpPar}[1]{\ensuremath{\exp_{#1}}}
\newcommand{\Rlog}[1]{\ensuremath{\log_{#1}}}
\DeclareMathOperator{\lgsymb}{\mathrm{lg}}
\newcommand{\RlogVR}[1]{\ensuremath{\lgsymb_{#1}}}
\newcommand{\RlogPar}[1]{\Rlog{#1}}

\newcommand{\domExpMax}[1]{\ensuremath{D^E_{#1}}}
\newcommand{\domLogMax}[1]{\ensuremath{D^L_{#1}}}

\newcommand{\grenzExp}[3][]{\ensuremath{R^{E\ifthenelse{\equal{#1}{}}{}{, #1}}_{#2, #3}}}
\newcommand{\grenzLog}[3][]{\ensuremath{R^{L\ifthenelse{\equal{#1}{}}{}{, #1}}_{#2, #3}}}

\newcommand{\BndFstAblRex}[3][]{\ensuremath{C^{E, (1)}_{#2, #3\ifthenelse{\equal{#1}{}}{}{, #1}}}}
\newcommand{\BndSndAblRex}[3][]{\ensuremath{C^{E, 2}_{#2, #3\ifthenelse{\equal{#1}{}}{}{, #1}}}}
\newcommand{\BndFstAblRlog}[3][]{\ensuremath{C^{L, (1)}_{#2, #3\ifthenelse{\equal{#1}{}}{}{, #1}}}}
\newcommand{\BndSndAblRlog}[3][]{\ensuremath{C^{L, 2}_{#2, #3\ifthenelse{\equal{#1}{}}{}{, #1}}}}

\newcommand{\VectFLok}[2]{\ensuremath{#1_{#2}}}
\newcommand{\RMetLok}[2]{\ensuremath{#1_{#2}}}
\newcommand{\weightLok}[2]{\ensuremath{#1_{#2}}}
\newcommand{\WeightsLok}[2]{\ensuremath{#1_{#2}}}

\newcommand{\WeightsAtl}[2]{\ensuremath{#1_{#2}}}
\newcommand{\weightAtl}[2]{\ensuremath{#1_{#2}}}

\newcommand{\atProd}[2]{\ensuremath{#1 \otimes #2}}
\newcommand{\atCap}[2]{\ensuremath{{#1}^{\cap #2}}}

\newcommand{\LogPlusVR}[3]{\ensuremath{\mathcal{L}_{#1, #2}^{#3}}}
\newcommand{\ExpMinusLok}[3]{\ensuremath{\mathcal{E}_{#1, #2}^{#3}}}

\newcommand{\QuotNorm}[2][g]{\ensuremath{Q_{#2}^{#1}}}

\newcommand{\disjointUkM}[1]{\disjointU[\kappa\in \atlasA]{#1}}
\newcommand{\famiAtl}[2]{\fami{#1}{\kappa}{#2}}

\newcommand{\ExpIdWeights}[4][]{\ensuremath{E_{#2, #3}^{#4\ifthenelse{\equal{#1}{}}{}{,#1}} }}
\newcommand{\ExpIdCWprod}[3][g]{\ExpIdWeights[#1]{#2}{#3}{\WeightsAtl{\GewFunk}{\atlasB}}}
\newcommand{\ExpIdCWfak}[3][\RMetLok{g}{\kappa}]{\ExpIdWeights[#1]{#2}{#3}{\WeightsLok{\GewFunk}{\kappa}}}

\newcommand{\LogIdWeights}[4][]{\ensuremath{L_{#2, #3}^{#4\ifthenelse{\equal{#1}{}}{}{,#1}} }}
\newcommand{\LogIdCWprod}[3][g]{\LogIdWeights[#1]{#2}{#3}{\WeightsAtl{\GewFunk}{\atlasB}}}
\newcommand{\LogIdCWfak}[3][\RMetLok{g}{\kappa}]{\LogIdWeights[#1]{#2}{#3}{\WeightsLok{\GewFunk}{\kappa}}}

\newcommand{\extMult}[2]{\ensuremath{\mathcal{M}_{#1, #2}}}

\newcommand{\RadExpFibInv}[3][]{\ensuremath{R_{#2, #3}\ifthenelse{\equal{#1}{}}{}{^{#1}} }}

\newcommand{\atlasA}{\ensuremath{\mathcal{A}}}
\newcommand{\atlasB}{\ensuremath{\mathcal{B}}}
\newcommand{\atlasC}{\ensuremath{\mathcal{C}}}

\newcommand{\minSatExt}[1][]{\ifthenelse{ \equal{#1}{M} }{M}{m}inimal saturated extension}

\newcommand{\DiffWmFKaTrMET}[6]{\Diff{#1,#2, #3}{#4, #5}{#6}}
\newcommand{\DiffWMg}{\DiffWmFKaTrMET{M}{g}{\omega}{\atlasA}{\atlasB}{\GewFunk}}
\newcommand{\DiffWvr}[2]{\DiffWmFKaTrMET{#1}{#2}{1_{#1}}{\atlasA}{\atlasB}{\GewFunk}}

\begin{document}

\title{Weighted diffeomorphism groups of Riemannian manifolds}
\author{Boris Walter}
\date{}
\publishers{
\small{}
Universität Paderborn\\ Institut für Mathematik\\
Warburger Straße 100\\ 33098 Paderborn\\
E-Mail: bwalter@math.upb.de
}

\maketitle
\begin{abstract}
	\noindent{}In this paper, we define locally convex vector spaces of
	weighted vector fields
	and use them as model spaces for Lie groups of weighted diffeomorphisms on Riemannian manifolds.
	We prove an easy condition on the weights that ensures
	that these groups contain the compactly supported diffeomorphisms.
	We finally show that for the special case where the manifold is the euclidean space,
	these Lie groups coincide with the ones constructed in the author's earlier work
	[\enquote{Weighted diffeomorphism groups of Banach spaces and weighted mapping groups}.
	In: Dissertationes Math. 484 (2012), p. 128. DOI: \href{http://dx.doi.org/10.4064/dm484-0-1}{10.4064/dm484-0-1}].
\end{abstract}

\section{Introduction}

Diffeomorphism groups of compact manifolds
are among the most important and well-studied examples of infinite dimensional Lie groups
(see for example \cite{MR0210147}, \cite{MR830252}, \cite{MR656198}, \cite{MR1421572} and \cite{MR1471480}).
While the diffeomorphism group $\Diff{K}{}{}$ of a compact manifold is modelled
on the Fréchet space $\ConDiff{K}{\Tang{K} }{\infty}$ of smooth vector fields on $K$,
for a non-compact smooth manifold $M$, it is not possible to make $\Diff{M}{}{}$
a Lie group modelled on the space of all smooth vector fields in a satisfying way
(see \cite{MilnorPreprint82}).
We mention that the LF-space $\CF{M}{\Tang{ M}}{c}{\infty}$
of compactly supported smooth vector fields can be used as the modelling space
for a Lie group structure on $\Diff{M}{}{}$.
But the topology on this Lie group is too fine for many purposes;
the group $\DiffcM$ of compactly supported diffeomorphisms
(those diffeomorphisms that coincide with the identity map outside some compact set)
is an open subgroup (see \cite{MR583436} and \cite{MilnorPreprint82}).

In view of these limitations, it is natural to look for Lie groups of diffeomorphisms
which are larger then $\Diff{M}{}{c}$ and modelled on larger Lie algebras of vector fields
than $\CF{M}{\Tang{ M}}{c}{\infty}$.
In \cite{MR2952176},
the author already turned
\[
	\Diff{M}{}{\cW} = \set{\phi \in \Diff{M}{}{}}%
		{\phi - \id{M},\phi^{-1} - \id{M} \in \CF{M}{M}{\cW}{\infty}}
\]
into an infinite-dimensional Lie group
(modelled on the space $\CF{M}{M}{\cW}{\infty}$ of smooth weighted functions)
when $M$ is a (possibly infinite-dimensional) Banach space.

For $M = \R$ and $\cW$ the set of polynomials,
this includes the Lie group of rapidly decreasing diffeomorphisms treated in the earlier work
\cite{MR2263211}
(compare also \cite{MR3132089}, \cite{MR3313140} and
\cite{MR3437748} for later developments).
In this work, which contains results of the author's dissertation,
we extend the construction from \cite{MR2952176}
to the case where $M$ is a Riemannian manifold.

As a motivating example, consider the direct product
\[
	M \ndef \R \times {\mathbb S}
\]
of the real line and the circle group.
Then smooth vector fields on $M$ can be identified with smooth functions
\[
	\gamma : \R\times \R \to \R^2
	: (x,y)\mapsto \gamma(x,y)
\]
which are $2\pi$-periodic in the $y$-variable.
To control the asymptotic behaviour of vector fields (and associated diffeomomorphisms) at infinity,
it is natural to impose that $\gamma$ (and its partial derivatives)
decay polynomially as $x\to \pm\infty$ in the sense that, for each $n \in \N$,
\[
	x^n\gamma(x,y)
\]
is bounded for $(x,y)\in \R^n$ (and hence $\gamma$ tends to $0$ as $x\to \pm \infty$).%
\footnote{Likewise, we could impose that $\gamma$ and all its partial derivatives
are bounded, or have exponential decay.}
The preceding approach hinges on the very specific situation considered;
namely, that we have the local diffeomorphism $q : (t,s) \mapsto (t,e^{i s})$
from $\R^2$ (on which vector fields can be identified with smooth functions $\R^2\to\R^2$)
onto $\R \times {\mathbb S}$.
Of course, one would like to be able to describe weighted vector fields as just encountered
also without reference to $q$, and for general manifolds
(none of whose covering manifolds
need to admit a global chart).

This is achieved in this paper, in the following form.
Let $(M, g)$ be a Riemannian manifold,
and $\atlasA$ an atlas for $M$ that is \emph{adapted} (see \refer{def:adapted_atlas} for the precise meaning) and \enquote{thin}.
In \refer{satz:Existenz_Liegruppe_Diffeos_MFK}, we prove:
\begin{thmOhneNummer}
	Let $\GewFunk \sub \cl{\R}^M$ with $1_M \in \GewFunk$.
	Then there exists a Lie group $\DiffWMg$ of \emph{weighted diffeomorphisms}
	that is a subgroup of $\Diff{M}{}{}$
	and modelled on the space $\CFM{M}{\GewFunk^e}{\infty}{\atlasA}$
	of weighted vector fields with regard to $\atlasA$.
	Further, the Riemannian logarithm provides a chart for $\DiffWMg$.
	Here $\atlasB$ denotes a suitable subatlas of $\atlasA$, and
	$\GewFunk^e$ denotes a \emph{\minSatExt{}} of $\GewFunk \cup \sset{\omega}$,
	where $\omega$ is an \emph{adjusted weight}.
\end{thmOhneNummer}
The terms \emph{\minSatExt{}} and \emph{adjusted weights} are defined in \refer{def:saturated_adjusted-weights};
and the model spaces $\CFM{M}{\GewFunk^e}{\infty}{\atlasA}$ in \refer{susec:Def_Weighted_VFs}.
Further, we prove in \refer{prop:Inclusion_of_compactly_supported_diffeomorphisms} that
the groups $\DiffWMg$ contain at least the identity component of $\DiffcM$,
provided that $\GewFunk$ consists of continuous weights.
Finally, we show in \refer{prop:Comparison_DiffLieGroups_VectorS}
that if $(M, g) = (\R^d, \SkaPrd{\cdot}{\cdot})$ and $\atlasA$ consists of identity mappings,
then the connected identity components and the topology of $\Diff{\R^d}{}{\GewFunk}$ and $\DiffWvr{\R^d}{\SkaPrd{\cdot}{\cdot}}$ coincide,
giving us plenty of examples for this construction.

In the case where $\cW = \sset{1_\SX}$, our group $\DiffW$ also has a counterpart
in the studies of Jürgen Eichhorn and collaborators (\cite{MR1856078}, \cite{MR1392288}, \cite{MR2343536}),
who studied certain diffeomorphism groups on non-compact manifolds with bounded geometry.
While an affine connection is used there to deal with higher derivatives,
we are working exclusively with derivatives in local charts.

\section{Definitions and previous results}
Before we start, we have to repeat some of the notation and results of \cite{Walter_Diss_p1} (and \cite{MR2952176}).
We set $\cl{S} \ndef S \cup \sset{\infty}$ for $S \in \sset{\R, \N}$,
and $\N \ndef \sset{0, 1, \dotsc}$.
Other notation is introduced when it is first used.

\subsection{Restricted products of weighted functions}

In \cite{MR2952176}, for normed spaces $\SX$, $\SY$, open $\UF \sub \SX$,
a set $\cW \sub \cl{\R}^\UF$ of weights and $k \in \cl{\N}$,
we defined weighted function spaces $\CF{\UF}{\SY}{\cW}{k}$.
Basically, these are the locally convex vector spaces defined by the quasinorms
\[
	\hn{\cdot}{f}{\ell}
	:\FC{\UF}{\SY}{k}\to [0,\infty]
	: \phi \mapsto \sup\set{\abs{f(x)}\,\Opnorm{\FAbl[\ell]{\phi}(x)} }{x\in \UF}
\]
(where $\ell \leq k$) on the space of $k$ times continuously Fréchet differentiable functions from $\UF$ to $\SY$. 

In \cite{Walter_Diss_p1}, we defined \emph{restricted products} of such weighted function spaces
and obtained results about the smoothness of operations between such spaces.
We give the definitions and results that are necessary for our discussion of weighted vector fields.

\subsubsection{Definitions and topological properties}

\begin{defi}[Weighted restricted products, maximal extensions]\label{def:res_prods-max_weights}
	Let $I$ be a nonempty set, $\famiI{U_i}$ a family such that each $U_i$
	is an open nonempty set of a normed space $X_i$,
	$\famiI{Y_i}$ another family of normed spaces,
	$\GewFunk \sub \cl{\R}^{\disjointUI{U_i}}$ a nonempty family of weights
	defined on the disjoint union $\disjointUI{U_i}$ of $\famiI{U_i}$,
	and $k \in \cl{\N}$.
	For $i \in I$ and $f \in \GewFunk$, we set $f_i \ndef \rest{f}{U_i}$,
	and further $\GewFunk_i \ndef \set{f_i}{f \in \GewFunk}$.
	We define
	\[
		\CcFproI{U_i}{Y_i}{k}
		\ndef	
		\set{\famiI{\phi_i} \in \prod_{i \in I} \FC{U_i}{Y_i}{k} }%
		{(\forall f \in \GewFunk, \ell \in \N, \ell \leq k)\, \sup_{i \in I} \hn{\phi_i}{f_i}{\ell} < \infty}.
	\]
	The seminorms that generate the topology on this space then are
	\[
		\hn{\famiI{\phi_i} }{f}{\ell}
		\ndef
		\sup_{i \in I} \hn{\phi_i}{f_i}{\ell},
	\]
	where $f \in \GewFunk$ and $\ell \in \N$ with $\ell \leq k$.
	
	Further, we define the \emph{maximal extension}
	$\ExtWeights{\GewFunk} \subseteq \cl{\R}^{\disjointUI{U_i}}$ of $\GewFunk$
	as the set of functions $f$ for which $\hn{\cdot}{f}{0}$ is a continuous seminorm
	on $\CcFproI{U_i}{Y_i}{0}$, for each family $\famiI{Y_i}$ of normed spaces.
	Obviously $\GewFunk \subseteq \ExtWeights{\GewFunk}$
	and we can show that $\hn{\cdot}{f}{\ell}$ is a continuous seminorm on each $\CcFproI{U_i}{Y_i}{\ell}$,
	provided that $f \in \ExtWeights{\GewFunk}$ and $\ell \leq k$.
\end{defi}
A useful tool for dealing with these spaces is the following, found in \cite[§ 4.2.1]{Walter_Diss_p1}: 
\begin{prop}\label{prop:Zerlegungssatz_Familie}
	Let $k \in \N$. Then for $\famiI{\phi_i} \in \prod_{i \in I} \FC{U_i}{Y_i}{1}$,
	we have
	\[
		\famiI{\phi_i} \in \CcFproI{U_i}{Y_i}{k + 1}
		\iff
		\famiI{\phi_i} \in \CcFproI{U_i}{Y_i}{0}
		\text{ and }
		\famiI{\FAbl{\phi_i}} \in \CcFproI{U_i }{ \Lin{X_i}{Y_i} }{k}.
	\]
	The map
	\begin{equation*}
		\CcFproI{U_i}{Y_i}{k + 1}
		\to
		\CcFproI{U_i}{Y_i}{0}
		\times
		\CcFproI{U_i }{ \Lin{X_i}{Y_i} }{k}
		:
		(\famiI{\phi_i}) \mapsto (\famiI{\phi_i}, \famiI{\FAbl{\phi_i} })
	\end{equation*}
	is linear and a topological embedding.
\end{prop}

\paragraph{Adjusting weights and open subsets of restricted products}

We give the definition of adjusting weights, and provide some results that will later be used.
These weights are important because they ensure
there are sufficiently many open subsets in restricted products
to make them useful as model spaces for topologies on vector fields.

Further information and proofs for the claims can be found in \cite[§ 4.2.2]{Walter_Diss_p1}.

\begin{defi}[Adjusting weights]
	Let $\famiI{U_i}$ and $\famiI{r_i}$ be families such that each $U_i$
	is an open nonempty set of the normed space $X_i$, and each $r_i \in ]0, \infty]$.
	We say that $\omega : \disjointUI{U_i} \to \R$ is
	an \emph{adjusting weight for $\famiI{r_i}$} if for each $i \in I$,
	we have that
	\[
		\sup_{x \in U_i} \abs{\omega_i(x)} < \infty
		\qquad\text{and}\qquad
		\inf_{x \in U_i} \abs{\omega_i(x)} \geq \max\bigl(\tfrac{1}{r_i}, 1\bigr).
	\]
	Notice that generally, $\omega$ itself is \emph{not} bounded.
\end{defi}
\begin{defi}
	Let $\famiI{U_i}$ and $\famiI{V_i}$ be families such that each $U_i$
	is an open nonempty set of the normed space $X_i$
	and each $V_i$ is an open nonempty subset of a normed space $Y_i$,
	$\GewFunk \sub \cl{\R}^{\disjointUI{U_i}}$ a nonempty set
	and $k \in \cl{\N}$.
	Let $\omega : \disjointUI{U_i} \to \R$ with $0 \notin \omega(\disjointUI{U_i})$.
	We set
	\begin{multline*}
		\CcFfoPROi{U_i}{V_i}{k}{\omega}
		\ndef
		\set{\famiI{\gamma_i} \in \CcFproI{U_i}{Y_i}{k}}%
		{\\(\exists r > 0)(\forall i \in I, x \in U_i)\,\gamma_i(x) + \Ball[Y_i]{0}{\tfrac{r}{\abs{\omega(x)}}} \sub V_i}.
	\end{multline*}
	In particular, we define
	\[
		\CcFoPROi{U_i}{V_i}{k} \ndef \CcFfoPROi{U_i}{V_i}{k}{(1_{\disjointUI{U_i}})}.
	\]
	Additionally, if each $V_i$ is star-shaped with center $0$,
	then $\omega$ is called an \emph{adjusting weight for $\famiI{V_i}$} if it is an adjusting weight for
	$
		\famiI{\dist{\sset{0}}{\partial V_i}}.
	$
	If it is clear to which family $\omega$ adjusts, we may call $\omega$ just an adjusting weight.
\end{defi}
The next two results show that the sets $\CcFfoPROi{U_i}{V_i}{k}{\omega}$ have nonempty interior,
provided that $\GewFunk$ contains an adjusting weights for $\famiI{V_i}$.
In fact, we can show that they contain a ball w.r.t. the adjusting weight.
\begin{lem}\label{lem:est_1-0-norm_f-0-norm}
	Let $X$ and $Y$ be normed spaces, $U \sub X$ an open nonempty set,
	$f  : U \to \cl{\R}$ such that
	$\inf_{x \in U} \abs{f(x)} \geq \max(\tfrac{1}{d}, 1)$ (where $d > 0$)
	and $\phi, \psi : U \to Y$. Then
	\begin{equation}\label{est:1-0-norm_f-0-norm_spezielles-f}
		\hn{\phi - \psi}{1_U}{0} \leq \min(d, 1) \hn{\phi - \psi}{f}{0}.
	\end{equation}
\end{lem}

\begin{lem}
	Let $\famiI{U_i}$ and $\famiI{V_i}$ be families such that each $U_i$
	is an open nonempty set of a normed space $X_i$
	and each $V_i$ is open and a star-shaped subset with center $0$ of a normed space $Y_i$,
	$k \in \cl{\N}$, $f : \disjointUI{U_i} \to \R$ an adjusting weight for $\famiI{V_i}$
	and $\GewFunk \sub \cl{\R}^{\disjointUI{U_i}}$ with $f \in \GewFunk$.
	Then $\CcFfoPROi{U_i}{V_i}{k}{f}$ is not empty. In particular, for $\tau > 0$ we have
		\begin{equation}\label{incl:1-Kugel_f0-norm_sub_CFof}
			\set{\eta \in \CcFproI{U_i}{Y_i}{k}}{ \hn{\eta}{f}{0} < \tau}
			\sub \CcFfoPROi{U_i}{\tau \cdot V_i}{k}{f}.
		\end{equation}
\end{lem}

\subsubsection{Smooth maps between restricted products}

We provide some of the results about smooth functions between the restricted products.
Before we do this, we define what smooth should mean.
We use the differential calculus for maps between locally convex spaces
that is known as Kellers $C^k_c$-theory.
This calculus was introduced in \cite{MR0177277} resp. \cite{MR0440592},
and is widely used to study maps between infinite-dimensional locally convex spaces,
for example in the works \cite{MR830252}, \cite{MR583436}, \cite{MR1911979} or \cite{MR2261066}.
\begin{defi}
	Let $\SX$ and $\SY$ be locally convex spaces,
	$\UF \subseteq \SX$ an open nonempty set
	and $f:\UF\to\SY$ a map.
	We say that $f$ is \emph{$\ConDiff{}{}{1}$} if for all $u\in\UF$ and $x\in\SX$,
	the directional derivative	
	\[
		\lim_{\substack{t\to 0\\ t \neq 0}}\frac{f(u + t x) - f(u)}{t}
			\defn \dA{f}{u}{x},
	\]
	exists and the map $\dA{f}{}{}  : \UF \times \SX \to \SY$ is continuous.
	Inductively, for a $k\in\N$ we call $f$ \emph{$\ConDiff{}{}{k}$}
	if $f$ is $\ConDiff{}{}{1}$ and
	$\dA{f}{}{} :\UF\times\SX\to\SY$
	is a $\ConDiff{}{}{k - 1}$-map.
	We write $\ConDiff{\UF}{\SY}{k}$ for the set of $k$-times differentiable maps.
\end{defi}

\paragraph{Superposition results}

We provide some results on simultaneous superposition operations on restricted products.
They are proved in \cite[§ 4.3.2, 4.3.3]{Walter_Diss_p1}.
\begin{lem}\label{lem:simultane_mult-multiplier}
	Let $\famiI{U_i}$ be a family such that each $U_i$
	is an open nonempty set of the normed space $X_i$,
	and $\famiI{Y_i^1}$, $\famiI{Y_i^2}$, $\famiI{Z_i}$ be families of normed spaces.
	Further, for each $i \in I$ let $M_i : U_i \to Y_i^1$ be smooth,
	and $\beta_i : Y_i^1 \times Y_i^2 \to Z_i$ a bilinear map such that
	\[
		\sup\set{\Opnorm{\beta_i}}{i \in I} < \infty .
	\]
	Assume that $\GewFunk \sub \cl{\R}^{\disjointUI{U_i}}$ is nonempty and
	\begin{equation}\label{cond:est_sim-multiplier_weights}
			(\forall f \in \GewFunk, \ell \in \N)
			(\exists g \in \ExtWeights{\GewFunk})
			\, (\forall i \in I)\, \hn{M_i}{1_{U_i}}{\ell} \abs{f_i} \leq \abs{g_i}.
	\end{equation}
	Then for $k \in \cl{\N}$, the map
	\[
				\CcFproI{U_i}{Y_i^2}{k}
		\to
		\CcFproI{U_i}{Z_i}{k}
		:
		\famiI{\gamma_i} \mapsto \famiI{\beta_i \circ (M_i, \gamma_i)}
	\]
	is defined and continuous linear.
\end{lem}
\begin{prop}\label{prop:simultane_SP_BCinf0_Produkt}
	Let $\famiI{U_i}$ and $\famiI{V_i}$ be families such that each $U_i$
	is an open nonempty set of the normed space $X_i$
	and each $V_i$ is an open, star-shaped subset with center $0$ of a normed space $Y_i$.
	Further, let $\famiI{Z_i}$ be another family of normed spaces
	and $\GewFunk \sub \cl{\R}^{\disjointUI{U_i}}$ contain an adjusting weight $\omega$.
	For each $i \in I$, let $\beta_i \in \FC{ U_i \times V_i}{ Z_i}{\infty}$
	with $\beta_i(U_i \times \sset{0}) = \sset{0}$.
	Further, assume that
	\begin{equation}\label{cond:est_SP-Abb_weights}
			(\forall f \in \GewFunk, \ell \in \N^*)
			(\exists g \in \ExtWeights{\GewFunk})
			\, (\forall i \in I)\, \hn{\beta_i}{1_{U_i \times V_i}}{\ell} \abs{f_i} \leq \abs{g_i}
	\end{equation}
	is satisfied.
	Then for $k \in \cl{\N}$, the map
	\[
		\beta_* \ndef \prod_{i \in I} (\beta_i)_* : \CcFfoPROi{U_i}{V_i}{k}{\omega}
		\to
		\CcFproI{U_i}{Z_i}{k}
		:
		\famiI{\gamma_i} \mapsto \famiI{\beta_i \circ (\id{U_i}, \gamma_i)}
	\]
	is defined and smooth.
\end{prop}
The following small lemma which will be useful later.
\begin{lem}\label{lem:vergleich_Bedingungen_simultane-multiplier_simu-Supo}
	Let $\famiI{U_i}$ be a family such that each $U_i$
	is an open nonempty set of a normed space $X_i$
	and $\famiI{Y_i}$ a family of normed spaces.
	Further, for each $i \in I$ let $\beta_i : U_i \to Y_i$ be a smooth map
	and $\GewFunk \sub \cl{\R}^{\disjointUI{U_i}}$ such that \ref{cond:est_sim-multiplier_weights}
	is satisfied for $\famiI{\FAbl{\beta_i}}$.
	Then for each $R > 0$, $\famiI{\rest{\dA{\beta_i}{}{}}{U_i \times \Ball[X_i]{0}{R}}}$ satisfies \ref{cond:est_SP-Abb_weights}.
\end{lem}

\paragraph{Composition and inversion}

We finally provide results on simultaneous composition and inversion.
These results are proved in \cite[§ 4.4]{Walter_Diss_p1}.

\begin{prop}\label{prop:Simultane_Koor-Kompo_diffbar}
	Let $\famiI{U_i}$, $\famiI{V_i}$ and $\famiI{W_i}$ be families such that
	for each $i \in I$, $U_i$, $V_i$ and $W_i$ are open nonempty sets of the normed space $X_i$
	with $U_i + V_i \sub W_i$, and $V_i$ is balanced.
	Further, let $\famiI{Y_i}$ be another family of normed spaces
	and $\GewFunk \sub \cl{\R}^{\disjointUI{W_i}}$ contain an adjusting weight $\omega$ for $\famiI{V_i}$.
	Then for $k, \ell \in \cl{\N}$, the map
	\begin{equation*}
		\compIdDiffKLcW{Y}{k}{\ell}
		\ndef \prod_{i \in I} \compIdDiffKLWeights{Y_i}{k}{\ell}{\GewFunk_i}
		:\left\lbrace
		\begin{aligned}
			\CcFproI{W_i}{Y_i}{k + \ell + 1} \times \CcFfoPROi{U_i}{V_i}{k}{\omega}
			&\to \CcFproI{U_i}{Y_i}{k}
			\\
			(\famiI{\gamma_i}, \famiI{\eta_i}) &\mapsto \famiI{\gamma_i \circ (\eta_i + \id{U_i})}
		\end{aligned}\right.
	\end{equation*}
	is defined and $\ConDiff{}{}{\ell}$.
\end{prop}

\begin{prop}\label{prop:Simultane_Inv-Kompo_glatt}
	Let $\famiI{U_i}$ and $\famiI{\widetilde{U}_i}$ be families
	such that $U_i$ and $\widetilde{U}_i$ are open nonempty sets of the Banach space $X_i$
	and $U_i$ is convex.
	Further assume that there exists $r > 0$ such that $\widetilde{U}_i + \Ball[X_i]{0}{r} \sub U_i$
	for all $i \in I$.
	Let $\GewFunk \sub \cl{\R}^{\disjointUI{U_i}}$ with $1_{\disjointUI{U_i}} \in \GewFunk$
	and $\tau \in ]0, 1[$.
	Then the map
	\[
		\InvIdcW{\widetilde{U}}
		\ndef
		\prod_{i \in I}
		\InvIdWeights{\widetilde{U}_i}{\GewFunk_i}
		:
		\mathcal{D}^\tau
		\to
		\CcFproI{\widetilde{U}_i}{X_i}{\infty}
		:
		\famiI{\phi_i} \mapsto \famiI{ \rest{(\phi_i + \id{U_i})^{-1}}{\widetilde{U}_i} - \id{\widetilde{U}_i} }
	\]
	is defined and smooth, where
	\[
		\mathcal{D}^\tau \ndef \left\set{\phi \in \CcFproI{U_i}{X_i}{\infty} }{			\hn{\phi}{1_{\disjointUI{U_i}}}{1} < \tau
			\text{ and }
			\hn{\phi}{1_{\disjointUI{U_i}}}{0}
			< \tfrac{r}{2} (1 - \tau)
		\right}.
	\]
\end{prop}

\subsection{Riemannian geometry and manifolds}
\label{susec:RiemannMFK}
We introduce notation and prove some results involving Riemannian geometry.
Before we do this, we need the following variants of the inverse function theorem.

\subsubsection{Inverse function theorems}

We need a quantitative version of the famous inverse mapping theorem,
and in addition also a parameterized variant of it.
Both can be proved using parameterized versions of the Banach fixed point theorem
as provided in \cite[App.~C]{MR586942}.
In particular, these assertions can be proved with Cor.~A.2.17 or Thm.~A.2.16 in \cite{arxiv_1006.5580v3}, respectively.

\begin{satz}[Quantitative version of the inverse function theorem]
\label{satz:QuantitativeInverseFunctionTheorem_wholeSet}
	Let $X$ be a Banach space, $U \sub X$ open and convex, $k \in \cl{\N}^*$,
	$g \in \FC{U}{X}{k}$ and $x \in U$ such that $\FAbl{g}(x)$ is invertible.
	Further, let $\sup_{y \in U} \Opnorm{\FAbl{g}(y) - \FAbl{g}(x) } < \delta$
	with $\delta < \tfrac{1}{ \Opnorm{ \FAbl{g}(x)^{-1} } }$.
	Then $g$ is a homeomorphism of $U$ onto an open subset of $X$, and $g^{-1}$ is $\FC{}{}{k}$.
	Further, if $U$ contains the ball $\Ball{x'}{r}$,
	then $g(U)$ contains the ball $\Ball{g(x')}{r'}$,
	where $r' \ndef \frac{r (1 -  \delta \cdot \Opnorm{ \FAbl{g}(x)^{-1} })}{\Opnorm{\FAbl{g}(x)^{-1}}}$.
	Further, $\rest{g^{-1}}{\Ball{g(x')}{r'}}$ is
	$\frac{\Opnorm{\FAbl{g}(x)^{-1}}}{1 -  \delta \cdot \Opnorm{ \FAbl{g}(x)^{-1} }}$-Lipschitz.
\end{satz}

\begin{prop}[Parameterized quantitative version of the inverse function theorem]\label{prop:Differential_vollerRang}
	Let $X$ be a normed space, $Y$ a Banach space, $U \sub X$ open and $V \sub Y$ open and convex,
	$k \in\cl{\N}^*$, $g \in \FC{U \times V}{Y}{k}$
	and $(x_0, y_0) \in U \times V$ such that $\parFAbl{2}{g}(x_0, y_0)$ is invertible.
	Further, let
	\[
		\sup_{(x, y) \in U \times V} \Opnorm{\parFAbl{2}{g}(x, y) - \parFAbl{2}{g}(x_0, y_0) } < \delta
	\]
	with $\delta < \tfrac{1}{ \Opnorm{ \parFAbl{2}{g}(x_0, y_0)^{-1} } }$.
	Then for each $x \in U$, $g_x \ndef g(x, \cdot) : V \to Y$ is a homeomorphism onto an open subset of $Y$.
	Further,
	\[
		\Ball[Y]{y}{r} \sub  V
		\implies
		\Ball{g_x(y)}{r'} \sub g_x(\Ball[Y]{y}{r}),
	\]
	where $r' \ndef \frac{r (1 -  \delta \cdot \Opnorm{ \parFAbl{2}{g}(x_0, y_0)^{-1} })}{\Opnorm{\parFAbl{2}{g}(x_0, y_0)^{-1}}}$.
	Further, $\rest{g_x^{-1}}{\Ball{g(x, y)}{r'}}$ is
	$\frac{\Opnorm{\parFAbl{2}{g}(x_0, y_0)^{-1}}}{1 -  \delta \cdot \Opnorm{ \parFAbl{2}{g}(x_0, y_0)^{-1} }}$-Lipschitz,
	and the map
	\[
		\bigcup_{x \in U} \sset{x} \times \Ball{g(x, y)}{r'} \to \Ball[Y]{y}{r}
		: (x, z) \mapsto g_x^{-1}(z)
	\]
	is $\FC{}{}{k}$.
\end{prop}

\subsubsection{Definitions and elementary results}
We need the following, well-known facts about Riemannian geometry:
\begin{defi}[Riemannian exponential function]
	Let $d \in \N^*$ and $(M, g)$ be a $d$-dimensional Riemannian manifold.
	Then the (maximal) domain $\domExpMax{g}$ of the \emph{Riemannian exponential map} $\Rexp{g} : \domExpMax{g} \to M$
	is an open subset of $\Tang{M}$.
	$\domExpMax{g}$ is an open neighborhood of the zero section in $\Tang{M}$,
	and for each $x \in M$, we have $[0,1] . (\domExpMax{g} \cap \Tang[x]{M}) \sub \domExpMax{g} \cap \Tang[x]{M}$.

	For each $x \in M$, we define $\RexpPar{x}$ as $\rest{\Rexp{g}}{\Tang[x]{M} \cap \domExpMax{g}}$.
	If $M$ is an open subset of $\R^d$, then for each $x \in M$ and $v, w \in \R^d$,
	we have the identity
	\begin{equation}\label{id:Ableitung_Rexp_Nullschitt}
		\dA{\Rexp{g}}{x, 0}{v, w} = v + w .
	\end{equation}
\end{defi}
In order to define the logarithm, we need the following definition.
\newcommand{\BallTang}[3]{\ensuremath{B^{#1}_{#3}(#2)}}
\begin{defi}
	Let $d \in \N^*$ and $(M, g)$ a Riemannian manifold. For $x$ and $h \in \Tang[x]{M}$,
	we define
	\[
		\norm{h}_{g_x} \ndef \sqrt{g(h, h)}
		.
	\]
	\newcommand{\Tangi}[2]{\Tang[#1]{#2}}%
	Obviously, each $\norm{\cdot}_{g_x}$ is a norm on $\Tang[x]{M}$. We also define
	\[
		\BallTang{g_x}{0}{r} \ndef \Ball[(\Tangi{x}{M}, \norm{\cdot}_{g_x})]{0}{r} .
	\]
	If $M$ is an open subset of $\R^d$, we set for $h \in \R^d$
	\[
		\norm{h}_{g_x} \ndef \sqrt{g((x,h), (x,h))}
		.
	\]
	Obviously, each $\norm{\cdot}_{g_x}$ is a norm on $\R^d$.
	In particular, we define
	\[
		\BallTang{g_x}{0}{r} \ndef \Ball[(\R^d, \norm{\cdot}_{g_x})]{0}{r} .
	\]
\end{defi}
\begin{defi}[Riemannian logarithm map]
	Let $d \in \N^*$ and $(M, g)$ be a $d$-dimensional Riemannian manifold.
	For all $x \in M$ there exists an open neighborhood $V_x \sub \Tang[x]{M}$ of $0_x$
	such that $\rest{\RexpPar{x}}{V_x}$ is a diffeomorphism onto its image, which is an open subset of $M$.
	So the \emph{Riemannian logarithm map} $\Rlog{g} : \domLogMax{g} \to \Tang{M}$
	can be defined, where
	\[
		\domLogMax{g}
		\ndef
		\bigcup_{x \in M} \sset{x} \times \RexpPar{x}(\BallTang{g_x}{0}{r_x})
		\sub M^2
	\]
	and $\Rlog{g}(x, y) \ndef \rest{\RexpPar{x}}{ \BallTang{g_x}{0}{r_x} }^{-1}$;
	here $r_x \ndef \sup \set{r > 0}{ \rest{\RexpPar{x}}{\BallTang{g_x}{0}{r}} \text{ is injective} }$.
	Further, for $x \in M$ we set $\RlogPar{x}(y) \ndef \Rlog{g}(x, y)$
	for $y \in M$ such that $(x, y) \in \domLogMax{g}$.

	Let $M$ be an open subset of $\R^d$. We define $\RlogVR{g} \ndef \pi_2 \circ \Rlog{g}$,
	where $\pi_2 : \R^{2 d} \to \R^d$ denotes the projection on the second factor.
	For each $x \in M$ and $v, w \in \R^d$,
	the identity $\dA{\RlogVR{g}}{x, x}{v, w} = w - v$ holds.
	This is an immediate consequence of \ref{id:Ableitung_Rexp_Nullschitt} and the chain rule
	since $\Rlog{g}$ and $(\pi_1, \Rexp{g})$ are inverse functions.
\end{defi}

We introduce notation for the set of vector fields.
For a manifold $M$ and $k \in \cl{\N}$,
We let $\VecFields[k]{M}$ denote the set of $\ConDiff{}{}{k}$ vector fields.
In particular, we set $\VecFields{M} \ndef \VecFields[\infty]{M}$.

\begin{defi}[Localizations]
	Let $M$ be a $d$-dimensional manifold, $\kappa : \UF \to \VF$ a chart of $M$,
	$k \in \cl{\N}$ and $X \in \VecFields[k]{\UF}$.
	Then we set $\VectFLok{X}{\kappa} \ndef \dA{\kappa}{}{}\circ X \circ\kappa^{-1}: \VF\to\R^d$.
	If $g$ is a metric on $M$, we set $\RMetLok{g}{\kappa}\ndef g \circ(\Tang{\kappa^{-1}}\oplus\Tang{\kappa^{-1}})$.
\end{defi}

\begin{bem}\label{bem:Riemann-Geometrie_Karte}
	Let $(M, g)$ be a Riemannian manifold and $\kappa : \UF \to \VF$ a chart of $M$.
	Then $\Tang{\kappa^{-1}}(\domExpMax{\RMetLok{g}{\kappa}}) \sub \domExpMax{g}$,
	and
	$\Rexp{\RMetLok{g}{\kappa}}
	= \rest{\kappa \circ \Rexp{g} \circ\, \Tang{\kappa^{-1}} }{\domExpMax{ \RMetLok{g}{\kappa}}}$.
	Further, $(\kappa^{-1} \times \kappa^{-1})(\domLogMax{\RMetLok{g}{\kappa}}) \sub \domLogMax{g}$,
	and
	$\Rlog{\RMetLok{g}{\kappa}}
	= \rest{\Tang{\kappa} \circ \Rlog{g} \circ (\kappa^{-1} \times \kappa^{-1}) }{\domLogMax{ \RMetLok{g}{\kappa}}}$.
\end{bem}

\begin{lem}\label{lem:Log-exp_unter_Kartenwechsel}\label{lem:SP_mitRexp-Vergleich_lokal_global}
	Let $(M, g)$ be a Riemannian manifold and $\kappa : \widetilde{\UF_\kappa} \to \UF_\kappa$,
	$\phi : \widetilde{\UF_\phi} \to \UF_\phi$ charts for $M$ such that $\widetilde{\UF_\kappa} \cap \widetilde{\UF_\phi} \neq \emptyset$.
	Then the following identities hold:
	\begin{assertions}
		\item
		On $\domExpMax{\RMetLok{g}{\rest{\kappa}{\kappa(\widetilde{\UF_\kappa} \cap \widetilde{\UF_\phi})} }}$, we have
		$\phi \circ \kappa^{-1} \circ \Rexp{\RMetLok{g}{\kappa}}
		= \Rexp{\RMetLok{g}{\phi}} \circ \Tang{(\phi \circ \kappa^{-1})}$.

		\item
		On $\domLogMax{ \RMetLok{g}{ \rest{\kappa}{\kappa(\widetilde{\UF_\kappa} \cap \widetilde{\UF_\phi})} }}$, we have
		$\Tang{(\phi \circ \kappa^{-1})} \circ \Rlog{\RMetLok{g}{\kappa}}
		= \Rlog{\RMetLok{g}{\phi}} \circ (\phi \circ \kappa^{-1} \times \phi \circ \kappa^{-1})$.
	\end{assertions}
	Let $X \in \VecFields[0]{M}$.
	\begin{assertions}[resume]
		\item
		$\Tang{(\phi \circ \kappa^{-1})} \circ (\id{\kappa(\widetilde{\UF_\kappa} \cap \widetilde{\UF_\phi})}, \VectFLok{X}{\kappa})
		= (\id{\phi(\widetilde{\UF_\kappa} \cap \widetilde{\UF_\phi})}, \VectFLok{X}{\phi}) \circ \phi \circ \kappa^{-1}$.
	\end{assertions}
	Additionally, let $\VF \sub \UF_\kappa$ such that $\image{(\Tang{\kappa} \circ \rest{X}{\VF}) } \sub \domExpMax{\RMetLok{g}{\kappa}}$.
	\begin{assertions}[resume]
		\item
		Then $\kappa \circ \Rexp{g} \circ X \circ \kappa^{-1} = \Rexp{\RMetLok{g}{\kappa}} \circ (\id{\kappa(\VF)}, \VectFLok{X}{\kappa})$ on $\kappa(\VF)$.
	\end{assertions}
\end{lem}
\begin{proof}
	These are easy computations involving \refer{bem:Riemann-Geometrie_Karte}.
\end{proof}

\subsubsection{Riemannian exponential function and logarithm on open subsets of \texorpdfstring{$\boldsymbol{\R^d}$}{R^d}}
\label{susec:RiemannExpLog_VR}

We examine functions that arise as the composition of
the second component $\RlogVR{g}$ of the Riemannian logarithm
or the exponential map $\Rexp{g}$ with functions of the form $(\id{}, X)$. 
Of particular interest are estimates for the function values and the values of the first derivatives
of such functions.

We also derive a sufficient criterion on a vector field $X$
that ensures that $\Rexp{g} \circ X$ is injective,
and gives a lower bound for the size of its image.
Further, we use that $\RlogVR{g}$ is the fiberwise inverse function to $\Rexp{g}$,
and apply the parameterized inverse function theorem \refer{prop:Differential_vollerRang}.
We will get estimates for the domain and the first partial derivative of $\RlogVR{g}$
in terms of those numbers for $\Rexp{g}$.

For open nonempty $U \sub \R^d$, we will tacitly identify $\Tang{U}$ with $U \times \R^d$
and, for $x \in U$, $\Tang[x]{U}$ with $\R^d$.

\paragraph{Superposition with the Riemannian exponential map}

We start with the exponential map.
\subparagraph{Estimates for function values and the first derivatives}
We derive estimates for the function values and first derivatives of $\Rexp{g} \circ(\id{}, X)$.
This is mostly done using the mean value theorem and the triangle inequality.
\begin{defi}
	Let $d \in \N^*$, $U \sub \R^d$ an open nonempty subset, $g$ a Riemannian metric on $U$
	and $K \sub U$ a relatively compact set.
	Using standard compactness arguments, we see that there exists $\tau > 0$ such that
	$\cl{K} \times \Ball{0}{ \tau } \sub \domExpMax{g}$ (note that this implies $\Rexp{g}( \cl{K} \times \Ball{0}{ \tau } ) \sub U$).
	We denote the supremum of such $\tau$ by $\grenzExp[g]{K}{U}$.
	If the metric discussed is obvious, we may omit it and write $\grenzExp{K}{U}$.
	Now let $0 < \delta < \grenzExp{K}{U}$.
	We define
	\[
		\BndFstAblRex[g]{K}{\delta}
		\ndef
		\sup_{x \in \cl{K}} \hn{\RexpPar{x}}{1_{ \clBall{0}{\delta}} }{1}
		= \hn{\parFAbl{2}{\Rexp{g}}}{ 1_{ \cl{K} \times \clBall{0}{\delta}} }{0}
	\]
	and
	\[
		\BndSndAblRex[g]{K}{\delta}
		\ndef
		\hn{\Rexp{g}}{1_{\cl{K} \times \clBall{0}{\delta}}}{2}.
	\]
	As above, if the metric discussed is clear, we may omit it and just write
	$\BndFstAblRex{K}{\delta}$ or $\BndSndAblRex{K}{\delta}$, respectively.
	Note that $\hn{\Rexp{g}}{1_{K \times \clBall{0}{\delta}}}{2}$
	relates to the norm $\norm{(v, w)} = \max(\norm{v}, \norm{w})$ on $\R^{2 d}$.
\end{defi}

\begin{lem}\label{lem:Naehe_expVF-id_C0}
	Let $d \in \N^*$, $U \sub \R^d$ an open nonempty subset, $g$ a Riemannian metric on $U$,
	$K \sub U$ a relatively compact set and $0 < \delta < \grenzExp{K}{U}$.
	\begin{assertions}
		\item\label{ass:Naehe_expx-id_C0}
		Then for all $x \in K$ and $y \in \clBall{0}{\delta}$ the following estimate holds:
		\[
			\norm{\Rexp{g}(x, y) - x}
			\leq \BndFstAblRex{K}{\delta} \norm{y}.
		\]

		\item\label{ass:Naehe_expVF-id_C0}
		Let $X : K \to \R^d$ with $\hn{X}{1_K}{0} \leq \delta$.
		Then for all $x \in K$, the following estimate holds:
		\[
			\norm{(\Rexp{g} \circ (\id{K}, X))(x) - \id{K}(x)}
			\leq \BndFstAblRex{K}{\delta}\norm{X(x)}
		\]
	\end{assertions}
\end{lem}
\begin{proof}
	\ref{ass:Naehe_expx-id_C0}
	We calculate using the mean value theorem
	\[
		\Rexp{g}(x, y) - x
		= \Rexp{g}(x, y) - \Rexp{g}(x, 0)
		= \Mint{ \FAbl{\RexpPar{x}}(t y) \eval y }{t}.
	\]
	From this and the definition of $\BndFstAblRex{K}{\delta}$, we easily derive the assertion.

	\ref{ass:Naehe_expVF-id_C0}
	This is an easy consequence of \ref{ass:Naehe_expx-id_C0}.
\end{proof}

\begin{lem}\label{lem:Naehe_expFunc-id_C1}
	Let $d \in \N^*$, $U \sub \R^d$ an open nonempty subset and
	$g$ a Riemannian metric on $U$.
	Further, let $W \sub U$ be an open, nonempty, relatively compact subset and $\delta \in ]0, \grenzExp{W}{U}[$.
	Then for each $X \in \ConDiff{W}{\R^d}{1}$ with $\hn{X}{1_W}{0} \leq \delta$, we have
	\begin{multline*}
		\norm{\FAbl{((\Rexp{g} \circ (\id{W}, X)) - \id{W})}(x)\eval v}
		\\
		\leq \hn{\Rexp{g}}{1_{W \times \clBall{0}{\delta} }}{2} \norm{(0, X(x))}\, \norm{(v, \FAbl{X}(x)\eval v)} + \norm{\FAbl{X}(x)\eval v}
	\end{multline*}
	for $x \in W$ and $v \in \R^d$. In particular, if we endow $\R^{2 d}$ with
	the norm $\norm{(v, w)} = \max(\norm{v}, \norm{w})$ and assume that $\hn{X}{1_W}{1} \leq 1$,
	we get the estimate
	\[
		\Opnorm{\FAbl{((\Rexp{g} \circ (\id{W}, X)) - \id{W})}(x)}
		\leq \BndSndAblRex{W}{\delta}\norm{X(x)} + \Opnorm{\FAbl{X}(x)}.
	\]
\end{lem}
\begin{proof}
	Let $x \in W$ and $v \in \R^d$. Then we calculate using that $ v = \FAbl{\Rexp{g}}(x, 0)\eval (v, 0) = \FAbl{\Rexp{g}}(x, 0)\eval (0, v)$
	and $0 = \FAbl{\Rexp{g}}(x, 0)\eval (\FAbl{X}(x)\eval v, -\FAbl{X}(x)\eval v)$
	\begin{align*}
		&\FAbl{(\Rexp{g} \circ (\id{W}, X) - \id{W}) }(x)\eval v\\
		=& \FAbl{\Rexp{g}}(x, X(x))\eval (v, \FAbl{X}(x)\eval v) - \FAbl{\Rexp{g}}(x, 0)\eval (v, 0)\\
		&\qquad +  \FAbl{\Rexp{g}}(x, 0)\eval (\FAbl{X}(x)\eval v, -\FAbl{X}(x)\eval v)\\
		=& \FAbl{\Rexp{g}}(x, X(x))\eval (v, \FAbl{X}(x)\eval v) - \FAbl{\Rexp{g}}(x, 0)\eval (v, \FAbl{X}(x)\eval v)
		 +  \FAbl{X}(x)\eval v.
	\end{align*}
	For the difference we derive using the mean value theorem
	\begin{multline*}
		( \FAbl{\Rexp{g}}(x, X(x)) - \FAbl{\Rexp{g}}(x, 0) )\eval (v, \FAbl{X}(x)\eval v)\\
		= \Mint{\FAbl{(\FAbl{\Rexp{g}) }(x, t X(x))} \eval  (0, X(x)) }{t} \eval (v, \FAbl{X}(x)\eval v).
	\end{multline*}
	From this, the assertion follows.
\end{proof}

\subparagraph{On invertibility and the size of the image}

Having established the estimates, we can give a criterion on when $\Rexp{g} \circ\, (\id{}, X)$ is injective,
and how large its image is.
The main tool used is a quantitative, parameterized version of the inverse function theorem
that is provided in \refer{satz:QuantitativeInverseFunctionTheorem_wholeSet}.
\begin{lem}\label{lem:Exp-VF_Diffeo--lokal}
	Let $d \in \N^*$, $U \sub \R^d$ open, $g$ a Riemannian metric on $U$,
	$r > 0$ such that $\clBall{0}{r} \sub U$ and $k \in \cl{\N}$ with $k \geq 1$.
	Further, let $\eps \in ]0, \frac{1}{2}[$, $\nu \in ]0, \grenzExp{\Ball{0}{r}}{U}[$ and $\delta > 0$
	with
	$
		\delta < \min\bigl(\frac{\eps r}{2 \BndFstAblRex{ \Ball{0}{r} }{\nu}}, \nu,
	\frac{\eps}{4 (\BndSndAblRex{\Ball{0}{r}}{\nu}  + 1) }\bigr)
	$.
	Then for $X\in \ConDiff{\Ball{0}{r} }{\R^d}{k}$ such that $\hn{X}{ 1_{\Ball{0}{r}} }{0} < \delta$
	and $\hn{X}{ 1_{\Ball{0}{r}} }{1} < \frac{\eps}{4}$,
	the map $\Rexp{g} \circ (\id{\Ball{0}{r}},  X)$ is a $\ConDiff{}{}{k}$-diffeomorphism onto its image,
	which is an open subset of $\R^d$ and contains $\Ball{0}{r(1 - 2 \eps)}$.
\end{lem}
\begin{proof}
	By \refer{lem:Naehe_expVF-id_C0}, for a function $X$ with
	$\hn{X}{1_{\Ball{0}{r}}}{0} <\min\bigl(\nu, \frac{\eps r}{2 \BndFstAblRex{ \Ball{0}{r} }{\nu}} \bigr)$, we have
	\[
		\tag{$\dagger$}\label{est:abschaetzung_expVF_Stelle0}
		\norm{\Rexp{g}(0, X(0) )} < \frac{\eps r}{2}.
	\]
	We set $W \ndef \Ball{0}{r}$.
	Since $\hn{X}{ 1_{\Ball{0}{r}} }{0} < \frac{\eps}{4 (\BndSndAblRex{\Ball{0}{r}}{\nu}  + 1) }$
	and  $\hn{X}{ 1_{\Ball{0}{r}} }{1} < \frac{\eps}{4} < 1$,
	we see with \refer{lem:Naehe_expFunc-id_C1} that
	$\hn{\Rexp{g} \circ(\id{W},  X) - \id{W}}{1_{\Ball{0}{r}} }{1} < \frac{\eps}{2}$.
	This implies that
	\[
		\Opnorm{\FAbl{(\Rexp{g} \circ (\id{W},  X))}(y) - \FAbl{(\Rexp{g} \circ (\id{W},  X))}(x)} < \eps
	\]
	for all $x, y \in \Ball{0}{r}$,
	and that $\FAbl{(\Rexp{g} \circ (\id{W},  X))}(0)$ is invertible with
	\[
		\tag{$\dagger\dagger$}\label{est:Opnorm_inverses_Differential_an0}
		\Opnorm{\FAbl{(\Rexp{g} \circ (\id{W},  X))}(0)^{-1}} < \frac{1}{1-\frac{\eps}{2}}.
	\]
	Since $\eps < \frac{2}{3}$, we conclude that
	$\eps < 1-\frac{\eps}{2} < \frac{1}{\Opnorm{\FAbl{(\Rexp{g} \circ (\id{W},  X))}(0)^{-1}}}$.
	Hence we can apply \refer{satz:QuantitativeInverseFunctionTheorem_wholeSet}
	to see that $\Rexp{g} \circ (\id{W},  X)$ is a diffeomorphism onto its image
	and that the image contains $\Ball{ \Rexp{g}(0,  X(0))}{r'}$, where
	$r' = r \left(\frac{1}{\Opnorm{\FAbl{(\Rexp{g} \circ (\id{W},  X))}(0)^{-1}}} - \eps\right)$.
	From this we deduce using \ref{est:abschaetzung_expVF_Stelle0}, \ref{est:Opnorm_inverses_Differential_an0}
	and the triangle inequality (where we need $\eps < \frac{1}{2}$)
	that the image of $\Rexp{g} \circ (\id{W},  X)$ contains $\Ball{0}{r (1 - 2\eps)}$.
\end{proof}

\paragraph{Superposition with the Riemannian logarithm}

We examine $\RlogVR{g}$.
In particular, we use that $\RlogVR{g}$ is the fiberwise inverse function to $\Rexp{g}$.
We show that its domain $\domLogMax{g}$ is a neighborhood of the diagonal,
and that we can quantify what is contained in it;
and we give estimates for its first derivative.
\\
Further, we examine maps that arise as the composition of $\RlogVR{g}$
with maps of the form $(\id{}, X + \id{})$.
Of particular interest are estimates for the function values and the derivatives
of these maps.
\subparagraph{Uniform estimates for Riemannian norms}
We start by establishing estimates for the Riemannian norms
and a given norm on $\R^d$.
\begin{lem}\label{lem:Vergleich_Normen_Riemann_Kompakta}
	Let $d \in \N^*$, $U \sub \R^d$ open, $g$ a Riemannian metric on $U$.
	Then for each $x \in U$, there exist $V \in \neighO{x}$ and $c, C > 0$ such that
	\[
		c \norm{\cdot} \leq \norm{\cdot}_{g_y} \leq C \norm{\cdot}
	\]
	for all $y \in V$.
\end{lem}
\begin{proof}
	In the proof, for $x \in U$ we let $G_x$ denote the matrix $(g((x, e_i), (x, e_j)))_{1 \leq i, j \leq d}$.

	There exists $\widetilde{C} > 0$ such that $\norm{\cdot}_{g_x} \leq \widetilde{C} \norm{\cdot}$.
	Further, for $\eps > 0$ there exists $V \in \neighO{x}$ such that
	\[
		\Opnorm{G_y - G_x} < \eps
	\]
	for all $y \in V$. Hence for $y \in V$ and $h \in \R^d$,
	\[
		\norm{h}_{g_y}^2 = \SkaPrd{h}{G_y \eval h}  - \SkaPrd{h}{G_x \eval h} + \SkaPrd{h}{G_x \eval h}
		= \SkaPrd{h}{(G_y - G_x) \eval h} + \norm{h}_{g_x}^2
		\leq \eps \norm{h}^2 + \widetilde{C}^2 \norm{h}^2
		.
	\]
	From this, we easily deduce the first estimate.

	For the second estimate, we have for $y \in U$ and $h \in \R^d$ that
	\[
		\norm{h}_{g_y}^2 = \SkaPrd{h}{G_y \eval h} = \SkaPrd{A_y\eval h}{A_y \eval h}
		= \norm{A_y \eval h}^2_2
		\geq \tfrac{1}{\Opnorm{A_y^{-1}}^2 } \norm{h}^2_2
		;
	\]
	where $A_y = \sqrt{G_y}$. Since the map $y \mapsto \Opnorm{\sqrt{G_y}^{-1}}$ is continuous,
	and there exists $\widetilde{c} > 0$ such that
	$\norm{\cdot}_2 \geq \widetilde{c} \norm{\cdot}$,
	we see that the assertion holds.
\end{proof}

\begin{defi}
	Let $d \in \N^*$, $U \sub \R^d$ an open nonempty subset, $g$ a Riemannian metric on $U$
	and $K \sub U$ a relatively compact set. We define
	\[
		\QuotNorm{K}
		\ndef \frac{\sup \set{c > 0}{(\forall x \in K)\, c \norm{\cdot} \leq \norm{\cdot}_{g_x} }}%
		{\inf \set{C > 0}{(\forall x \in K)\, \norm{\cdot}_{g_x} \leq C \norm{\cdot} }}
		.
	\]
	Note that because of \refer{lem:Vergleich_Normen_Riemann_Kompakta}, $\QuotNorm{K} \in ]0, 1]$.
\end{defi}

\subparagraph{Applying the parametrized inverse function theorem}
We use \refer{prop:Differential_vollerRang} to derive estimates for the domain and the first derivatives
of $\RlogVR{g}$, under a certain condition on the partial differentials of $\Rexp{g}$.
Further, we show that $\domLogMax{g}$ is a neighborhood of the diagonal.

\begin{defi}
	Let $d \in \N^*$, $U \sub \R^d$ open, $(U, g)$ a Riemannian manifold,
	$V$ a nonempty relatively compact set with $\cl{V} \sub U$ and $\sigma \in ]0,1[$.
	We define
	\[
		\RadExpFibInv[g]{V}{\sigma}
		\ndef
		\sup\set{r \in ]0, \grenzExp{V}{U}[}{(\forall x \in \cl{V})\,\hn{\RexpPar{x} - \id{\R^d}}{1_{\clBall{0}{r}}}{1} < \sigma}.
	\]
	If the metric can not be confused, we may omit it in the notation and just write $\RadExpFibInv{V}{\sigma}$.
	Note that $\RadExpFibInv{V}{\sigma} > 0$ as one can prove using compactness arguments
	and \ref{id:Ableitung_Rexp_Nullschitt}.
\end{defi}

\begin{lem}\label{lem:estimates_constants_log_from_exp}\label{lem:Est_Dom_Rlog}
	Let $d \in \N^*$, $U \sub \R^d$ open, $g$ a Riemannian metric on $U$,
	$V \sub U$ an open, nonempty, relatively compact set, $\sigma \in ]0, 1[$ and
	$\tau \in ]0, \RadExpFibInv[g]{V}{\sigma}[$.
	Then the following assertions hold:
	\begin{assertions}
		\item
		$\BndFstAblRex{V}{\tau} \leq 1 + \sigma$.
	\end{assertions}
	Let $x \in V$. Then
	\begin{assertions}[resume]
		\item\label{ass1:Rexp_Diffeo}
		$\rest{\RexpPar{x}}{\Ball{0}{\tau} }$ is a diffeomorphism onto its image,

		\item\label{ass1:Kugel_in_img_Rexp}
		$\Ball{x}{(1 - \sigma) r} \sub \RexpPar{x}(\Ball{0}{r})$ for all $r \in [0,\tau]$, and

		\item\label{ass1:est_Lip_Rlog}
		$\rest{(\rest{\RexpPar{x}}{\Ball{0}{\tau}}^{-1})}{\Ball{x}{(1 - \sigma)\tau}}$ is $\tfrac{1}{1 - \sigma}$-Lipschitz.
	\end{assertions}
	Finally, assume that $\tau <  \QuotNorm{V} \RadExpFibInv[g]{V}{\sigma}$. Then
	\begin{assertions}[resume]
		\item\label{ass1:contaiment_dom_log}
		$\bigcup_{x \in V} \sset{x} \times \RexpPar{x}(\Ball{0}{\tau}) \sub \domLogMax{g}$.
	\end{assertions}
\end{lem}
\begin{proof}
	The assertion about $\BndFstAblRex{V}{\tau}$ follows from a simple application of the triangle inequality
	to $\Opnorm{\parFAbl{2}{\Rexp{g}} \pm \id{\R^d} }$.
	To prove \ref{ass1:Rexp_Diffeo}-\ref{ass1:est_Lip_Rlog}, let $x, y \in V$. Then for each $z \in \Ball{0}{\tau}$, we have
	\[
		\Opnorm{\parFAbl{2}{\Rexp{g}}(x, z) -  \parFAbl{2}{\Rexp{g}}(x, 0) }
		= \Opnorm{\parFAbl{2}{\Rexp{g}}(x, z) -  \id{\R^d} }
		< \sigma .
	\]
	Further, $\sigma < 1 = \frac{1}{\Opnorm{(\parFAbl{2}{\Rexp{g}}(x, 0))^{-1}}}$.
	So we can apply \refer{prop:Differential_vollerRang} to derive
	the assertions about $\rest{\RexpPar{x}}{\Ball{0}{\tau}}$ and
	$\rest{(\rest{\RexpPar{x}}{\Ball{0}{\tau}}^{-1})}{\Ball{x}{(1 - \sigma)\tau}}$.

	\ref{ass1:contaiment_dom_log}
	We see using \refer{lem:Vergleich_Normen_Riemann_Kompakta}
	and standard compactness arguments
	that for each $x \in V$,
	\[
		\Ball{0}{\tau} \sub \BallTang{g_x}{0}{\tau C} \sub \Ball{0}{\frac{\tau}{\QuotNorm{V}}}
		;
	\]
	here $C$ denotes the denominator in the definition of $\QuotNorm{V}$.
	Since $\frac{\tau}{\QuotNorm{V}} < \RadExpFibInv[g]{V}{\sigma}$ by our assumption,
	we see with \ref{ass1:Rexp_Diffeo} that each map $\rest{\RexpPar{x}}{ \BallTang{g_x}{0}{\tau C} }$
	is injective, and can conclude that $\sset{x}\times\RexpPar{x}(\Ball{0}{\tau}) \sub \domLogMax{g}$.
\end{proof}

\subparagraph{Estimates for function values and first derivatives}
Before we establish the estimates, we make the following definitions.
\newcommand{\domLogProd}[2]{\ensuremath{#1^{\times #2}}}
\newcommand{\domLogProdCL}[2]{\ensuremath{#1^{\cl{\times} #2}}}
\begin{defi}
	Let $X$ be a normed space, $S \sub X$ and $\tau > 0$. We set
	\[
		\domLogProd{S}{\tau} \ndef \bigcup_{x \in S} \sset{x} \times \Ball{x}{\tau}
	\qquad
	\text{and}
	\qquad
		\domLogProdCL{S}{\tau} \ndef \bigcup_{x \in \cl{S}} \sset{x} \times \clBall{x}{\tau}.
	\]
\end{defi}

\begin{defi}
	Let $d \in \N^*$, $U \sub \R^d$ an open nonempty subset, $g$ a Riemannian metric on $U$
	and $K \sub U$ a relatively compact set.
	By \refer{lem:Est_Dom_Rlog} (more precisely, \ref{ass1:Kugel_in_img_Rexp} and \ref{ass1:contaiment_dom_log}),
	there exists $\tau > 0$ such that $\domLogProd{\cl{K}}{\tau} \sub \domLogMax{g}$.
	We denote the supremum of such $\tau$ by $\grenzLog[g]{K}{U}$.
	If the metric discussed is obvious, we may omit it and just write $\grenzLog{K}{U}$.

	Now let $0 < \delta < \grenzLog{K}{U}$.
	We define
	\[
		\BndFstAblRlog[g]{K}{\delta}
		\ndef
		\sup_{x \in \cl{K}} \hn{\pi_2 \circ \RlogPar{x}}{1_{ \clBall{x}{\delta}} }{1}
		= \hn{\parFAbl{2}{ \RlogVR{g} }}{ 1_{ \domLogProdCL{\cl{K}}{\delta} }}{0}
	\]
	and
	\[
		\BndSndAblRlog[g]{K}{\delta}
		\ndef
		\hn{ \RlogVR{g} }{1_{\domLogProdCL{\cl{K}}{\delta}}}{2}.
	\]
	As above, if the metric discussed is clear, we may omit it in the notation and just write
	$\BndFstAblRlog{K}{\delta}$ or $\BndSndAblRlog{K}{\delta}$, respectively.
	Note that $\hn{\cdot}{1_{\domLogProdCL{\cl{K}}{\delta}}}{2}$
	relates to the norm $\norm{(v, w)} = \max(\norm{v}, \norm{w})$ on $\R^{2 d}$.
\end{defi}
We rephrase some results of \refer{lem:estimates_constants_log_from_exp}.
\begin{lem}\label{lem:Estimates_Rex_Rlog_DomRlog}
	Let $d \in \N^*$, $U \sub \R^d$ open, $g$ a Riemannian metric on $U$,
	$W \sub U$ an open, nonempty, relatively compact set, $\sigma \in ]0, 1[$ and
	$\tau \in ]0, \RadExpFibInv[g]{W}{\sigma} \QuotNorm{W}[$.
	Then $\BndFstAblRex{W}{\tau} \leq 1 + \sigma$, $(1 - \sigma) \tau < \grenzLog{W}{U}$
	and
	$\BndFstAblRlog{W}{(1 - \sigma) \tau} \leq \tfrac{1}{1 - \sigma}$.
\end{lem}
\begin{proof}
	The assertions follow from \refer{lem:estimates_constants_log_from_exp}.
\end{proof}
Now we prove the estimates.
\begin{lem}\label{lem:Norm_LogFunc_C0}
	Let $d \in \N^*$, $U \sub \R^d$ an open nonempty subset, $g$ a Riemannian metric on $U$,
	$K \sub U$ a relatively compact set and $\delta \in ]0, \grenzLog{K}{U}[$.
	\begin{assertions}
		\item\label{ass:Norm_LogxFunc_C0}
		Then for all $x \in K$ and $y \in \clBall{x}{\delta}$ the following estimate holds:
		\[
			\norm{\RlogVR{g}(x, y + x)}
			\leq \BndFstAblRlog{K}{\delta} \norm{y}.
		\]

		\item\label{ass:Norm_LogFunc_C0}
		Let $X : K \to \R^d$ with $\hn{X}{1_K}{0} \leq \delta$.
		Then for all $x \in K$, the following estimate holds:
		\[
			\norm{(\RlogVR{g} \circ (\id{K}, X + \id{K}))(x)}
			\leq \BndFstAblRlog{K}{\delta}\norm{X(x)}
			.
		\]
	\end{assertions}
\end{lem}
\begin{proof}
	\ref{ass:Norm_LogxFunc_C0}
	We calculate using the mean value theorem and $\RlogVR{g}(x, x) = 0$:
	\[
		\RlogVR{g}(x, y + x)
		= \RlogVR{g}(x, y + x) - \RlogVR{g}(x, x)
		= \Mint{ \parFAbl{2}{\RlogVR{g}} (x, x + t y) \eval y }{t}.
	\]
	From this and the definition of $\BndFstAblRlog{K}{\delta}$, we easily derive the assertion.

	\ref{ass:Norm_LogFunc_C0}
	This is an easy consequence of \ref{ass:Norm_LogxFunc_C0}.
\end{proof}

\begin{lem}
	Let $d \in \N^*$, $U \sub \R^d$ open, $g$ a Riemannian metric on $U$,
	$W \sub U$ an open, nonempty, relatively compact set,
	$\tau \in ]0, \grenzLog{W}{U}[$
	and $X \in \ConDiff{W}{ \Ball{0}{\tau}}{1}$.
	Then for $x \in W$
	\begin{equation}\label{est:1-norm_SP_Rlog}
		\Opnorm{\FAbl{ (\RlogVR{g} \circ (\id{W}, X + \id{W})) (x) }}
		\\
		\leq \BndSndAblRlog[g]{W}{\tau} \norm{X(x)}  + \BndFstAblRlog[g]{W}{\tau}\Opnorm{\FAbl{X}(x)}
	\end{equation}
\end{lem}
\begin{proof}
	We get with the Chain Rule that
	\begin{multline*}
		\FAbl{ (\RlogVR{g} \circ (\id{W}, X + \id{W})) (x) }
		= \FAbl{ \RlogVR{g} \circ (\id{W}, X + \id{W}) (x) }
		\MaMu (\idco, \FAbl{X}(x) + \idco)
		\\
		= \FAbl{ \RlogVR{g} \circ (\id{W}, X + \id{W}) (x) } \MaMu (\idco, \idco)
		+ \parFAbl{2}{\RlogVR{g} \circ (\id{W}, X + \id{W}) (x)} \MaMu \FAbl{X}(x).
	\end{multline*}
	We get the desired estimate for the second summand, and now treat the first.
	To this end, let $v \in \R^d$. Then we get, using that $\FAbl{\RlogVR{g}}(x, x)\eval (v, v) = v - v = 0$:
	\begin{multline*}
		\FAbl{ \RlogVR{g} \circ (x, X(x) + x)  }\eval (v,v)
		- \FAbl{\RlogVR{g}}(x, x)\eval (v, v)
		\\
		= \Mint{\FAbl{(\FAbl{\RlogVR{g}})}(x, x + t X(x))\eval(0, X(x)) }{t}\eval(v,v)
		.
	\end{multline*}
	From this, we also get the desired estimate.
\end{proof}

\section{Spaces of weighted vector fields on manifolds}

We define spaces $\CcFM{M}{k}{\atlasA}$ of weighted vector fields on manifolds,
where $\atlasA$ is an atlas for $M$.
We do this in such a way that the map
$\CcFM{M}{k}{\atlasA} \to \CFpro{U_\kappa}{\R^d}{\WeightsAtl{\GewFunk}{\atlasA}}{k}{\kappa}{\atlasA}$
that sends a vector field to the family of its localizations is an embedding.
Of particular concern is when $\CcFM{M}{k}{\atlasA} = \CcFM{M}{k}{\atlasB}$
for another atlas $\atlasB$. We derive a criterion on $\GewFunk$ ensuring this.

Further, we will discuss the simultaneous composition and inversion of weighted functions
that arise as simultaneous superposition with the localized exponential maps.
Again, this will be possible if the weights satisfy certain conditions.

After having made assumptions on the weights,
we have to know if there exist weight sets that satisfy them.
In particular, we will prove that every set of weights has a \enquote{\minSatExt}.

\subsection{Definition and properties}\label{susec:Def_Weighted_VFs}

We give the definition of weighted vector fields.
\begin{defi}[Weighted vector fields and localizations]
	Let $d \in \N^*$, $M$ a $d$-dimensional manifold,
	$\atlasA = \sset{\kappa : \widetilde{U}_\kappa \to U_\kappa}$ an atlas for $M$,
	$\GewFunk \sub \cl{\R}^M$ nonempty and $k\in \cl{\N}$.
	For $f \in \GewFunk$ and $\ell \in \N$ with $\ell \leq k$,
	we define
	\[
		\hnM{\cdot}{f}{\ell}{\atlasA}
		:
		\prod_{\kappa \in \atlasA} \ConDiff{U_\kappa}{\R^d}{k} \to [0, \infty]
		:
		\famiAtl{\gamma_\kappa}{\atlasA}
		\mapsto
		\sup_{\kappa \in \atlasA} \hn{\gamma_\kappa}{f \circ \kappa^{-1}}{\ell}.
	\]
	For $X \in \VecFields[k]{M}$, we define
	\[
		\hnM{X}{f}{\ell}{\atlasA}
		\ndef \hnM{ \famiAtl{\VectFLok{X}{\kappa}}{\atlasA} }{f}{\ell}{\atlasA}
	\]
	and with that
	\[
		\CcFM{M}{k}{\atlasA}
		\ndef
		\set{X \in \VecFields[k]{M}}%
		{(\forall f \in \GewFunk, \ell \in \N, \ell \leq k)\, \hnM{X}{f}{\ell}{\atlasA} < \infty}.
	\]
	Obviously $\CcFM{M}{k}{\atlasA}$ is a vector space.
	We endow it with the locally convex topology induced by the seminorms $\hnM{\cdot}{f}{\ell}{\atlasA}$.
	We call its elements \emph{weighted vector fields}.
	Furthermore, we set for $f \in \GewFunk$ and $\kappa \in \atlasA$
	\[
		\weightLok{f}{\kappa} \ndef f \circ \kappa^{-1} : U_\kappa \to \cl{\R}
		\qquad
		\text{and}
		\qquad
		\WeightsLok{\GewFunk}{\kappa} \ndef \set{\weightLok{f}{\kappa}}{f \in \GewFunk}.
	\]
	Finally, we define
	\[
		\weightAtl{f}{\atlasA}
		\ndef \disjointUkM{\weightLok{f}{\kappa} }
		\in \cl{\R}^{\disjointUkM{U_\kappa} }
		\qquad
		\text{and}
		\qquad
		\WeightsAtl{\GewFunk}{\atlasA}
		\ndef \set{ \weightAtl{f}{\atlasA} }{f \in \GewFunk}.
	\]
\end{defi}

\begin{lem}\label{lem:WeiVF_closed_image}
	Let $d \in \N^*$, $M$ a $d$-dimensional manifold,
	$\atlasA = \sset{\kappa : \widetilde{U}_\kappa \to U_\kappa}$ an atlas for $M$,
	$k\in \cl{\N}$ and $\GewFunk \sub \cl{\R}^M$ such that for each $p \in M$,
	there exists $f_p \in \GewFunk$ with $f_p(p) \neq 0$.
	Then the map
	\[
		\embMcWatl{\atlasA}
		:
		\CcFM{M}{k}{\atlasA} \to
		\CFpro{U_\kappa}{\R^d}{\WeightsAtl{\GewFunk}{\atlasA}}{k}{\kappa}{\atlasA}
		: \phi \mapsto \fami{\phi_\kappa}{\kappa}{\atlasA}
	\]
	is a linear topological embedding, with closed image.
\end{lem}
\begin{proof}
	That the map is defined and an embedding is obvious from the definition of
	$\CcFM{M}{k}{\atlasA}$ and $\CFpro{U_\kappa}{\R^d}{\WeightsAtl{\GewFunk}{\atlasA}}{k}{\kappa}{\atlasA}$.
	To see that the image is closed, let $\famiI{X^i}$ be a net in $\CcFM{M}{k}{\atlasA}$
	such that $\famiI{\embMcWatl{\atlasA}(X^i)}$ converges to $\famiAtl{X_\kappa}{\atlasA}$.
	We have to show that for $\kappa_1, \kappa_2 \in \atlasA$
	with $\widetilde{U}_{\kappa_1} \cap \widetilde{U}_{\kappa_2} \neq \emptyset$,
	\[
		\tag{\ensuremath{\dagger}}
		\label{id:char_VF_Kartenwechsel}
		\rest{\VectFLok{X}{\kappa_1}}{\kappa_1(\widetilde{U}_{\kappa_1} \cap \widetilde{U}_{\kappa_2})}
		= \dA{(\kappa_1 \circ \kappa_2^{-1})}{}{} \circ (\id{U_{\kappa_2}}, \VectFLok{X}{\kappa_2})
			\circ \rest{\kappa_2\circ \kappa_1^{-1}}{\kappa_1(\widetilde{U}_{\kappa_1} \cap \widetilde{U}_{\kappa_2})}
		.
	\]
	But since the stated assumption on $\GewFunk$ implies that $\famiI{X^i_\kappa}$
	converges pointwise to $X_\kappa$ for each $\kappa \in \atlasA$,
	and since \ref{id:char_VF_Kartenwechsel} holds for all $X^i_{\kappa_1}$ and $X^i_{\kappa_2}$,
	we see that it also holds for $X_{\kappa_1}$ and $X_{\kappa_2}$.
\end{proof}

\subsubsection{Comparison of weighted vector fields with regard to different atlases}
We examine the relationship between spaces $\CcFM{M}{k}{\atlasA}$ and $\CcFM{M}{k}{\atlasB}$
for atlases $\atlasA$ and $\atlasB$.
To this end, we define some terminology for atlases.
\begin{defi}
	Let $d \in \N^*$, $M$ a $d$-dimensional manifold,
	and $\atlasA = \sset{\kappa : \widetilde{U}_\kappa \to U_\kappa}$ an atlas for $M$.
	We call $\atlasA$ \emph{locally finite} if $\famiAtl{\widetilde{U}_\kappa}{\atlasA}$
	is a locally finite cover of $M$.
	Let $\atlasB$ be another atlas for $M$. We call $\atlasB$ \emph{subordinate to $\atlasA$}
	if for each chart $\kappa : \widetilde{V}_\kappa \to V_\kappa$ of $\atlasB$
	there exists $\hat{\kappa} \in \atlasA$ such that
	$\kappa = \rest[V_\kappa]{\hat{\kappa}}{\widetilde{V}_\kappa}$.
	Finally, we define
	\[
		\atProd{\atlasA}{\atlasB} \ndef \set{(\kappa, \phi) \in \atlasA \times \atlasB}{\widetilde{U}_\kappa \cap \widetilde{U}_\phi \neq \emptyset}
	\]
	and
	\[
		\atCap{\atlasA}{\atlasB} \ndef \set{\rest{\kappa}{\widetilde{U}_\kappa \cap \widetilde{U}_\phi} }{(\kappa, \phi) \in \atProd{\atlasA}{\atlasB}}.
	\]
\end{defi}
We state two easy results. First, we show that $\CcFM{M}{k}{\atlasA} = \CcFM{M}{k}{\atCap{\atlasA}{\atlasB}}$,
and that $\CcFM{M}{k}{\atlasA} \sub \CcFM{M}{k}{\atlasB}$ if $\atlasB$ is subordinate to $\atlasA$.
\begin{lem}\label{lem:VGL_gewVF_zerstueckelterAtlas}
	Let $d \in \N^*$, $M$ a $d$-dimensional manifold,
	$\atlasA$ and $\atlasB$ atlases for $M$,
	$\GewFunk \sub \cl{\R}^M$ nonempty and $k\in \cl{\N}$.
	Then $\CcFM{M}{k}{\atlasA} = \CcFM{M}{k}{\atCap{\atlasA}{\atlasB}}$.
\end{lem}
\begin{proof}
	This is obvious since for $X \in \VecFields[k]{M}$, $f \in \GewFunk$ and $\ell \in \N$ with $\ell \leq k$,
	the sets
	\[
		\set{\abs{f(x)}\,\Opnorm{ \FAbl[\ell]{\VectFLok{X}{\kappa}}(\kappa(x)) } }{\kappa \in \atlasA, x \in U_\kappa}
	\]
	and
	\[
		\set{\abs{f(x)}\,\Opnorm{ \FAbl[\ell]{\VectFLok{X}{\kappa}}(\kappa(x)) } }%
		{(\kappa, \phi) \in \atProd{\atlasA}{\atlasB}, x \in U_\kappa \cap U_\phi}
	\]
	are the same.
\end{proof}

\begin{lem}\label{lem:Inklusion_gewVF_verkleinerter-Atlas}
	Let $d \in \N^*$, $M$ a $d$-dimensional manifold,
	$\atlasA = \sset{\kappa : \widetilde{U}_\kappa \to U_\kappa}$ an atlas for $M$,
	$\atlasB = \sset{\kappa : \widetilde{V}_\kappa \to V_\kappa}$ an atlas subordinate to $\atlasA$,
	$\GewFunk \sub \cl{\R}^M$ nonempty and $k\in \cl{\N}$.
	Then $\CcFM{M}{k}{\atlasA} \sub \CcFM{M}{k}{\atlasB}$,
	and the inclusion map is continuous linear.
\end{lem}
\begin{proof}
	Let $f \in \GewFunk$ and $\ell \in \N$ with $\ell \leq k$.
	Since for each $\kappa \in \atlasB$ there exists $\hat{\kappa} \in \atlasA$
	with $\kappa = \rest[V_\kappa]{\hat{\kappa}}{\widetilde{V}_\kappa}$,
	we have for $X \in \VecFields[k]{M}$ that
	\begin{multline*}
		\hn{\VectFLok{X}{\kappa}}{\weightLok{f}{\kappa}}{\ell}
		= \sup_{x \in V_\kappa} \abs{(f \circ \kappa^{-1})(x)}\,
			\Opnorm{\FAbl[\ell]{(\dA{\kappa}{}{} \circ X \circ \kappa^{-1})(x)}}
		\\
		\leq \sup_{x \in U_{\hat{\kappa}}} \abs{(f \circ \hat{\kappa}^{-1})(x)}\,
					\Opnorm{\FAbl[\ell]{(\dA{\hat{\kappa}}{}{} \circ X \circ \hat{\kappa}^{-1})(x)}}
		= \hn{\VectFLok{X}{\hat{\kappa}}}{\weightLok{f}{\hat{\kappa}}}{\ell}.
	\end{multline*}
	This shows the assertion.
\end{proof}
\paragraph{Weights with transition maps as multipliers}
We show the main result of this subsection. If for two atlases $\atlasA$, $\atlasB$,
the differentials of the transition maps from $\atlasB$ to $\atlasA$ are \enquote{simultaneous multipliers} for $\GewFunk$
(that is they satisfy \ref{cond:est_sim-multiplier_weights}),
then $\CcFM{M}{k}{\atlasB} \sub \CcFM{M}{k}{\atCap{\atlasA}{\atlasB}}$.
If additionally $\atlasB$ is subordinate to $\atlasA$, we have that
$\CcFM{M}{k}{\atlasB} = \CcFM{M}{k}{\atlasA}$.

\begin{prop}\label{prop:VGL_VF_Topos-Chart_Changes_Weights}
	Let $d \in \N^*$, $M$ a $d$-dimensional manifold,
	$\atlasA = \sset{\kappa : \widetilde{U}_\kappa \to U_\kappa}$ and
	$\atlasB = \sset{\phi : \widetilde{V}_\phi \to V_\phi}$ atlases for $M$ and $k\in \cl{\N}$.
	Further, let $\GewFunk \sub \cl{\R}^M$ such that \ref{cond:est_sim-multiplier_weights} is satisfied for
	$\WeightsAtl{\GewFunk}{ \atCap{\atlasB}{\atlasA} }$ and
	$
		\fami{\rest{\FAbl{(\kappa \circ \phi^{-1})}}{\phi(\widetilde{U}_\kappa \cap \widetilde{U}_\phi)} }{(\kappa, \phi)}{ \atProd{\atlasA}{\atlasB}}
	$
	and there exists $\omega \in \GewFunk$ with $\abs{\omega} \geq 1$.
	Then the following assertions hold:
	\begin{assertions}
		\item\label{ass1:SuPo_Diff_ChartChange_glatt}
		The map
		\[
			\tag{\ensuremath{\dagger}}
			\label{SuPo_mit_Diff_TransitionMap}
			\prod_{\substack{(\kappa, \phi) \in\\ \atProd{\atlasA}{\atlasB}}} \dA{(\kappa \circ \phi^{-1})}{}{}_*
			:
			\CFpro{\phi(\widetilde{U}_\kappa \cap \widetilde{U}_\phi)}{\R^d}{\WeightsAtl{\GewFunk}{\atCap{\atlasB}{\atlasA}}}{k}{(\kappa, \phi)}{\atProd{\atlasA}{\atlasB}}
			\to \CFpro{\phi(\widetilde{U}_\kappa \cap \widetilde{U}_\phi)}{\R^d}{\WeightsAtl{\GewFunk}{\atCap{\atlasB}{\atlasA}}}{k}{(\kappa, \phi)}{\atProd{\atlasA}{\atlasB}}
		\]
		is continuous linear.

		\item\label{ass1:kontrKompo_ChartChange_cont}
		The map
		\begin{equation}
			\tag{\ensuremath{\dagger\dagger}}
			\label{Kompo_mit_TransitionMap}
			\begin{aligned}
				\CFpro{\phi(\widetilde{U}_\kappa \cap \widetilde{U}_\phi)}{\R^d}{\WeightsAtl{\GewFunk}{\atCap{\atlasB}{\atlasA}}}{k}{(\kappa, \phi)}{\atProd{\atlasA}{\atlasB}}
				&\to \CFpro{\kappa(\widetilde{U}_\kappa \cap \widetilde{U}_\phi)}{\R^d}{\WeightsAtl{\GewFunk}{\atCap{\atlasA}{\atlasB}}}{k}{(\kappa, \phi)}{\atProd{\atlasA}{\atlasB}}
				\\
				\fami{\gamma_{\kappa, \phi}}{(\kappa, \phi)}{\atProd{\atlasA}{\atlasB}}
				&\mapsto
				\fami{\gamma_{\kappa, \phi} \circ \phi \circ \kappa^{-1} }{(\kappa, \phi)}{\atProd{\atlasA}{\atlasB}}
			\end{aligned}
		\end{equation}
		is defined and continuous.

		\item\label{ass1:Inklusion_gewAbb_MFKs}
		$\CcFM{M}{k}{\atlasB} \sub \CcFM{M}{k}{\atCap{\atlasA}{\atlasB}}$,
		and the inclusion map is continuous linear.
	\end{assertions}
\end{prop}
\begin{proof}
	\ref{ass1:SuPo_Diff_ChartChange_glatt}
	Since $\abs{\weightLok{\omega}{\phi}} \geq 1 = \max(\tfrac{1}{1}, 1)$ for each $\phi \in \atlasB$,
	we can apply \ref{est:1-0-norm_f-0-norm_spezielles-f} to see that for
	$\fami{\gamma_{\kappa, \phi}}{(\kappa, \phi)}{\atProd{\atlasA}{\atlasB}} \in \CFpro{\phi(\widetilde{U}_\kappa \cap \widetilde{U}_\phi)}{\R^d}{\WeightsAtl{\GewFunk}{\atCap{\atlasB}{\atlasA}}}{k}{(\kappa, \phi)}{\atProd{\atlasA}{\atlasB}}$,
	\[
		(\forall (\kappa, \phi) \in \atProd{\atlasA}{\atlasB})\,
		\noma{\gamma_{(\kappa, \phi)}}
		\leq \hn{ \gamma_{(\kappa, \phi)} }{ \weightLok{\omega}{\phi}}{0}
		\leq \hn{ \fami{\gamma_{(\kappa, \phi)}}{(\kappa, \phi)}{\atProd{\atlasA}{\atlasB}} }{ \weightAtl{\omega}{\atCap{\atlasB}{\atlasA} }}{0}.
	\]
	Hence
	\[
		\CFpro{\phi(\widetilde{U}_\kappa \cap \widetilde{U}_\phi)}{\R^d}{\WeightsAtl{\GewFunk}{\atCap{\atlasB}{\atlasA}}}{k}{(\kappa, \phi)}{\atProd{\atlasA}{\atlasB}}
		= \bigcup_{R>1}\CFoPRO{\phi(\widetilde{U}_\kappa \cap \widetilde{U}_\phi)}{ \Ball{0}{R} }{\WeightsAtl{\GewFunk}{\atCap{\atlasB}{\atlasA}}}{k}{(\kappa, \phi)}{\atProd{\atlasA}{\atlasB}},
	\]
	and the sets on the right hand side are open subsets of the space on the left hand side.
	Using our other assumption on $\GewFunk$ and \refer{lem:vergleich_Bedingungen_simultane-multiplier_simu-Supo},
	we can apply \refer{prop:simultane_SP_BCinf0_Produkt}
	to see that \ref{SuPo_mit_Diff_TransitionMap} is smooth on each set
	$\CFoPRO{\phi(\widetilde{U}_\kappa \cap \widetilde{U}_\phi)}{ \Ball{0}{R} }{\WeightsAtl{\GewFunk}{\atCap{\atlasB}{\atlasA}}}{k}{(\kappa, \phi)}{\atProd{\atlasA}{\atlasB}}$,
	and hence on $\CFpro{\phi(\widetilde{U}_\kappa \cap \widetilde{U}_\phi)}{\R^d}{\WeightsAtl{\GewFunk}{\atCap{\atlasB}{\atlasA}}}{k}{(\kappa, \phi)}{\atProd{\atlasA}{\atlasB}}$.
	It is obviously linear since each $\dA{(\kappa \circ \phi^{-1})}{}{}$ is so in its second argument.

	\ref{ass1:kontrKompo_ChartChange_cont}
	We prove this with an induction on $k$.

	$k = 0$:
	Let $f \in \GewFunk$. For all $(\kappa, \phi) \in \atProd{\atlasA}{\atlasB}$,
	we have $\weightLok{f}{\kappa} = \weightLok{f}{\phi} \circ \phi \circ \kappa^{-1}$.
	Hence for $\gamma_{\kappa, \phi} \in \CF{\phi(\widetilde{U}_\kappa \cap \widetilde{U}_\phi)}{\R^d}{\WeightsLok{\GewFunk}{\phi}}{k}$
	and $x \in \kappa(\widetilde{U}_\kappa \cap \widetilde{U}_\phi)$, we have that
	\[
		\abs{\weightLok{f}{\kappa}(x)} \, \norm{(\gamma_{\kappa, \phi} \circ \phi \circ \kappa^{-1})(x)}
		= \abs{ (\weightLok{f}{\phi} \circ \phi \circ \kappa^{-1})(x) } \, \norm{(\gamma_{\kappa, \phi} \circ \phi \circ \kappa^{-1})(x)}
		\leq \hn{\gamma_{\kappa, \phi}}{\weightLok{f}{\phi}}{0}
		.
	\]
	Since \ref{Kompo_mit_TransitionMap} is linear and $(\kappa, \phi) \in \atProd{\atlasA}{\atlasB}$
	was arbitrary, we see with this estimate that \ref{Kompo_mit_TransitionMap} is defined and continuous.

	$k \to k + 1$:
	We use \refer{prop:Zerlegungssatz_Familie}. We calculate that for $(\kappa, \phi) \in \atProd{\atlasA}{\atlasB}$,
	\[
		\FAbl{ (\gamma_{\kappa, \phi} \circ \phi \circ \kappa^{-1}) }
		= \FAbl{ \gamma_{\kappa, \phi} } \circ \phi \circ \kappa^{-1} \MaMu \FAbl{(\phi \circ \kappa^{-1})}
		.
	\]
	We see using the inductive hypothesis that
	\[
		\fami{\FAbl{ \gamma_{\kappa, \phi} } \circ \phi \circ \kappa^{-1}}{(\kappa, \phi)}{\atProd{\atlasA}{\atlasB}}
		\in \CFpro{\kappa(\widetilde{U}_\kappa \cap \widetilde{U}_\phi)}{\Lin{\R^d}{\R^d}}{\WeightsAtl{\GewFunk}{\atCap{\atlasA}{\atlasB}}}{k}{(\kappa, \phi)}{\atProd{\atlasA}{\atlasB}},
	\]
	and that the corresponding map is continuous.
	Finally, we get the assertion using \refer{lem:simultane_mult-multiplier}.

	\ref{ass1:Inklusion_gewAbb_MFKs}
	Let $(\kappa, \phi) \in \atProd{\atlasA}{\atlasB}$.
	On $\kappa(\widetilde{U}_\kappa \cap \widetilde{U}_\phi)$, we have the identity
	\[
		\VectFLok{X}{\kappa}
		= \dA{\kappa}{}{} \circ X \circ \kappa^{-1}
		= \pi_2 \circ \Tang{(\kappa \circ \phi^{-1})} \circ \Tang{\phi} \circ X \circ \phi^{-1} \circ \phi \circ \kappa^{-1}
		= (\phi \circ \kappa^{-1})^*(\dA{(\kappa \circ \phi^{-1})}{}{}_*(\VectFLok{X}{\phi}))
		.
	\]
	Since $\kappa$ and $\phi$ were arbitrary, we can use that
	the maps \ref{SuPo_mit_Diff_TransitionMap} and \ref{Kompo_mit_TransitionMap} are continuous linear
	to derive estimates which ensure that
	$\CcFM{M}{k}{\atCap{\atlasB}{\atlasA}} \sub \CcFM{M}{k}{\atCap{\atlasA}{\atlasB}}$,
	and that the inclusion map is continuous.
	Since $\atCap{\atlasB}{\atlasA}$ is subordinate to $\atlasB$,
	we derive the assertion using \refer{lem:Inklusion_gewVF_verkleinerter-Atlas}.
\end{proof}

\begin{cor}\label{cor:gewVF-verschiedeneAtlanten-Chart_Changes_Weights}
	Let the data be as in \refer{prop:VGL_VF_Topos-Chart_Changes_Weights}, and additionally assume that $\atlasB$ is subordinate to $\atlasA$.
	Then $\CcFM{M}{k}{\atlasB} = \CcFM{M}{k}{\atlasA}$ as topological vector space.
\end{cor}
\begin{proof}
	We know from \refer{lem:Inklusion_gewVF_verkleinerter-Atlas} that $\CcFM{M}{k}{\atlasA} \sub \CcFM{M}{k}{\atlasB}$,
	and from \refer{prop:VGL_VF_Topos-Chart_Changes_Weights} and \refer{lem:VGL_gewVF_zerstueckelterAtlas} that $\CcFM{M}{k}{\atlasB} \sub \CcFM{M}{k}{\atlasA}$.
\end{proof}

\subsection{Simultaneous composition, inversion and superposition with Riemannian exponential map and logarithm}

We study simultaneous composition and inversion
(see \refer{prop:Simultane_Koor-Kompo_diffbar} and \refer{prop:Simultane_Inv-Kompo_glatt}, respectively)
on families of functions that, roughly speaking, arise as the simultaneous application of the superposition with the Riemannian exponential function;
and the application of simultaneous superposition with the logarithm after these operations.
\paragraph{Rephrasing some previous results}
We apply some results to the special case of functions that are defined on the disjoint union
of chart domains for a manifold.
We start with the simultaneous superposition with (slightly modified) Riemannian exponential maps and logarithms.
\begin{defi}
	Let $d \in \N^*$, $U \sub \R^d$ open, $g$ a Riemannian metric on $U$ and
	$V$ an open, nonempty, relatively compact set such that $\cl{V} \sub U$.
	\begin{assertions}
		\item
		Let $\delta \in ]0,  \grenzExp[g]{V}{U}[$.
		We set
		\[
			\ExpMinusLok{V}{\delta}{g} : V \times \Ball{0}{\delta} \to U
			: (x, y) \mapsto \Rexp{g}(x, y) - x
			.
		\]

		\item
		Let $\delta \in ]0,  \grenzLog[g]{V}{U}[$.
		We set
		\[
			\LogPlusVR{V}{\delta}{g} : V \times \Ball{0}{\delta} \to \R^d
			: (x, y) \mapsto \RlogVR{g} (x, x + y)
		\]
	\end{assertions}
\end{defi}

\begin{lem}
	Let $d \in \N^*$, $M$ a $d$-dimensional manifold, $g$ a Riemannian metric on $M$,
	$\atlasA = \sset{\kappa : \widetilde{U}_\kappa \to U_\kappa} $ an atlas for $M$,
	$\GewFunk \sub \cl{\R}^M$ nonempty and $k\in \cl{\N}$.
	Further, for each $\kappa \in \atlasA$ let $V_\kappa$ be an open,
	nonempty, relatively compact set such that $\cl{V_\kappa} \sub U_\kappa$,
	$\atlasB \ndef \set{\rest[V_\kappa]{\kappa}{\kappa^{-1}(V_\kappa)}}{\kappa \in \atlasA}$ is an atlas for $M$,
	and $\delta_\kappa > 0$.
	Further, we assume that $\WeightsAtl{\GewFunk}{\atlasB}$ contains
	an adjusting weight $\omega$ for $\famiAtl{\delta_\kappa}{\atlasA}$.
	\begin{assertions}
		\item
		Assume that $\delta_\kappa < \grenzExp[\RMetLok{g}{\kappa}]{V_\kappa}{U_\kappa}$,
		and that $\WeightsAtl{\GewFunk}{\atlasB}$ satisfies \ref{cond:est_SP-Abb_weights},
		where $I = \atlasA$ and
		$\beta_\kappa = \ExpMinusLok{V_\kappa}{\delta_\kappa}{\RMetLok{g}{\kappa}}$
		for $\kappa \in \atlasA$.
		Then the map
		\[
			\ExpIdCWprod{V}{\delta}
			\ndef \prod_{\kappa \in \atlasA} \ExpIdCWfak{V_\kappa}{\delta_\kappa}
			: \left\lbrace
			\begin{aligned}
				\CcFfoPROatK{V}{\Ball{0}{\delta_\kappa}}{\atlasB}{k}{\omega}{\atlasA}
				&\to \CcFproAtK{V}{\R^d}{\atlasB}{k}{\atlasA}
				\\
				\phi &\mapsto \famiAtl{\Rexp{\RMetLok{g}{\kappa}} \circ (\id{V_\kappa}, \phi_\kappa) - \id{V_\kappa} }{\atlasA}
			\end{aligned}
			\right.
		\]
		is defined and smooth.

		\item
		Assume that $\delta_\kappa < \grenzLog[\RMetLok{g}{\kappa}]{V_\kappa}{U_\kappa}$,
		and that $\WeightsAtl{\GewFunk}{\atlasB}$ satisfies \ref{cond:est_SP-Abb_weights},
		where $I = \atlasA$ and
		$\beta_\kappa = \LogPlusVR{V_\kappa}{\delta_\kappa}{\RMetLok{g}{\kappa}}$
		for $\kappa \in \atlasA$.
		Then the map
		\[
			\LogIdCWprod{V}{\delta}
			\ndef \prod_{\kappa \in \atlasA} \LogIdCWfak{V_\kappa}{\delta_\kappa}
			:  \left\lbrace
			\begin{aligned}
				\CcFfoPROatK{V}{\Ball{0}{\delta_\kappa}}{\atlasB}{k}{\omega}{\atlasA}
				&\to \CcFproAtK{V}{\R^d}{\atlasB}{k}{\atlasA}
				\\
				\phi &\mapsto \famiAtl{ \RlogVR{\RMetLok{g}{\kappa}}\circ (\id{V_\kappa}, \phi_\kappa + \id{V_\kappa}) }{\atlasA}
			\end{aligned}
			\right.
		\]
		is defined and smooth.
	\end{assertions}
\end{lem}
\begin{proof}
	In both cases, we see that $\beta_\kappa$ maps $V_\kappa \times \sset{0}$ to $\sset{0}$,
	for each $\kappa \in \atlasA$.
	Hence the assertion follows from \refer{prop:simultane_SP_BCinf0_Produkt}.
\end{proof}
We turn to composition and inversion.
\begin{lem}
	Let $d \in \N^*$, $M$ a $d$-dimensional manifold,
	$\atlasA = \sset{\kappa : \widetilde{U}_\kappa \to U_\kappa} $ an atlas for $M$,
	$\GewFunk \sub \cl{\R}^M$ nonempty and $k, \ell \in \cl{\N}$.
	Further, let $r > 0$ and for each $\kappa \in \atlasA$
	let $W_\kappa, V_\kappa \sub U_\kappa$ be open nonempty sets such that $W_\kappa + \Ball{0}{r} \sub V_\kappa$
	and $\famiAtl{\kappa^{-1}(W_\kappa)}{\atlasA}$ is a cover of $M$.
	Then the map
	\begin{equation*}
		\compIdDiffKLWeights{\R^d}{k}{\ell}{\WeightsAtl{\GewFunk}{\atlasB}}
		\ndef \prod_{\kappa \in \atlasA} \compIdDiffKLWeights{\R^d}{k}{\ell}{\WeightsLok{\GewFunk}{\kappa}} : \left\lbrace
		\begin{aligned}
			\CcFproAtK{V}{\R^d}{\atlasB}{k + \ell + 1}{\atlasA} \times \CcFoPROatK{W}{\Ball{0}{r}}{\atlasB}{k}{\atlasA}
			&\to \CcFproAtK{W}{\R^d}{\atlasB}{k}{\atlasA}
			\\
			(\gamma, \eta) &\mapsto \famiAtl{\gamma_\kappa \circ (\eta_\kappa + \id{W_\kappa})}{\atlasA}
		\end{aligned}
		\right.
	\end{equation*}
	is defined and $\ConDiff{}{}{\ell}$,
	where $\atlasB \ndef \set{\rest[V_\kappa]{\kappa}{\kappa^{-1}(V_\kappa)} }{\kappa \in \atlasA}$.
\end{lem}
\begin{proof}
	This is a direct consequence of \refer{prop:Simultane_Koor-Kompo_diffbar}.
	Note that $1_{\disjointUkM{V_\kappa}}$ (eventually multiplied with $\frac{1}{r}$)
	is an adjusting weight for $\famiAtl{\Ball{0}{r}}{\atlasA}$.
\end{proof}

\begin{lem}
	Let $d \in \N^*$, $M$ a $d$-dimensional manifold,
	$\atlasA = \sset{\kappa : \widetilde{U}_\kappa \to U_\kappa} $ an atlas for $M$
	such that each $U_\kappa$ is convex,
	and $\GewFunk \sub \cl{\R}^M$ containing $1_M$.
	Further, let $r > 0$ and for each $\kappa \in \atlasA$
	let $V_\kappa \sub U_\kappa$ be an open nonempty set such that
	$V_\kappa + \Ball{0}{r} \sub U_\kappa$
	and $\famiAtl{\kappa^{-1}(V_\kappa)}{\atlasA}$ is a cover of $M$.
	Then for each $\tau \in ]0, 1[$, the map
	\begin{equation*}
		\InvIdWeights{V}{\WeightsAtl{\GewFunk}{\atlasA}}
		\ndef
		\prod_{\kappa \in \atlasA}
		\InvIdWeights{V_\kappa}{\WeightsLok{\GewFunk}{\kappa}}
		 :
		\mathcal{D}^\tau \to \CcFproAtK{V}{\R^d}{\atlasA}{\infty}{\atlasA}
		:
		\phi \mapsto \famiAtl{ \rest{(\phi_\kappa + \id{U_\kappa})^{-1}}{V_\kappa} - \id{V_\kappa} }{\atlasA}
	\end{equation*}
	is defined and smooth, where
	\[
		\mathcal{D}^\tau \ndef
		\left\set{\phi \in \CcFproAtK{U}{\R^d}{\atlasA}{\infty}{\atlasA} }{%
			\hn{\phi}{1_{\disjointUkM{U_\kappa}}}{1} < \tau
			\text{ and }
			\hn{\phi}{1_{\disjointUkM{U_\kappa}}}{0}
			< \tfrac{r}{2} (1 - \tau)
		\right}.
	\]
\end{lem}
\begin{proof}
	This is a direct consequence of \refer{prop:Simultane_Inv-Kompo_glatt}.
\end{proof}

\paragraph{Composing the operations}
We compose the maps that were introduced in this subsection.
The main difficulty is keeping track of whether the simultaneous operations can be applied.
\begin{lem}\label{lem:Composition_and_Inversion_smooth-patched_version}
	Let $d \in \N^*$, $M$ a $d$-dimensional manifold, $g$ a Riemannian metric on $M$, $r > 0$ and
	$\atlasA = \sset{\kappa : \widetilde{U}_\kappa \to U_\kappa} $ an atlas for $M$.
	For each $\kappa \in \atlasA$, let $V_\kappa, W_\kappa \sub U_\kappa$
	be open, nonempty, relatively compact sets with $W_\kappa + \Ball{0}{r} \sub V_\kappa$
	such that $\cl{V_\kappa} \sub U_\kappa$, $V_\kappa$ is convex
	and $\famiAtl{\kappa^{-1}(W_\kappa)}{\atlasA}$ is a cover of $M$.
	We set $\atlasB \ndef \set{\rest[V_\kappa]{\kappa}{\kappa^{-1}(V_\kappa)}}{\kappa \in \atlasA}$.

	For each $\kappa \in \atlasA$, let $\delta_\kappa^E \in]0, \grenzExp{V_\kappa}{U_\kappa}[$
	and $\delta^L_\kappa \in]0, \grenzLog{W_\kappa}{U_\kappa}[$.
	Let $\GewFunk \sub \cl{\R}^M$ such that $\WeightsAtl{\GewFunk}{\atlasB}$ contains
	an adjusting weight $\omega^E$ for
	\[
		\tag{\ensuremath{\dagger}}
		\label{Compo_patched_adjust_exp}
		\left\famiAtl{\min(\delta^E_\kappa,
			\frac{1}{1 + \BndSndAblRex[\RMetLok{g}{\kappa}]{V_\kappa}{\delta^E_\kappa}},
			\frac{1}{\BndFstAblRex[\RMetLok{g}{\kappa}]{V_\kappa}{\delta^E_\kappa}}
			)\right}{\atlasA}
	\]
	and $\omega^L$ that is adjusting for $\famiAtl{\delta^L_\kappa}{\atlasA}$,
	and satisfies
	\begin{equation}
		\label{est:adjustWeights_exp-log}
		(\forall \kappa \in \atlasA )\,
		\abs{\omega^L_\kappa} \leq
		\frac{1}{\BndFstAblRex[\RMetLok{g}{\kappa}]{V_\kappa}{\delta^E_\kappa} \BndFstAblRlog[\RMetLok{g}{\kappa}]{W_\kappa}{\delta^L_\kappa}}
		\abs{\omega^E_\kappa}
		.
	\end{equation}
	Additionally, assume that $\WeightsAtl{\GewFunk}{\atlasB}$ satisfies \ref{cond:est_SP-Abb_weights}
	for the families
	$\famiAtl{ \ExpMinusLok{V_\kappa}{\delta^E_\kappa}{\RMetLok{g}{\kappa}} }{\atlasA}$
	and $\famiAtl{ \LogPlusVR{W_\kappa}{\delta^L_\kappa}{\RMetLok{g}{\kappa}} }{\atlasA}$,
	respectively.
	\begin{assertions}
		\item\label{ass:patched_Compo_Rexplog}
		Then the map
		\[
			\mathbf{C}_E^L
			:
			D_1 \times D_2 \to R
			: (\gamma, \eta) \mapsto \LogIdCWprod[g]{W}{\delta^L}
			( \compIdWeights{\R^d}{\WeightsAtl{\GewFunk}{\atlasB}}
				(\ExpIdCWprod{V}{\delta^E}(\gamma), \ExpIdCWprod{W}{\delta^E}(\eta)
				) + \ExpIdCWprod{W}{\delta^E}(\eta))
		\]
		is defined and smooth. Here
		\[
			D_1 \ndef \set{\gamma \in \CcFproAtK{V}{\R^d}{\atlasB}{\infty}{\atlasA}}%
			{\hn{\gamma}{\omega^E}{0} < \tfrac{1}{2}
			\text{ and }
			\hn{\gamma}{1_{\disjointUkM{V_\kappa}}}{1} < \tfrac{1}{2}
			},
		\]
		\[
			D_2 \ndef \set{\eta \in \CcFproAtK{W}{\R^d }{\atlasB}{\infty}{\atlasA} }%
			{\hn{\eta}{\omega^E}{0} < \min(\tfrac{1}{4}, r) }
			,
		\]
		and
		\[
			R \ndef \set{\phi \in \CcFproAtK{W}{\R^d }{\atlasB}{\infty}{\atlasA}}{\hn{\phi}{\omega^L}{0} < 1}.
		\]
		Moreover, we have
		\begin{equation}
			\label{est:log0-normCompo}
			\hn{\mathbf{C}_E^L(\gamma, \eta)}{\omega^L}{0}
			\leq (1 + \hn{\gamma}{\omega^E}{0} + \hn{\gamma}{1_{\disjointUkM{V_\kappa}}}{1}) \hn{\eta}{\omega^E}{0} + \hn{\gamma}{\omega^E}{0}.
		\end{equation}
		For $X, Y \in \VecFields{M}$ such that $\embMcWatl{\atlasB}(X) \in D_1$ and $\embMcWatl{\atlasB}(Y) \in D_2$,
		we have for $\kappa \in \atlasA$ that
		\begin{equation}\label{id:Lokale_Darstellung_Komposition_durch_glatte_Abb}
			\begin{multlined}[0.8\columnwidth]
			\LogIdCWfak[\RMetLok{g}{\kappa}]{W_\kappa}{\delta_\kappa}
				( \compIdWeights{\R^d}{\WeightsLok{\GewFunk}{\kappa}}
						(\ExpIdCWfak{V_\kappa}{\delta_\kappa}(\VectFLok{X}{\kappa}), \ExpIdCWfak{W_\kappa}{\delta_\kappa}(\VectFLok{Y}{\kappa}))
					+ \ExpIdCWfak{W_\kappa}{\delta_\kappa}(\VectFLok{Y}{\kappa})
				)
			\\
			= \RlogVR{\RMetLok{g}{\kappa}} \circ (\id{W_\kappa}, \Rexp{\RMetLok{g}{\kappa}} \circ (\id{V_\kappa}, \VectFLok{X}{\kappa}) \circ \Rexp{\RMetLok{g}{\kappa}} \circ (\id{W_\kappa}, \VectFLok{Y}{\kappa}) ).
			\end{multlined}
		\end{equation}

		\item\label{ass:patched_Inversion_Rexplog}
		Additionally, let $\rho \in ]0, 1[$. Then the map
		\[
			\mathbf{I}_E^L
			:
			D_{\rho} \to R_\rho
			: \phi \mapsto \LogIdCWprod{W}{\delta^L}
			( \InvIdWeights{W}{\WeightsAtl{\GewFunk}{\atlasB}}
				(\ExpIdCWprod{V}{\delta^E}(\phi))
		\]
		is defined and smooth. Here
		\begin{equation*}
			D_{\rho} \ndef \set{\phi \in \CcFproAtK{V}{\R^d}{\atlasB}{\infty}{\atlasA}}%
			{\hn{\phi}{\omega^E}{0}
			< \tfrac{(1 - \rho)}{2} \min(\rho, r)
			\text{ and }
			\hn{\phi}{1_{\disjointUkM{V_\kappa}}}{1} < \tfrac{\rho}{2}
			}
		\end{equation*}
		and
		\[
			R_\rho \ndef \set{\phi \in \CcFproAtK{W}{\R^d }{\atlasB}{\infty}{\atlasA}}%
			{\hn{\phi}{\omega^L}{0} < \tfrac{\min(r, \rho)}{2} }.
		\]
		Moreover, we have that
		\begin{equation}
			\label{est:log0-normInv}
			\hn{\mathbf{I}_E^L(\phi)}{\omega^L}{0}
			\leq
			\frac{\hn{\phi}{\omega^E}{0} }{1 - (\hn{\phi}{\omega^E}{0} + \hn{\phi}{1_{\disjointUkM{V_\kappa}} }{1})}
		\end{equation}
		For $X\in \VecFields{M}$ such that $\embMcWatl{\atlasB}(X) \in D_{\rho}$,
		we have for $\kappa \in \atlasA$ that
		\begin{equation}\label{id:Lokale_Darstellung_Inversion_durch_glatte_Abb}
			\LogIdCWfak{W_\kappa}{\delta^L_\kappa}
				( \InvIdWeights{W_\kappa}{\WeightsLok{\GewFunk}{\kappa}}
					(\ExpIdCWfak{V_\kappa}{\delta^E_\kappa}(\VectFLok{X}{\kappa})))
			= \RlogVR{\RMetLok{g}{\kappa}} \circ(\id{W_\kappa}, \rest{(\Rexp{\RMetLok{g}{\kappa}} \circ (\id{V_\kappa}, \VectFLok{X}{\kappa}))^{-1}}{W_\kappa}).
		\end{equation}
	\end{assertions}
	Note that above we occasionally identified maps with their restriction.
\end{lem}
\begin{proof}
	\ref{ass:patched_Compo_Rexplog}
	By our previous elaborations,
	the map $\mathbf{C}_E^L$ is smooth if it is defined,
	so we shall prove the latter.
	Let $\gamma \in D_1$ and $\eta \in D_2$.
	Since $\omega^E$ is an adjusting weight for $\famiAtl{\delta^E_\kappa}{\atlasA}$
	and $\hn{\gamma}{\omega^E}{0}, \hn{\eta}{\omega^E}{0} < 1$ by our assumptions,
	we see with \ref{incl:1-Kugel_f0-norm_sub_CFof} that
	$\gamma \in \CcFfoPROatK{V}{\Ball{0}{\delta^E_\kappa}}{\atlasB}{\infty}{\omega^E}{\atlasA}$
	and
	$\eta \in \CcFfoPROatK{W}{\Ball{0}{\delta^E_\kappa}}{\atlasB}{\infty}{\omega^E}{\atlasA}$.
	Hence we can apply
	$\ExpIdCWprod{V}{\delta^E}$ to $\gamma$ and $\ExpIdCWprod{W}{\delta^E}$ to $\eta$.
	\\
	We see with \refer{lem:Naehe_expVF-id_C0} using that $\omega^E$ is adjusted to \ref{Compo_patched_adjust_exp}
	(and \ref{est:1-0-norm_f-0-norm_spezielles-f}) that for $\kappa \in \atlasA$,
	\[
		\hn{\ExpIdCWfak{W_\kappa}{\delta_\kappa}(\eta_\kappa) }{1_{W_\kappa}}{0}
		\leq \tfrac{\BndFstAblRex[\RMetLok{g}{\kappa}]{W_\kappa}{\delta^E_\kappa}}%
			{\BndFstAblRex[\RMetLok{g}{\kappa}]{V_\kappa}{\delta^E_\kappa}}
		\hn{\eta_\kappa }{\omega^E_\kappa}{0}
		< r.
	\]
	This implies that we can apply $\compIdWeights{\R^d}{\WeightsAtl{\GewFunk}{\atlasB}}$
	to $(\ExpIdCWprod{V}{\delta}(\gamma), \ExpIdCWprod{W}{\delta}(\eta))$.
	Further, we conclude from \refer{lem:Naehe_expFunc-id_C1} using \ref{est:1-0-norm_f-0-norm_spezielles-f}
	that for each $\kappa \in \atlasA$,
	\[
		\hn{\ExpIdCWfak{V_\kappa}{\delta_\kappa}(\gamma_\kappa)}{1_{V_\kappa}}{1}
		\leq\BndSndAblRex[\RMetLok{g}{\kappa}]{V_\kappa}{\delta^E_\kappa} \hn{\gamma_\kappa}{1_{V_\kappa}}{0}
			+  \hn{\gamma_\kappa}{1_{V_\kappa}}{1}
		\leq \tfrac{\BndSndAblRex[\RMetLok{g}{\kappa}]{V_\kappa}{\delta^E_\kappa}}%
			{1 + \BndSndAblRex[\RMetLok{g}{\kappa}]{V_\kappa}{\delta^E_\kappa}}
			\hn{\gamma_\kappa}{\omega^E_\kappa}{0}
			+ \hn{\gamma_\kappa}{1_{V_\kappa}}{1}
		.
	\]
	Using the elementary estimate \cite[(4.1.3.1)]{arxiv_1006.5580v3},
	the last estimate, \refer{lem:Naehe_expVF-id_C0}
	(and the triangle inequality) we see that for $f \in \GewFunk$ and $\kappa \in \atlasA$,
	\begin{equation*}
		\tag{\ensuremath{\ast}}
		\label{est:kompo_Exp_beschr_Abl}
		\begin{aligned}
		&\hn{\compIdWeights{\R^d}{\WeightsLok{\GewFunk}{\kappa}}
			(\ExpIdCWfak{V_\kappa}{\delta^E_\kappa}(\gamma_\kappa), \ExpIdCWfak{W_\kappa}{\delta^E_\kappa}(\eta_\kappa)
			) + \ExpIdCWfak{W_\kappa}{\delta^E_\kappa}(\eta_\kappa)}{f_\kappa }{0}
		\\
		\leq& \hn{\ExpIdCWfak{V_\kappa}{\delta^E_\kappa}(\gamma_\kappa)}{1_{V_\kappa} }{1} \hn{ \ExpIdCWfak{W_\kappa}{\delta^E_\kappa}(\eta_\kappa)}{f_\kappa }{0}
				+ \hn{\ExpIdCWfak{V_\kappa}{\delta^E_\kappa}(\gamma_\kappa)}{f_\kappa }{0}
				+ \hn{ \ExpIdCWfak{W_\kappa}{\delta^E_\kappa}(\eta_\kappa)}{f_\kappa }{0}
		\\
		\leq& \BndFstAblRex[\RMetLok{g}{\kappa}]{V_\kappa}{\delta^E_\kappa}
		\bigl((\hn{\gamma_\kappa}{\omega^E_\kappa}{0} + \hn{\gamma_\kappa}{1_{V_\kappa}}{1} + 1)
				\hn{ \eta_\kappa}{f_\kappa }{0} + \hn{ \gamma_\kappa}{f_\kappa }{0}\bigr)
				.
		\end{aligned}
	\end{equation*}
	From this estimate we derive using \ref{est:adjustWeights_exp-log} and
	$\hn{\gamma_\kappa}{\omega^E_\kappa}{0}, \hn{\gamma_\kappa}{1_{V_\kappa}}{1} < \tfrac{1}{2}$,
	$\hn{ \eta_\kappa}{\omega^E_\kappa }{0} < \tfrac{1}{4}$
	that
	\begin{multline*}
		\hn{\compIdWeights{\R^d}{\WeightsLok{\GewFunk}{\kappa}}
					(\ExpIdCWfak{V_\kappa}{\delta^E_\kappa}(\gamma_\kappa), \ExpIdCWfak{W_\kappa}{\delta^E_\kappa}(\eta_\kappa)
					) + \ExpIdCWfak{W_\kappa}{\delta^E_\kappa}(\eta_\kappa)}{\omega^L_\kappa }{0}
		\\
		< \BndFstAblRex[\RMetLok{g}{\kappa}]{V_\kappa}{\delta^E_\kappa}
				(2 \hn{ \eta_\kappa}{\omega^L_\kappa }{0} + \hn{ \gamma_\kappa}{\omega^L_\kappa }{0})
		\leq \tfrac{\BndFstAblRex[\RMetLok{g}{\kappa}]{V_\kappa}{\delta^E_\kappa}}%
		{\BndFstAblRex[\RMetLok{g}{\kappa}]{V_\kappa}{\delta^E_\kappa}\BndFstAblRlog[\RMetLok{g}{\kappa}]{W_\kappa}{\delta^L_\kappa}}
		(2 \hn{ \eta_\kappa}{\omega^E_\kappa }{0} + \hn{ \gamma_\kappa}{\omega^E_\kappa }{0})
		< 1
		.
	\end{multline*}
	We conclude with \ref{incl:1-Kugel_f0-norm_sub_CFof} that we can apply $\LogIdCWprod{W}{\delta^L}$ to
	$\compIdWeights{\R^d}{\WeightsAtl{\GewFunk}{\atlasB}}
				(\ExpIdCWprod{V}{\delta^E}(\gamma), \ExpIdCWprod{W}{\delta^E}(\eta)
				) + \ExpIdCWprod{W}{\delta^E}(\eta) $.
	Further, for each $\kappa \in \atlasA$ we have (using \refer{lem:Norm_LogFunc_C0}
	and \ref{est:adjustWeights_exp-log}) that
	\[
		\hn{\LogIdCWfak{W_\kappa}{\delta^L_\kappa}(\phi)}{\omega^L_\kappa }{0}
		\leq \tfrac{\BndFstAblRlog[\RMetLok{g}{\kappa}]{V_\kappa}{\delta^L_\kappa}}%
		{\BndFstAblRex[\RMetLok{g}{\kappa}]{V_\kappa}{\delta^E_\kappa} \BndFstAblRlog[\RMetLok{g}{\kappa}]{W_\kappa}{\delta^L_\kappa}}
		\hn{\phi}{\omega^E_\kappa }{0}
	\]
	for suitable $\phi$. From this and \ref{est:kompo_Exp_beschr_Abl}
	we derive the assertion on the containment in $R$,
	and also that \ref{est:log0-normCompo} holds.
	It remains to prove \ref{id:Lokale_Darstellung_Komposition_durch_glatte_Abb}.
	To this end, let $p \in W_\kappa$. Then
	\begin{align*}
		&\LogIdCWfak[\RMetLok{g}{\kappa}]{W_\kappa}{\delta_\kappa}
				( \compIdWeights{\R^d}{\WeightsLok{\GewFunk}{\kappa}}
					(\ExpIdCWfak{V_\kappa}{\delta_\kappa}(\VectFLok{X}{\kappa}), \ExpIdCWfak{W_\kappa}{\delta_\kappa}(\VectFLok{Y}{\kappa}))
					+ \ExpIdCWfak{W_\kappa}{\delta_\kappa}(\VectFLok{Y}{\kappa}))(p)
		\\
		= &\RlogVR{\RMetLok{g}{\kappa}}(p, \compIdWeights{\R^d}{\WeightsLok{\GewFunk}{\kappa}}
				(\ExpIdCWfak{V_\kappa}{\delta_\kappa}(\VectFLok{X}{\kappa}), \ExpIdCWfak{W_\kappa}{\delta_\kappa}(\VectFLok{Y}{\kappa}))(p)
				+ \ExpIdCWfak{W_\kappa}{\delta_\kappa}(\VectFLok{Y}{\kappa})(p) + p)
		\\
		= &\RlogVR{\RMetLok{g}{\kappa}}(p, \ExpIdCWfak{V_\kappa}{\delta_\kappa}(\VectFLok{X}{\kappa})(\Rexp{\RMetLok{g}{\kappa}}(p, \VectFLok{Y}{\kappa}(p)) )
					+ \Rexp{\RMetLok{g}{\kappa}}(p, \VectFLok{Y}{\kappa}(p)))
		\\
		= &\RlogVR{\RMetLok{g}{\kappa}}(p, \Rexp{\RMetLok{g}{\kappa}}(\Rexp{\RMetLok{g}{\kappa}}(p, \VectFLok{Y}{\kappa}(p)), \VectFLok{X}{\kappa}(\Rexp{\RMetLok{g}{\kappa}}(p, \VectFLok{Y}{\kappa}(p))) )
		\\
		= &\RlogVR{\RMetLok{g}{\kappa}}(p, (\Rexp{\RMetLok{g}{\kappa}} \circ (\id{V_\kappa}, \VectFLok{X}{\kappa}))(\Rexp{\RMetLok{g}{\kappa}}(p, \VectFLok{Y}{\kappa}(p))) )
		\\
		=& \RlogVR{\RMetLok{g}{\kappa}}(p, (\Rexp{\RMetLok{g}{\kappa}} \circ (\id{V_\kappa}, \VectFLok{X}{\kappa}) \circ \Rexp{\RMetLok{g}{\kappa}} \circ (\id{W_\kappa}, \VectFLok{Y}{\kappa}))(p)).
	\end{align*}
	This shows that \ref{id:Lokale_Darstellung_Komposition_durch_glatte_Abb} holds.

	\ref{ass:patched_Inversion_Rexplog}
	By our previous elaborations,
	the map $\mathbf{I}_E^L$ is smooth if it is defined,
	so we shall prove the latter.
	Let $\phi \in D_{\rho}$.
	Since $\omega^E$ is an adjusting weight for $\famiAtl{\delta^E_\kappa}{\atlasA}$
	and $\hn{\phi}{\omega^E}{0} < 1$ by our assumptions,
	we see with \ref{incl:1-Kugel_f0-norm_sub_CFof} that
	$\phi \in \CcFfoPROatK{V}{\Ball{0}{\delta^E_\kappa}}{\atlasB}{\infty}{\omega}{\atlasA}$.
	Hence we can apply $\ExpIdCWprod{V}{\delta^E}$ to $\phi$.
	We see with \refer{lem:Naehe_expVF-id_C0} using that $\omega^E$
	is adjusting for \ref{Compo_patched_adjust_exp} (and \ref{est:1-0-norm_f-0-norm_spezielles-f})
	that for each $\kappa \in \atlasA$,
	\[
		\hn{\ExpIdCWfak{V_\kappa}{\delta^E_\kappa}(\phi_\kappa)}{1_{V_\kappa}}{0}
		<  \tfrac{\BndFstAblRex[\RMetLok{g}{\kappa}]{V_\kappa}{\delta^E_\kappa}}{\BndFstAblRex[\RMetLok{g}{\kappa}]{V_\kappa}{\delta^E_\kappa}}
		\hn{\phi}{\omega^E}{0}
		< \tfrac{r}{2} (1 - \rho).
	\]
	Similarly, we conclude from \refer{lem:Naehe_expFunc-id_C1} and \ref{est:1-0-norm_f-0-norm_spezielles-f}
	that for each $\kappa \in \atlasA$,
	\[
		\label{est:simExp_1-1-norm}
		\tag{\ensuremath{\ast}}
		\hn{\ExpIdCWfak{V_\kappa}{\delta^E_\kappa}(\phi_\kappa)}{1_{V_\kappa}}{1}
		\leq \BndSndAblRex[\RMetLok{g}{\kappa}]{V_\kappa}{\delta^E_\kappa} \hn{\phi_\kappa}{1_{V_\kappa}}{0}
			+  \hn{\phi_\kappa}{1_{V_\kappa}}{1}
		\leq \tfrac{\BndSndAblRex[\RMetLok{g}{\kappa}]{V_\kappa}{\delta^E_\kappa}}%
			{\BndSndAblRex[\RMetLok{g}{\kappa}]{V_\kappa}{\delta^E_\kappa} + 1} \hn{\phi_\kappa}{\omega^E_\kappa}{0}
			+ \hn{\phi_\kappa}{1_{V_\kappa}}{1}
		.
	\]
	Hence we see using $\hn{\phi_\kappa}{\omega^E_\kappa}{0}, \hn{\phi_\kappa}{1_{V_\kappa}}{1} < \tfrac{\rho}{2}$ and
	the last two estimates that we can apply $\InvIdWeights{W}{\WeightsAtl{\GewFunk}{\atlasB}}$
	to $\ExpIdCWprod{V}{\delta^E}(\phi)$.
	Further, using the elementary estimate \cite[(4.2.5.1)]{arxiv_1006.5580v3},
	\refer{lem:Naehe_expVF-id_C0} and \ref{est:simExp_1-1-norm}
	we see that for $f \in \GewFunk$ and $\kappa \in \atlasA$,
	\[
		\label{est:simInversionExp_w0-norm}
		\tag{\ensuremath{\ast\ast}}
		\hn{\InvIdWeights{W_\kappa}{\WeightsLok{\GewFunk}{\kappa}}
					(\ExpIdCWfak{V_\kappa}{\delta^E_\kappa}(\phi_\kappa)) }{f_\kappa}{0}
		\leq
		\tfrac{\hn{\ExpIdCWfak{V_\kappa}{\delta^E_\kappa}(\phi_\kappa)}{f_\kappa}{0}}%
		{1 - \hn{\ExpIdCWfak{V}{\delta^E_\kappa}(\phi_\kappa)}{1_{W_\kappa}}{1} }
		< \tfrac{\BndFstAblRex[\RMetLok{g}{\kappa}]{V_\kappa}{\delta^E_\kappa}}%
				{1 - (\hn{\phi_\kappa}{\omega^E_\kappa}{0} + \hn{\phi_\kappa}{1_{V_\kappa}}{1})}
		\hn{\phi_\kappa}{f_\kappa}{0}.
	\]
	From this, we conclude with \ref{est:adjustWeights_exp-log},
	using $\hn{\phi_\kappa}{\omega^E_\kappa}{0} < \tfrac{\rho (1 - \rho)}{2}$
	and $\hn{\phi_\kappa}{1_{V_\kappa}}{1} < \tfrac{\rho}{2}$, that
	\[
		\hn{\InvIdWeights{W_\kappa}{\WeightsLok{\GewFunk}{\kappa}}
							(\ExpIdCWfak{V_\kappa}{\delta^E_\kappa}(\phi_\kappa)) }{\omega^L_\kappa}{0}
		\leq \tfrac{\BndFstAblRex[\RMetLok{g}{\kappa}]{V_\kappa}{\delta^E_\kappa}}%
				{(1 - \rho)\BndFstAblRex[\RMetLok{g}{\kappa}]{V_\kappa}{\delta^E_\kappa}}
			\hn{\phi_\kappa}{\omega^E_\kappa}{0}
		< \tfrac{\rho}{2}.
	\]
	Since $\omega^L$ is adjusting to $\delta^L$, we see from this estimate using
	\ref{incl:1-Kugel_f0-norm_sub_CFof} that we can apply $\LogIdCWprod{W}{\delta^L}$ to $\phi$.
	Another application of \refer{lem:Norm_LogFunc_C0} and \ref{est:simInversionExp_w0-norm}
	shows that
	\begin{multline*}
		\hn{\LogIdCWfak{W_\kappa}{\delta^L_\kappa}
			(\InvIdWeights{W_\kappa}{\WeightsLok{\GewFunk}{\kappa}}
			(\ExpIdCWfak{V_\kappa}{\delta^E_\kappa}(\phi_\kappa))) }{\omega^L_\kappa}{0}
		\leq \BndFstAblRlog[\RMetLok{g}{\kappa}]{W_\kappa}{\delta^L_\kappa}
			\hn{\InvIdWeights{W_\kappa}{\WeightsLok{\GewFunk}{\kappa}}
				(\ExpIdCWfak{V_\kappa}{\delta^E_\kappa}(\phi_\kappa)) }{\omega^L_\kappa}{0}
		\\
		< \frac{\BndFstAblRlog[\RMetLok{g}{\kappa}]{W_\kappa}{\delta^L_\kappa} \BndFstAblRex[\RMetLok{g}{\kappa}]{V_\kappa}{\delta^E_\kappa}}%
			{1 - (\hn{\phi_\kappa}{\omega^E_\kappa}{0} + \hn{\phi_\kappa}{1_{V_\kappa}}{1})}
				\hn{\phi_\kappa}{\omega^L_\kappa}{0}
		\leq
		\frac{\hn{\phi_\kappa}{\omega^E_\kappa}{0}}%
			{1 - (\hn{\phi_\kappa}{\omega^E_\kappa}{0} + \hn{\phi_\kappa}{1_{V_\kappa}}{1})}
	\end{multline*}
	So we derive the assertion on the containment in $R_\rho$, and that \ref{est:log0-normInv} holds.
	To prove \ref{id:Lokale_Darstellung_Inversion_durch_glatte_Abb},
	let $\kappa \in \atlasA$ and $p \in W_\kappa$. Then
	\begin{multline*}
		\LogIdCWfak[\RMetLok{g}{\kappa}]{W_\kappa}{\delta_\kappa}
				( \InvIdWeights{W_\kappa}{\WeightsLok{\GewFunk}{\kappa}}
					(\ExpIdCWfak{V_\kappa}{\delta_\kappa}(\VectFLok{X}{\kappa})))(p)
		= \RlogVR{\RMetLok{g}{\kappa}}(p, \InvIdWeights{W_\kappa}{\WeightsLok{\GewFunk}{\kappa}}
						(\ExpIdCWfak{V_\kappa}{\delta_\kappa}(\VectFLok{X}{\kappa}))(p) + p)
		\\
		= \RlogVR{\RMetLok{g}{\kappa}}(p, \rest{(\ExpIdCWfak{V_\kappa}{\delta_\kappa}(\VectFLok{X}{\kappa}) + \id{V_\kappa})^{-1}}{W_\kappa}(p))
		= \RlogVR{\RMetLok{g}{\kappa}}(p, \rest{(\Rexp{\RMetLok{g}{\kappa}} \circ (\id{V_\kappa}, \VectFLok{X}{\kappa}))^{-1}}{W_\kappa}(p)).
	\end{multline*}
	This shows that \ref{id:Lokale_Darstellung_Inversion_durch_glatte_Abb} holds.
\end{proof}

\subsection{Construction of weights on manifolds}
\label{susec:Construction_weights_manifolds}

We first define the terms saturated resp. adjusted sets of weights,
and then show that such weight sets exist.
\begin{defi}[Saturated and adjusted sets of weights]\label{def:saturated_adjusted-weights}
	Let $d \in \N^*$, $M$ a $d$-dimensional manifold, $\GewFunk \sub \cl{\R}^M$,
	$\atlasA = \sset{\kappa : \widetilde{U}_\kappa \to U_\kappa}$ an atlas for $M$
	and $\delta_\kappa > 0$ for each $\kappa \in \atlasA$.
	\begin{assertions}
		\item
		We call $\omega : M \to \R$ \emph{adjusted to $(\atlasA , \famiAtl{\delta_\kappa}{\atlasA})$} if
		there exists $K > 0$ such that $\weightAtl{K \cdot \omega}{\atlasA}$ is an adjusting weight for $\famiAtl{\delta_\kappa}{\atlasA}$.
		We call $\GewFunk$ \emph{adjusted to $(\atlasA , \famiAtl{\delta_\kappa}{\atlasA})$}
		if there exists $\omega \in \GewFunk$ that is adjusted to this pair.

		\item\label{ass1:chart_change_Abl-Mult}
		Let $\atlasA_1$ and $\atlasA_2$ be atlases for $M$.
		We say $\GewFunk$ is \emph{saturated with respect to $(\atlasA_1, \atlasA_2)$} if
		\ref{cond:est_sim-multiplier_weights} is satisfied for
		$\WeightsAtl{\GewFunk}{ \atCap{\atlasA_1}{\atlasA_2} }$ and
		$
			\fami{\rest{\FAbl{(\kappa \circ \phi^{-1})}}{\phi(\widetilde{U}_\kappa \cap \widetilde{U}_\phi)} }{(\kappa, \phi)}{ \atProd{\atlasA_1}{ \atlasA_2 }}
		$.

		\item\label{ass1:log_exp_Abl-Mult}
		Let $g$ be a Riemannian metric on $M$, and for each $\kappa \in \atlasA$ let
		$\widetilde{V}_\kappa$ be a relatively compact set with $\cl{\widetilde{V}_\kappa} \sub \widetilde{U}_\kappa$
		such that $\famiAtl{\widetilde{V}_\kappa}{\atlasA}$ is a cover of $M$,
		$\delta^E_\kappa \in ]0, \grenzExp[\RMetLok{g}{\kappa}]{V_\kappa}{U_\kappa}[$,
		$\delta^L_\kappa \in ]0, \grenzLog[\RMetLok{g}{\kappa}]{V_\kappa}{U_\kappa}[$
		(where $V_\kappa \ndef \kappa(\widetilde{V}_\kappa)$)
		and $\atlasB \ndef \set{\rest[V_\kappa]{\kappa}{\widetilde{V}_\kappa}}{\kappa \in \atlasA}$.
		We say $\GewFunk$ is \emph{saturated with respect to $(\atlasA, \atlasB, g)$} if
		$\WeightsAtl{\GewFunk}{\atlasB}$ satisfies \ref{cond:est_SP-Abb_weights}
		for $\famiAtl{ \ExpMinusLok{V_\kappa}{\delta^E_\kappa}{\RMetLok{g}{\kappa}} }{\atlasA}$
		and $\famiAtl{ \LogPlusVR{V_\kappa}{\delta^L_\kappa}{\RMetLok{g}{\kappa}} }{\atlasA}$
		respectively.%
	\end{assertions}
	If both \ref{ass1:chart_change_Abl-Mult} and \ref{ass1:log_exp_Abl-Mult} hold,
	we call $\GewFunk$ \emph{saturated with respect to $((\atlasA_1, \atlasA_2), (\atlasA, \atlasB, g))$}.
	Occasionally, we may just say that $\GewFunk$ is adjusted or saturated.
\end{defi}

\paragraph{Construction of adjusted weights}
We show that for a locally finite atlas whose chart domains are relatively compact, adjusted weights exist.

\begin{lem}[Construction of an adjusted weight]%
\label{lem:construction_adjusted_weight_MFK}
	Let $d \in \N^*$, $M$ be a $d$-dimensional manifold and
	$\atlasA = \sset{\kappa : \widetilde{U}_\kappa \to U_\kappa}$ a locally finite atlas for $M$.
	For each $\kappa \in \atlasA$, let $\eps_\kappa > 0$
	and $\widetilde{V}_\kappa \sub \widetilde{U}_\kappa$
	be an open, nonempty, relatively compact set
	such that $\famiAtl{\widetilde{V}_\kappa}{U}$ is a cover of $M$.
	Then there exists a weight $\omega : M \to \R$ adjusted to $(\atlasB, \famiAtl{\eps_\kappa}{\atlasA})$,
	where $\atlasB \ndef \set{\rest[\kappa(\widetilde{V}_\kappa)]{\kappa}{\widetilde{V}_\kappa} }{\kappa \in \atlasA}$.
\end{lem}
\begin{proof}
	For each $\kappa \in \atlasA$, let $f_\kappa : M \to \R$ be a function
	such that $\supp{f_\kappa} \sub \widetilde{U}_\kappa$, $\sup_{x \in M} \abs{f_\kappa}(x) < \infty$
	and $\inf_{x \in \widetilde{V}_\kappa} \abs{f_\kappa}(x) \geq \max(\tfrac{1}{\eps_\kappa}, 1)$.
	For $x \in M$, we set
	\[
		\omega(x) \ndef \max_{\kappa \in \atlasA} \abs{f_\kappa(x)};
	\]
	note that this definition is possible because each $x \in M$ is only contained in finitely many sets
	$\widetilde{U}_\kappa$.
	Then
	$
	\abs{\omega}(x)
	\geq \abs{f_\kappa}(x)
	\geq \max(\tfrac{1}{\eps_\kappa}, 1)
	$
	for each $\kappa \in \atlasA$ and $x \in \widetilde{V}_\kappa$. Further, since each $\widetilde{V}_\kappa$ is relatively compact,
	it has nonempty intersections with only finitely many sets $\set{\widetilde{U}_\kappa}{\kappa \in \atlasA}$.
	That implies that $\sup_{x \in \widetilde{V}_\kappa} \abs{\omega}(x) < \infty$.
	Hence $\weightAtl{\omega}{\atlasB}$
	is an adjusting weight for $\famiAtl{\eps_\kappa}{\atlasA}$.
\end{proof}
\paragraph{Saturating weights}
We not only show that saturated weight sets exist,
but moreover that each set of weights has a \enquote{\minSatExt}.
We first prove a variation of this assertion for a single weight.
\begin{lem}\label{lem:extension_weights_Multiplier}
	Let $d \in \N^*$, $M$ be a $d$-dimensional manifold,
	$\famiAtl{\widetilde{U}_\kappa}{\atlasA}$ a locally finite cover of $M$,
	$I$ a nonempty set, $f : M \to \cl{\R}$ and $\fami{B_{\kappa, i}}{(\kappa, i)}{\atlasA \times I}$
	a family of nonnegative real numbers.
	Then there exists a set $\extMult{f}{B} \sub \cl{\R}^M$ such that
	\[
		(\forall x \in M)(\exists V \in \neigh{x})(\forall g \in \extMult{f}{B})(\exists K > 0)
		\, \abs{\rest{g}{V}} \leq K \cdot \abs{\rest{f}{V}}
	\]
	and
	\begin{equation}
		\label{multMFK}
		(\forall i \in I)(\exists g \in \extMult{f}{B})(\forall \kappa \in \atlasA)
		\,B_{\kappa, i} \cdot \left\abs{\rest{f}{\widetilde{U}_\kappa}\right} \leq \left\abs{\rest{g}{\widetilde{U}_\kappa}\right}.
	\end{equation}
	The set $\extMult{f}{B}$ is minimal in the sense that
	for any $\mathcal{H} \sub \cl{\R}^M$ that also satisfies \ref{multMFK}, we have
	\begin{equation}
		\label{minimalMult}
		(\forall g \in \extMult{f}{B}) (\exists h \in \mathcal{H})
		\,\abs{g} \leq \abs{h}.
	\end{equation}
	
\end{lem}
\begin{proof}
	Let $i \in I$ and $x \in M$. Then we define
	\[
		g_i(x) \ndef \max \set{B_{\kappa, i} \cdot f(x)}{ \kappa \in \atlasA, x \in \widetilde{U}_\kappa }.
	\]
	This definition makes sense since $\famiAtl{\widetilde{U}_\kappa}{U}$ is locally finite.
	In particular, for each $x \in M$ there exist $\kappa_1, \dotsc, \kappa_n \in \atlasA$
	and $V \in \neigh{x}$ such that for $\kappa \in \atlasA$,
	\[
		\widetilde{U}_\kappa \cap V \neq \emptyset
		\iff
		\kappa \in \sset{\kappa_1, \dotsc, \kappa_n}.
	\]
	Hence
	\[
		\abs{\rest{g_i}{V}} \leq \max(B_{\kappa_1, i}, \dotsc, B_{\kappa_n, i} ) \cdot \abs{\rest{f}{V}}.
	\]
	Further, for $\hat{\kappa} \in \atlasA$ such that $x \in \widetilde{U}_{\hat{\kappa}}$, we have
	\[
		B_{\hat{\kappa}, i} \cdot \abs{f(x)}
		\leq \max \set{B_{\kappa, i}}{ \kappa \in \atlasA, x \in \widetilde{U}_\kappa } \cdot \abs{f(x)}
		= \abs{g_i(x)}.
	\]
	So the set
	\[
		\extMult{f}{B} \ndef \set{g_i}{i \in I},
	\]
	has the first two properties. To prove the minimality,
	let $\mathcal{H} \sub \cl{\R}^M$ satisfying \ref{multMFK}.
	Then for each $i \in I$, there exists $h \in \mathcal{H}$ such that
	$
		B_{\kappa, i} \cdot \left\abs{\rest{f}{\widetilde{U}_\kappa}\right} \leq \left\abs{\rest{h}{\widetilde{U}_\kappa}\right}
	$
	for all $\kappa \in \atlasA$. So for $x \in M$, we have
	\[
		(\forall \kappa \in \atlasA: x \in \widetilde{U}_\kappa)\, B_{\kappa, i} \abs{f(x)} \leq \abs{h(x)}.
	\]
	Hence
	\[
		\abs{g_i(x)} = \max \set{B_{\kappa, i} \cdot \abs{f(x)}}{ \kappa \in \atlasA, x \in \widetilde{U}_\kappa }
		\leq \abs{h(x)},
	\]
	which finishes the proof.
\end{proof}

\begin{bem}
	In the last lemma, we proved that
	$\abs{\rest{g}{V}} \leq K \cdot \abs{\rest{f}{V}}$
	for every neighborhood $V$ that has nonempty intersection with only finitely many cover sets.
\end{bem}
Before we show that each weight set has a minimal saturated superset, we make the following definition.
\begin{defi}%
	Let $M$ be a topological space and $f, g : M \to \cl{\R}$.
	We call $g$ \emph{locally $f$-bounded} if
	\[
		(\forall x \in M)(\exists U \in \neigh{x}, K > 0)
		\, \abs{\rest{g}{U}} \leq K \cdot \abs{\rest{f}{U}}.
	\]
	Let $\cW_1, \cW_2 \sub \cl{\R}^M$. We call $\cW_2$ \emph{locally $\cW_1$-bounded}
	if for all $g \in \cW_2$ there exists $f \in \cW_1$ such that $g$ is locally $f$-bounded.
	As usual, we call $f$ \emph{locally bounded} if it us locally $1_M$-bounded.%
\end{defi}
In the next lemma, we need the definition of the maximal extension of weights,
see \refer{def:res_prods-max_weights}.
\begin{lem}[{\minSatExt[M]}]
\label{lem:Erzeugung_Gewichtsmengen_MFK}%
	Let $d \in \N^*$, $(M, g)$ a $d$-dimensional Riemannian manifold,
	$\GewFunk \sub \cl{\R}^M$
	and $\atlasA = \sset{\kappa : \widetilde{U}_\kappa \to U_\kappa}$, $\widetilde{\atlasA}$ locally finite atlases for $M$.
	For each $\kappa \in \atlasA$, let $\widetilde{V}_\kappa$ a relatively compact set such that
	$\cl{\widetilde{V}_\kappa} \sub \widetilde{U}_\kappa$ and
	$\atlasB \ndef \set{\rest[V_\kappa]{\kappa}{\widetilde{V}_\kappa}}{\kappa \in \atlasA}$
	is an atlas for $M$ (here, $V_\kappa \ndef \kappa(\widetilde{V}_\kappa)$ for $\kappa \in \atlasA$),
	$\delta^E_\kappa \in ]0, \grenzExp[\RMetLok{g}{\kappa}]{V_\kappa}{U_\kappa}[$
	and $\delta^L_\kappa \in ]0, \grenzLog[\RMetLok{g}{\kappa}]{V_\kappa}{U_\kappa}[$.
	Then there exists $\GewFunk^e \sub \cl{\R}^M$ that is locally $\GewFunk$-bounded
	and saturated w.r.t. $((\atlasA, \atlasB, g), (\atlasA, \widetilde{\atlasA}))$.

	The set $\GewFunk^e$ is minimal in the sense that
	for any $\mathcal{G} \sub \cl{\R}^M$ that is saturated to the same data and contains $\GewFunk$, we have
	$\GewFunk^e \sub \ExtWeights{\mathcal{G}}$. We call $\GewFunk^e$ a \emph{\minSatExt} of $\GewFunk$.%
\end{lem}
\begin{proof}
	We define the following three families:
	\begin{gather*}
		B^1 : \atlasA \times \N^\ast \to [0, \infty[
		: (\kappa, \ell) \mapsto \hn{ \ExpMinusLok{V_\kappa}{\delta^E_\kappa}{\RMetLok{g}{\kappa}} }%
		{1_{V_\kappa \times \Ball{0}{\delta^E_\kappa}}}{\ell},
		\\
		B^2 : \atlasA \times \N^\ast \to [0, \infty[
		: (\kappa, \ell) \mapsto \hn{ \LogPlusVR{V_\kappa}{\delta^L_\kappa}{\RMetLok{g}{\kappa}} }%
		{1_{V_\kappa \times \Ball{0}{\delta^L_\kappa}}}{\ell},
		\\
		B^3 : \atProd{\atlasA}{\widetilde{\atlasA}} \times \N \to [0, \infty[
		: ((\kappa, \phi), \ell) \mapsto
		\hn{ \rest{\FAbl{(\kappa \circ \phi^{-1})}}{\phi(\widetilde{U}_\kappa \cap \widetilde{U}_\phi)} }%
		{1_{\phi(\widetilde{U}_\kappa \cap \widetilde{U}_\phi)}}{\ell}.
	\end{gather*}
	We inductively define $\GewFunk_0 \ndef \GewFunk$, and
	if $\GewFunk_k$ is defined for $k \in \N$, we set
	\[
		\GewFunk_{k + 1}
		\ndef
		\bigcup_{f \in \GewFunk_k}\extMult{f}{B^1}
		\cup
		\bigcup_{f \in \GewFunk_k}\extMult{f}{B^2}
		\cup
		\bigcup_{f \in \GewFunk_k}\extMult{f}{B^3}
		;
	\]
	we defined $\extMult{f}{B^i}$ in \refer{lem:extension_weights_Multiplier}.
	Finally, we set $\GewFunk^e \ndef \bigcup_{k \in \N} \GewFunk_k$.
	Since we can show with an easy induction (using \refer{lem:extension_weights_Multiplier}, of course)
	that each $\GewFunk_k$ is locally $\GewFunk$-bounded, so is $\GewFunk^e$.
	Finally, we see with another application of \refer{lem:extension_weights_Multiplier}
	that \ref{ass1:log_exp_Abl-Mult} and \ref{ass1:chart_change_Abl-Mult} in \refer{def:saturated_adjusted-weights} are satisfied.

	We prove the minimality condition by induction. More precisely, we prove that
	$\GewFunk_k \sub \ExtWeights{\mathcal{G}}$ for all $k \in \N$.
	The case $k = 0$ is satisfied by the assumptions on $\mathcal{G}$.
	Suppose it holds for $k \in \N$, and let $f \in \GewFunk_k$. Because $\mathcal{G}$ is saturated,
	it satisfies \ref{multMFK} for $f$ and the $B \in \set{B^j}{j=1,2,3}$.
	Hence we derive from \ref{minimalMult} that $\extMult{f}{B^j} \sub \ExtWeights{\mathcal{G}}$ for $j \in \sset{1,2,3}$,
	so obviously $\GewFunk_{k + 1} \sub \ExtWeights{\mathcal{G}}$.
\end{proof}

\section{Diffeomorphisms on Riemannian manifolds}

We construct \emph{weighted diffeomorphisms} on Riemannian manifolds,
and turn them into a Lie group that is modelled on weighted vector fields.
In order to do this, we prove a criterion when the composition of the exponential function
with a vector field is a diffeomorphism.
Then, we can use the local group operations treated in \refer{lem:Composition_and_Inversion_smooth-patched_version}
to construct Lie group structures from local data.
We state the main result concerning these Lie groups in \refer{satz:Existenz_Liegruppe_Diffeos_MFK}.
Finally, we compare these Lie groups with other well-known Lie groups of diffeomorphisms.

\subsection{Generating diffeomorphisms from vector fields}
\label{susec:Diffeos_from_VecFields}

In \refer{lem:Exp-VF_Diffeo--lokal}, we established under which conditions
the map $\Rexp{\RMetLok{g}{\kappa}} \circ \VectFLok{X}{\kappa}$ is a diffeomorphism,
where $\kappa$ is a chart and $\VectFLok{X}{\kappa}$ a localized vector field.
We show that similar assumptions also allow the global behavior of $\Rexp{g} \circ X$ to be controlled.
\newcommand{\DiVe}[1]{\ensuremath{\phi_{#1}}}
\begin{prop}\label{prop:exVF-lokDiffeo-sur-estPreIm}
	Let $(M, g)$ be a Riemannian manifold
	and $\atlasA = \sset{\kappa : \widetilde{U}_\kappa \to U_\kappa }$ an atlas for $M$.
	For each $\kappa \in \atlasA$, let $r_\kappa > 0$ such that $\clBall{0}{r_\kappa} \sub U_\kappa$ and
	$\set{\kappa^{-1}(\Ball{0}{r_\kappa}) }{\kappa \in \atlasA}$ is a cover of $M$,
	$\nu_\kappa \in ]0, \grenzExp[\RMetLok{g}{\kappa}]{ \Ball{0}{r_\kappa} }{ U_\kappa } [$
	and $\eps_\kappa \in ]0, \frac{1}{2}[$.
	Further let $k \in \cl{\N}$ with $k \geq 1$
	and $X \in \VecFields[k]{M}$ such that for each $\kappa \in \atlasA$,
	\[
		\hn{\VectFLok{X}{\kappa} }{ 1_{\Ball{0}{r_\kappa}} }{0}
		< \min\Bigl(\tfrac{\eps_\kappa r_\kappa}{2 \BndFstAblRex[\RMetLok{g}{\kappa}]{ \Ball{0}{r_\kappa} }{\nu_\kappa}}, \nu_\kappa,
			\tfrac{\eps_\kappa}{4 (\BndSndAblRex[\RMetLok{g}{\kappa}]{\Ball{0}{r_\kappa}}{\nu_\kappa}  + 1) }\Bigr),
	\]
	and
	$\hn{\VectFLok{X}{\kappa} }{ 1_{\Ball{0}{r_\kappa}} }{1} < \frac{\eps_\kappa}{4} $.
	Then the following assertions hold:
	\begin{assertions}
		\item\label{ass1:ExpVF_locDiff}
		We have that $\image{X} \sub \domExpMax{g}$, the map $\DiVe{X} \ndef \Rexp{g} \circ X$ maps each connected component of $M$ into itself,
		and for each $\kappa \in \atlasA$, $\rest{\kappa \circ \DiVe{X} \circ \kappa^{-1} }{ \Ball{0}{r_\kappa} }$
		is a $\ConDiff{}{}{k}$-diffeomorphism whose image contains $\Ball{0}{r_\kappa(1 - 2 \eps_\kappa)}$.

		\item\label{ass1:OnThePreimage_of_ExpVF}
		For each $y \in M$,
		$
				\# (\Rexp{g} \circ X)^{-1}(y) \leq \# \atlasA_y
		$,
		where $\atlasA_y \ndef \set{\kappa \in \atlasA}{y \in \widetilde{U}_\kappa}$.
	\end{assertions}
	In addition, assume that $M$ is connected, $\set{\widetilde{U}_\kappa}{\kappa \in \atlasA}$ is a locally finite cover of $M$
	and $\set{\kappa^{-1}( \Ball{0}{\widetilde{r_\kappa} } ) }{\kappa \in \atlasA}$ is also a cover of $M$,
	where each $\widetilde{r_\kappa} < (1 -  \eps_\kappa) r_\kappa$. Then
	\begin{assertions}[resume]
		\item\label{ass1:ExpVF_properMap}
		$\Rexp{g} \circ X$ is a proper map.
	\end{assertions}
	Assume that each $\widetilde{r_\kappa} < (1 -  2 \eps_\kappa) r_\kappa$. Then
	\begin{assertions}[resume]
		\item\label{ass1:ExpVF_covering_finiteSheets}
		$\Rexp{g} \circ X$ is a covering map with finitely many sheets.
	\end{assertions}
	Finally, assume that there exists a point in $M$ that is only contained in one $\widetilde{U}_\kappa$. Then
	\begin{assertions}[resume]
		\item\label{ass1:ExpVF_globDiffeo}
		$\Rexp{g} \circ X$ is a diffeomorphism.
	\end{assertions}
\end{prop}
\begin{proof}
	\ref{ass1:ExpVF_locDiff}
	Since $\hn{\VectFLok{X}{\kappa} }{ 1_{\Ball{0}{r_\kappa}} }{0} < \nu_\kappa$ for each $\kappa \in \atlasA$,
	we see with \refer{bem:Riemann-Geometrie_Karte} that
	\[
		\Tang{\kappa^{-1}}(( \id{U_\kappa}, \VectFLok{X}{\kappa}))(\Ball{0}{r_\kappa})
		\sub \Tang{\kappa^{-1}}(\domExpMax{\RMetLok{g}{\kappa}}) \sub \domExpMax{g}.
	\]
	It is obvious from the definition of $\Rexp{g}$ that $\DiVe{X}$ maps each connected component of $M$ into itself.
	By our assumptions, for each $\kappa \in \atlasA$ we can apply \refer{lem:Exp-VF_Diffeo--lokal} to
	the function $\VectFLok{X}{\kappa}$ and the exponential function $\Rexp{\RMetLok{g}{\kappa}}$
	to see that
	$\rest{\DiVe{\VectFLok{X}{\kappa}} }{ \Ball{0}{r_\kappa} }$
	is a $\ConDiff{}{}{k}$-diffeomorphism whose image contains $\Ball{0}{r_\kappa(1 - 2 \eps_\kappa)}$
	(here $\DiVe{\VectFLok{X}{\kappa}} \ndef \Rexp{\RMetLok{g}{\kappa}} \circ (\id{U_\kappa}, \VectFLok{X}{\kappa})$).
	Since
	$\rest{\kappa \circ \DiVe{X} \circ \kappa^{-1} }{ \Ball{0}{r_\kappa} }
		= \rest{\DiVe{\VectFLok{X}{\kappa}} }{ \Ball{0}{r_\kappa} }$
	by \refer{lem:SP_mitRexp-Vergleich_lokal_global}, the assertion holds.

	\ref{ass1:OnThePreimage_of_ExpVF}
	Let $y \in M$. For $\kappa \in \atlasA$,  we set $W_\kappa \ndef \kappa^{-1}(\Ball{0}{r_\kappa})$ and define
	\[
		\atlasA_{X, y} \ndef \set{\kappa \in \atlasA}{(\exists x \in W_\kappa) \, \DiVe{X}(x) = y}.
	\]
	Since the map $\rest{\DiVe{X}}{W_\kappa}$ is injective for each $\kappa \in \atlasA_{X, y}$,
	there exists at most one $x_\kappa \in W_\kappa$ with $\DiVe{X}(x_\kappa) = y$.
	Further $y = \DiVe{X}(x_\kappa) \in \DiVe{X}(W_\kappa) \sub \widetilde{U}_\kappa$
	(since $\hn{\VectFLok{X}{\kappa}}{1_{\Ball{0}{r_\kappa}}}{0} < \nu_\kappa$), hence $\atlasA_{X, y} \sub \atlasA_y$.
	The map $\atlasA_{X, y} \to \DiVe{X}^{-1}(y) : \kappa \mapsto x_\kappa$ is surjective because
	$\set{W_\kappa}{\kappa \in \atlasA}$ is a cover of $M$, so we derive the assertion.

	\ref{ass1:ExpVF_properMap}
	Let $K \sub M$ be a compact set. Since $\set{\widetilde{U}_\kappa}{\kappa \in \atlasA}$ is a locally finite cover,
	using a straightforward compactness argument we can show that there exists
	a finite set $F \sub \atlasA$ such that for $\kappa \in \atlasA$ the equivalence
	\[
		\widetilde{U}_\kappa \cap K \neq \emptyset
		\iff
		\kappa \in F
	\]
	holds. We then define
	\[
		\widetilde{K} \ndef
		\bigcup_{\kappa \in F}
		\kappa^{-1}\bigl( ( \kappa(K) \cap \clBall{0}{\widetilde{r_\kappa} + \frac{\eps_\kappa r_\kappa}{2} } )+ \clBall{0}{ \frac{\eps_\kappa r_\kappa}{2} } \bigr).
	\]
	This is a compact set and we prove that it contains $\DiVe{X}^{-1}(K)$.
	To this end, let $y \in \DiVe{X}^{-1}(K)$.
	Then there exists $\kappa \in \atlasA$ such that $y \in \kappa^{-1}(\Ball{0}{ \widetilde{r_\kappa} })$, and by our assumptions on $X$,
	we have that $\DiVe{X}(y) \in \widetilde{U}_\kappa$, hence $\kappa \in F$.
	Further, using \refer{lem:Naehe_expVF-id_C0} we get
	\begin{equation*}
		\norm{\kappa(\DiVe{X}(y)) - \kappa(y)}
		= \norm{\DiVe{\VectFLok{X}{\kappa}}(\kappa(y)) - \kappa(y)}
		\leq \frac{\eps_\kappa r_\kappa}{2}.
	\end{equation*}
	This implies that
	\[
		\norm{\kappa(\DiVe{X}(y)) }
		\leq \norm{\kappa(y) } + \norm{\kappa(\DiVe{X}(y) )  - \kappa(y) }
		< \widetilde{r_\kappa} + \frac{\eps_\kappa r_\kappa}{2}.
	\]
	So we see that
	\[
		\kappa(y) = \kappa(\DiVe{X}(y) ) +  \kappa(y) - \kappa(\DiVe{X}(y))
		\in \kappa(K) \cap \clBall{0}{\widetilde{r_\kappa} + \frac{\eps_\kappa r_\kappa}{2} } + \clBall{0}{ \frac{\eps_\kappa r_\kappa}{2} },
	\]
	which shows that $y \in \widetilde{K}$.

	\ref{ass1:ExpVF_covering_finiteSheets}
	$\DiVe{X}$ is surjective since the image of $\rest{\DiVe{X}}{\kappa^{-1}(\Ball{0}{r_\kappa}) }$
	contains $\kappa^{-1}( \Ball{0}{ (1 -  2 \eps_\kappa) r_\kappa} )$ by \ref{ass1:ExpVF_locDiff}, and these sets cover $M$ by assumption.
	Since we also proved in \ref{ass1:ExpVF_locDiff} that $\DiVe{X}$ is a local homeomorphism
	and is a proper map by \ref{ass1:ExpVF_properMap},
	we can use \cite[Theorem 4.22]{MR648106} to see that it is a covering map.

	\ref{ass1:ExpVF_globDiffeo}
	We showed in \ref{ass1:ExpVF_locDiff} that $\DiVe{X}$ is a local diffeomorphism,
	and by \ref{ass1:ExpVF_covering_finiteSheets} it is a covering map.
	We see with the hypothesis of \ref{ass1:ExpVF_globDiffeo}
	and the assertion of \ref{ass1:OnThePreimage_of_ExpVF} that it has only one sheet,
	so it is a bijection and hence a diffeomorphism.
\end{proof}

\subsection{Lie groups of weighted diffeomorphisms}

We show that on a Riemannian manifold,
for each locally finite, \emph{adapted} atlas $\atlasA$ (we will introduce this terminology soon)
and each set $\GewFunk$ of weights containing $1_M$,
there exists a Lie group of \emph{weighted diffeomorphisms}.
The Lie group is modelled on the space $\CFM{M}{\GewFunk^e}{\infty}{\atlasA}$
of weighted vector fields, where $\GewFunk^e$ is a minimal saturated extension of $\GewFunk \cup \sset{\omega}$,
where $\omega$ is a suitable adjusted weight.

We then examine under which conditions the compactly supported diffeomorphisms
are a subset of the weighted diffeomorphisms,
and see that if the manifold is $\R^d$ with the scalar product,
the weighted diffeomorphisms constructed here are the same as in \cite[§4]{MR2952176}.

\subsubsection{Lie groups modelled on weighted vector fields}

We first transfer the results of \refer{lem:Composition_and_Inversion_smooth-patched_version}
to weighted vector fields.
For the inversion, before the introduction of \refer{prop:exVF-lokDiffeo-sur-estPreIm}
this was not possible since we had only developed criteria for local invertibility.
Further, we use these results to construct a Lie group modelled on weighted vector fields.
Note that we assume the existence of suitable weights, but even with \refer{lem:construction_adjusted_weight_MFK}
it is not clear that adjusting weights that satisfy \ref{est:adjustWeights_exp-log} exist.

Before we begin, we make the following definition.
\begin{defi}\label{def:adapted_atlas}
	Let $d \in \N^*$, $M$ a $d$-dimensional manifold,
	$\atlasA = \sset{\kappa : \widetilde{U}_\kappa \to U_\kappa}$ an atlas for $M$,
	$\famiAtl{r_\kappa}{\atlasA}$, $\famiAtl{\eps_\kappa}{\atlasA}$ families of positive real numbers and $R > 0$.

	We call $\atlasA$ \emph{adapted to $(\famiAtl{r_\kappa}{\atlasA}, \famiAtl{\eps_\kappa}{\atlasA}, R)$} if
	$\clBall{0}{r_\kappa + R} \sub U_\kappa$ for all $\kappa \in \atlasA$,
	$\famiAtl{\kappa^{-1}(\Ball{0}{r_\kappa})}{\atlasA}$ is a cover of $M$ and
	$r_\kappa < \bigl(\tfrac{1}{ 2\eps_\kappa} - 1\bigr) R$ for all $\kappa \in \atlasA$.
	Note that this implies that each $\eps_\kappa <\tfrac{1}{2}$.
	Sometimes, we may call such an atlas $\atlasA$ just adapted.
\end{defi}

\begin{bem}
	Note that on a manifold with a countable base every atlas is adapted,
	see \cite[Theorem 3.3]{MR1931083}.
\end{bem}

\begin{lem}\label{lem:Liegrupp_DiffeosMFK_geeigneteGewichte}
	Let $d \in \N^*$, $(M, g)$ a $d$-dimensional connected Riemannian manifold,
	$\atlasA = \sset{\kappa : \widetilde{U}_\kappa \to U_\kappa}$ a locally finite atlas for $M$,
	$R > 0$ and $\famiAtl{r_\kappa}{\atlasA}$, $\famiAtl{\eps_\kappa}{\atlasA}$ families of positive real numbers
	such that $\atlasA$ is adapted to $(\famiAtl{r_\kappa}{\atlasA}, \famiAtl{\eps_\kappa}{\atlasA}, R)$
	and $\eps \ndef \inf_{\kappa \in \atlasA} \eps_\kappa > 0$.

	We then set $V_\kappa \ndef \Ball{0}{r_\kappa + R}$,
	$\atlasB \ndef \set{\rest[V_\kappa]{\kappa}{\kappa^{-1}(V_\kappa)}}{\kappa \in \atlasA}$
	and $\atlasC \ndef \set{\rest[\Ball{0}{r_\kappa}]{\kappa}{\kappa^{-1}(\Ball{0}{r_\kappa})}}{\kappa \in \atlasA}$.
	Further, for each $\kappa \in \atlasA$, let $\delta^E_\kappa \in]0, \grenzExp[\RMetLok{g}{\kappa}]{V_\kappa}{U_\kappa}[$
	and $\delta^L_\kappa \in]0, \grenzLog[\RMetLok{g}{\kappa}]{V_\kappa}{U_\kappa}[$.
	Let $\GewFunk \sub \cl{\R}^M$ contain weights $\omega^E$, $\omega^L$ such that $\weightAtl{\omega^E}{\atlasB}$
	is an adjusting weight for
	\begin{equation}
		\tag{\ensuremath{\dagger}}\label{adjusting_weight_to-local_group}
		\left\famiAtl{
			\min\Bigl(\delta^E_\kappa,
			\tfrac{\min(\eps_\kappa r_\kappa, 1)}{2 \BndFstAblRex[\RMetLok{g}{\kappa}]{ V_\kappa }{\delta^E_\kappa}},
			\tfrac{\eps_\kappa}{4 (\BndSndAblRex[\RMetLok{g}{\kappa}]{ V_\kappa }{\delta^E_\kappa}  + 1) }
			\Bigr)
		\right}{\atlasA};
	\end{equation}
	$\weightAtl{\omega^L}{\atlasB}$ is an adjusting weight for $\famiAtl{\delta^L_\kappa}{\atlasA}$,
	and \ref{est:adjustWeights_exp-log} is satisfied for $\omega^L$ and $\omega^E$.
	Further, assume that $\WeightsAtl{\GewFunk}{\atlasB}$ satisfies \ref{cond:est_SP-Abb_weights}
	for $\famiAtl{ \ExpMinusLok{V_\kappa}{\delta^E_\kappa}{\RMetLok{g}{\kappa}} }{\atlasA}$
	and $\famiAtl{ \LogPlusVR{V_\kappa}{\delta^L_\kappa}{\RMetLok{g}{\kappa}} }{\atlasA}$,
	respectively.
	\begin{assertions}
		\item\label{ass:lokale_Kompo_VF}
		Then the map
		\begin{equation*}
			C_{\VecFields{M}}
			:
			D_1^\atlasB \times D_2^\atlasB \to R^\atlasC
			:
			(X, Y) \mapsto
			\Rlog{g} \circ(\id{M}, (\Rexp{g} \circ X) \circ  (\Rexp{g} \circ Y))
		\end{equation*}
		is defined and smooth, where
		\begin{equation*}
			D_1^\atlasB \ndef \set{X \in \CcFM{M}{\infty}{\atlasB} }%
			{\hnM{X}{\omega^E}{0}{\atlasB} < \tfrac{1}{2}
			\text{ and }
			\hnM{X}{1_M}{1}{\atlasB} < \tfrac{1}{2}
			}
		\end{equation*}
		and
		\[
			D_2^\atlasB \ndef \set{X \in \CcFM{M}{\infty}{\atlasB} }%
			{\hnM{X}{\omega^E}{0}{\atlasB} < \min(\tfrac{1}{4}, R)}
		\]
		and
		\[
			R^\atlasC \ndef \set{X \in \CcFM{M}{\infty}{\atlasC} }{ \hnM{X}{\omega^L}{0}{\atlasC} < 1}.
		\]
	\end{assertions}
	Assume that there exists a point in $M$ that is contained in only one $\widetilde{U}_\kappa$.
	\begin{assertions}[resume]
		\item\label{ass:lokale_Inv_VF}
		Then for each $\rho \in ]0,1[$, the map
		\[
			I_{\VecFields{M}}
			:
			D_{\rho}^\atlasB \to R_\rho^\atlasC
			:
			X \mapsto \Rlog{g} \circ(\id{M}, (\Rexp{g} \circ X)^{-1} )
		\]
		is defined and smooth, where
		\begin{equation*}
			D_{\rho}^\atlasB \ndef \set{X \in \CcFM{M}{\infty}{\atlasB} }%
			{
			\hnM{X}{\omega^E}{0}{\atlasB} < \tfrac{(1 - \rho)\min(\rho, R)}{2}
			,\,
			\hnM{X}{1_M}{1}{\atlasB} < \min(\tfrac{\rho}{2}, \tfrac{\eps}{4})
			}
		\end{equation*}
		and
		\[
			R_\rho^\atlasC \ndef \set{ X \in \CcFM{M}{\infty}{\atlasC}}%
			{ \hnM{X}{\omega^L}{0}{\atlasC} < \tfrac{\min(R, \rho)}{2} }.
		\]
	\end{assertions}
	We set
	\begin{equation*}
		\mathcal{D}_D \ndef
		\set{\Rexp{g} \circ\, X }%
					{X \in D_1^\atlasB \cap D_2^\atlasB \cap D_{\rho}^\atlasB},
	\end{equation*}
	and assume that \ref{cond:est_sim-multiplier_weights} is satisfied for
	$\WeightsAtl{\GewFunk}{ \atCap{\atlasC}{\atlasB} }$ and
	$
		\fami{\rest{\FAbl{(\kappa \circ \phi^{-1})}}{\phi(\widetilde{U}_\kappa \cap \widetilde{U}_\phi)} }{(\kappa, \phi)}{ \atProd{\atlasB}{\atlasC}}
	$.
	\begin{assertions}[resume]
		\item\label{ass:Lie_group}
		Then there exists a Lie group structure on the subgroup of $\Diff{M}{}{}$
		generated by
		\[
			\mathcal{D}_D \cap \mathcal{D}_D^{-1}.
		\]
		The restriction of the map
		\[
			\mathcal{L}
			: \mathcal{D}_D \to \CcFM{M}{\infty}{\atlasB}
			: \phi \mapsto \Rlog{g} \circ (\id{M}, \phi)
		\]
		is a chart for this set.
	\end{assertions}
\end{lem}
\begin{proof}
	\ref{ass:lokale_Kompo_VF}
	Using \refer{lem:Composition_and_Inversion_smooth-patched_version}
	and \ref{id:Lokale_Darstellung_Komposition_durch_glatte_Abb} (together with \refer{lem:Log-exp_unter_Kartenwechsel}),
	we get the commutative diagram
	\[
		\xymatrix{
		{D_1^\atlasB \times D_2^\atlasB} \ar@{>->}[rr]^-{\embMcWatl{\atlasB} \times \embMcWatl{\atlasB}} \ar@{->}[d]|{C_{\VecFields{M}}} && {D_1 \times D_2} \ar[d]|{\mathbf{C}_E^L}
		\\
		{R^\atlasC} \ar@{>->}[rr]^-{\embMcWatl{\atlasC}} && {R}
		.
		}
	\]
	In particular,
	$\image{\mathbf{C}_E^L \circ (\embMcWatl{\atlasB} \times \embMcWatl{\atlasB})} \sub \image{\embMcWatl{\atlasC}}$,
	and the corestriction of $\mathbf{C}_E^L \circ (\embMcWatl{\atlasB} \times \embMcWatl{\atlasB})$
	to $\image{\embMcWatl{\atlasC}}$ is smooth
	by \cite[Prop.~A.1.12]{MR2952176}
	since we proved in \refer{lem:WeiVF_closed_image} the vector fields are a closed subset of the product.
	Since $\embMcWatl{\atlasC}$ is an embedding,
	this proves our assertion that $C_{\VecFields{M}}$ is defined and smooth.

	\ref{ass:lokale_Inv_VF}
	We know from \refer{prop:exVF-lokDiffeo-sur-estPreIm} that for all
	$X \in \CcFM{M}{\infty}{\atlasC}$ with
	$
		\hnM{X}{\omega}{0}{\atlasB} < 1$ and $\hnM{X}{1_M}{1}{\atlasB} < \tfrac{\eps}{4}
	$,
	the map $\Rexp{g} \circ \, X$ is a diffeomorphism.
	We can apply \ref{prop:exVF-lokDiffeo-sur-estPreIm}
	since $r_\kappa < (1 - 2\eps_\kappa)(r_\kappa + R)$ (that is shown with a short calculation),
	and using our assumptions on $\omega$ stated in \ref{adjusting_weight_to-local_group},
	together with \ref{est:1-0-norm_f-0-norm_spezielles-f}.
	\\
	The rest of the proof follows along the same lines as \ref{ass:lokale_Kompo_VF}.

	\ref{ass:Lie_group}
	We calculate using \refer{lem:Naehe_expVF-id_C0} and \ref{est:adjustWeights_exp-log}
	that for all $X \in  D_1^\atlasB \cap D_2^\atlasB \cap D_{\rho}^\atlasB$ and $\kappa \in \atlasA$, we have
	\[
		\hn{ \ExpIdCWfak{V_\kappa}{\delta_\kappa}(\VectFLok{X}{\kappa})}{\weightLok{\omega^L}{\kappa}}{0}
		\leq \hn{ \VectFLok{X}{\kappa}}{\weightLok{\omega^E}{\kappa}}{0}
		< 1.
	\]
	Since $\omega^L$ is adjusting for $\delta^L$, we know from this estimate that we can apply
	$\Rlog{g}$ to $(\id{M}, \Rexp{g} \circ X)$, so $\mathcal{L}$ is well-defined.

	At the next step, we show that $\mathcal{D}_D \cap \mathcal{D}_D^{-1}
	= \mathcal{D}_D \cap \mathcal{L}^{-1}(I_{\VecFields{M}}^{-1}( \mathcal{L}( \mathcal{D}_D) ))$.
	To this end, let $\phi \in \mathcal{D}_D \cap \mathcal{D}_D^{-1}$. Then there exists $\psi \in \mathcal{D}_D$
	such that $\phi^{-1} = \psi$, and $X, Y \in \mathcal{L}(\mathcal{D}_D)$ with $\phi = \Rexp{g} \circ X$, $\psi = \Rexp{g} \circ Y$.
	Then
	\[
		Y = \Rlog{g} \circ (\id{M}, \psi) = \Rlog{g} \circ (\id{M}, (\Rexp{g} \circ X)^{-1}) = I_{\VecFields{M}}(X).
	\]
	Hence $X \in I_{\VecFields{M}}^{-1}( \mathcal{L}( \mathcal{D}_D) )$
	(note that we used that $\mathcal{L}( \mathcal{D}_D) \sub D_{\rho}^\atlasB \sub R_\rho^\atlasC$),
	and $\phi \in \mathcal{L}^{-1}(I_{\VecFields{M}}^{-1}( \mathcal{L}( \mathcal{D}_D) ))$.
	On the other hand, if $\phi \in \mathcal{D}_D$ such that
	$\mathcal{L}( \mathcal{D}_D) \in I_{\VecFields{M}}^{-1}( \mathcal{L}( \mathcal{D}_D) )$,
	then there exists $X \in \mathcal{L}( \mathcal{D}_D)$
	with $X = I_{\VecFields{M}}( \mathcal{L}( \phi) ) = \Rlog{g} \circ (\id{M}, \phi^{-1})$.
	Hence $\phi^{-1} = \Rexp{g} \circ X \in \mathcal{D}_D$, so $\phi \in \mathcal{D}_D^{-1}$.

	We show that $\mathcal{L}(\mathcal{D}_D \cap \mathcal{D}_D^{-1})$ is open in $\CcFM{M}{\infty}{\atlasB}$.
	By the definition of adjusting weights, $\abs{\omega^E} \geq 1$.
	Hence we can apply \refer{cor:gewVF-verschiedeneAtlanten-Chart_Changes_Weights} to see that
	$ \CcFM{M}{\infty}{\atlasB} = \CcFM{M}{\infty}{\atlasC} $.
	Hence $\mathcal{L}( \mathcal{D}_D )$ is open in $\CcFM{M}{\infty}{\atlasC} $,
	and by \ref{ass:lokale_Inv_VF}, so is $I_{\VecFields{M}}^{-1}( \mathcal{L}( \mathcal{D}_D) )$ in $\CcFM{M}{\infty}{\atlasB} $.

	Since we proved in \ref{ass:lokale_Inv_VF} that $I_{\VecFields{M}}$ is smooth on $\mathcal{L}(\mathcal{D}_D \cap \mathcal{D}_D^{-1})$,
	the inversion map is smooth on $\mathcal{D}_D \cap \mathcal{D}_D^{-1}$,
	with respect to the manifold structure induced by $\mathcal{L}$.
	Since this set is symmetric and open, and we can deduce from the things we proved in \ref{ass:lokale_Kompo_VF}
	that the composition
	\[
		(\mathcal{D}_D \cap \mathcal{D}_D^{-1}) \times (\mathcal{D}_D \cap \mathcal{D}_D^{-1})
		\to \Rexp{g} \circ R^\atlasC \cap \Diff{M}{}{}
	\]
	is smooth, it is possible to apply the theorem about generation from local data
	\cite[La.~B.2.5]{MR2952176}
	to get the assertion.
\end{proof}
\paragraph{Restricting the domain of $\boldsymbol{\Rexp{g}}$}
We restrict the domain of the exponential function,
which allows us to show that adjusting weights satisfying \ref{est:adjustWeights_exp-log} exist.
In order to do this, we need the results of \refer{susec:RiemannExpLog_VR},
in particular \refer{lem:Estimates_Rex_Rlog_DomRlog}.

\begin{lem}\label{lem:adjusting_weight_log_from_weight_exp}
	Let  $d \in \N^*$, $(M, g)$ a $d$-dimensional Riemannian manifold,
	$\atlasA = \sset{\kappa : \widetilde{U}_\kappa \to U_\kappa}$ an atlas for $M$
	and $\sigma \in ]0, 1[$.
	Further, for each $\kappa \in \atlasA$ let $V_\kappa$ be a relatively compact set with $\cl{V_\kappa} \sub U_\kappa$,
	$\delta_\kappa \in]0, \RadExpFibInv[\RMetLok{g}{\kappa}]{V_\kappa}{\sigma} \QuotNorm[\RMetLok{g}{\kappa}]{V_\kappa}[$
	and $\omega : M \to \R$ be adjusted to
	$\famiAtl{\tfrac{(1 - \sigma)^2}{1 + \sigma} \delta_\kappa}{\atlasA}$
	such that $\abs{\omega} \geq \tfrac{1 + \sigma}{1 - \sigma}$.
	Then $\omega^L \ndef \tfrac{1 - \sigma}{1 + \sigma} \omega$ is adjusted to
	$\famiAtl{ (1 - \sigma) \delta_\kappa }{\atlasA}$,
	we have $ (1 - \sigma) \delta_\kappa < \grenzLog[\RMetLok{g}{\kappa}]{V_\kappa}{U_\kappa}$
	and the weights $\omega$,  $\omega^L$ satisfy \ref{est:adjustWeights_exp-log}.
\end{lem}
\begin{proof}
	Let $\kappa \in \atlasA$. Then we have that
	\[
		\abs{\weightLok{\omega^L}{\kappa}}
		= \tfrac{1 - \sigma}{1 + \sigma} \abs{\weightLok{\omega}{\kappa}}
		\geq \frac{1}{\tfrac{1 + \sigma}{1 - \sigma}} \frac{1}{\tfrac{(1 - \sigma)^2}{1 + \sigma} \delta_\kappa}
		= \frac{1}{(1 - \sigma) \delta_\kappa},
	\]
	hence $\omega^L$ is adjusted to $\famiAtl{ (1 - \sigma) \delta_\kappa }{\atlasA}$
	since we assumed that $\abs{\omega} \geq \tfrac{1 + \sigma}{1 - \sigma}$.
	Further, we know from \refer{lem:Estimates_Rex_Rlog_DomRlog}
	that $ (1 - \sigma) \delta_\kappa < \grenzLog[\RMetLok{g}{\kappa}]{V_\kappa}{U_\kappa}$,
	$\BndFstAblRex[\RMetLok{g}{\kappa}]{ V_\kappa }{\delta_\kappa} \leq 1 + \sigma$
	and $\BndFstAblRlog[\RMetLok{g}{\kappa}]{V_\kappa}{ (1 - \sigma) \delta_\kappa} \leq \tfrac{1}{1 - \sigma}$.
	Hence for $\kappa \in \atlasA$,
	\[
		\abs{\weightLok{\omega^L}{\kappa}}
		= \tfrac{1 - \sigma}{1 + \sigma} \abs{\weightLok{\omega}{\kappa}}
		\leq \tfrac{1}{\BndFstAblRex[\RMetLok{g}{\kappa}]{ V_\kappa }{\delta_\kappa} \BndFstAblRlog[\RMetLok{g}{\kappa}]{V_\kappa}{ (1 - \sigma) \delta_\kappa}}
		\abs{\weightLok{\omega}{\kappa}}.
	\]
	This finishes the proof.
\end{proof}
We are ready to prove the main result.
\begin{satz}\label{satz:Existenz_Liegruppe_Diffeos_MFK}
	Let  $d \in \N^*$, $(M, g)$ a $d$-dimensional Riemannian manifold,
	$\GewFunk \sub \cl{\R}^M$ with $1_M \in \GewFunk$
	and $\atlasA = \sset{\kappa : \widetilde{U}_\kappa \to U_\kappa}$ a locally finite atlas for $M$
	such that there exists a point in $M$ that is contained in only one $\widetilde{U}_\kappa$.
	Further, for each $\kappa \in \atlasA$ let $\eps_\kappa \in ]0, \tfrac{1}{2}[$ and $r_\kappa > 0$
	such that
	$\eps \ndef \inf_{\kappa \in \atlasA} \eps_\kappa > 0$ and
	$r \ndef \inf_{\kappa \in \atlasA} r_\kappa > 0$.
	Suppose that there exists $R > 0$ such that $\atlasA$ is adapted to
	$(\famiAtl{r_\kappa}{\atlasA}, \famiAtl{\eps_\kappa}{\atlasA}, R)$.

	Then there exists a subgroup $\Diff{M,g, \omega}{\atlasA, \atlasB}{\GewFunk}$ of $\Diff{M}{}{}$
	that is generated by $\mathcal{D}_D \cap \mathcal{D}_D^{-1}$,
	where
	\begin{equation*}
		\mathcal{D}_D \ndef
		\set{\Rexp{g} \circ\, X }%
					{X \in \CFM{M}{\GewFunk^e}{\infty}{\atlasA}, \hnM{X}{\omega}{0}{\atlasB},
								\hnM{X}{1_M}{1}{\atlasB} < \alpha }
	\end{equation*}
	with some suitable $\alpha > 0$,
	$\atlasB \ndef \set{\rest[\Ball{0}{r_\kappa + R}]{\kappa}{\kappa^{-1}(\Ball{0}{r_\kappa + R})}}{\kappa \in \atlasA}$
	and $\omega \in \GewFunk^e$ that is adjusted to $(\atlasB, \famiAtl{\widetilde{\delta}_\kappa}{\atlasA})$,
	where
	\begin{equation*}
		\widetilde{\delta}_\kappa
		\ndef
		\left\famiAtl{
			\min\Bigl(\tfrac{(1 - \sigma)^2}{1 + \sigma} \delta_\kappa,
			\tfrac{\eps_\kappa}{4 (\BndSndAblRex[\RMetLok{g}{\kappa}]{ V_\kappa }{\delta_\kappa}  + 1) }
			\Bigr)
		\right}{\atlasA}
	\end{equation*}
	and each $\delta_\kappa \in]0,  \RadExpFibInv[\RMetLok{g}{\kappa}]{V_\kappa}{\sigma} \QuotNorm[\RMetLok{g}{\kappa}]{V_\kappa}[$ with some $\sigma \in ]0, 1[$.
	Further, $\GewFunk^e \sub \cl{\R}^M$ is locally $\GewFunk$-bounded
	and a \minSatExt{} of $\GewFunk \cup \sset{\omega}$ with respect to $((\atlasA, \atlasB, g), (\atlasA, \atlasA))$.
	The map
	\[
		\mathcal{D}_D \cap \mathcal{D}_D^{-1} \to \CFM{M}{\GewFunk^e}{\infty}{\atlasB}
		: \phi \mapsto \Rlog{g} \circ (\id{M}, \phi)
	\]
	is a chart for $\DiffWMg$.
\end{satz}
\begin{proof}
	We use \refer{lem:construction_adjusted_weight_MFK} to construct a weight $\omega : M \to \R$
	that is adjusted to $(\atlasB, \famiAtl{\widetilde{\delta}_\kappa}{\atlasA})$.
	Note that $\weightAtl{\omega}{\atlasB}$, after an eventual multiplication of $\omega$ with a constant, is also adjusting for
	$
		\famiAtl{\tfrac{\min(1, \eps_\kappa r_\kappa)}{2 \BndFstAblRex[\RMetLok{g}{\kappa}]{ V_\kappa }{\delta_\kappa}}}{\atlasA}
	$
	since $\inf_{\kappa \in \atlasA} r_\kappa \eps_\kappa > 0$ by our assumption,
	and $\BndFstAblRex[\RMetLok{g}{\kappa}]{ V_\kappa }{\delta_\kappa} \leq 1 + \sigma$
	by \refer{lem:Estimates_Rex_Rlog_DomRlog}.
	Further, we see with \refer{lem:adjusting_weight_log_from_weight_exp} that
	there exists an adjusted weight $\omega^L$ such that $\omega$ and $\omega^L$ satisfy \ref{est:adjustWeights_exp-log}
	(we may assume w.l.o.g. that $\abs{\omega} \geq \tfrac{1 + \sigma}{1 - \sigma}$).
	Since $\omega$ is locally $1_M$-bounded, $\GewFunk \cup \sset{\omega}$
	is locally $\GewFunk$-bounded,
	and so is the \minSatExt{} $\GewFunk^e$ of $\GewFunk \cup \sset{\omega}$ w.r.t. $((\atlasA, \atlasB, g), (\atlasA, \atlasA))$
	that was constructed in \refer{lem:Erzeugung_Gewichtsmengen_MFK}.
	We get the desired result by applying \refer{lem:Liegrupp_DiffeosMFK_geeigneteGewichte}.
\end{proof}

\subsubsection{Inclusion of compactly supported diffeomorphisms}

We want to examine which assumptions on the weight set $\GewFunk$ ensure
that the group $\DiffWMg$ contains the identity component $\DiffcM_0$ of the group of compactly supported diffeomorphisms.
To this end, we need some tools to handle the topology on the compactly supported vector fields,
which are the modelling space of $\DiffcM_0$.%
\paragraph{Sums and the topology of $\boldsymbol{\DCcInf{M}{\Tang{M}}}$}
We use tools provided in the article \cite{arxiv-math-0408008v1}.
\begin{bem}\label{bem:topology_CcM}
	For a $d$-dimensional manifold $M$, the smooth vector fields with compact support $\DCcInfM{M}$
	are usually endowed with the inductive limit topology of the inclusion maps
	$\CF{M}{\Tang{M}}{K}{\infty} \to \DCcInfM{M}$.
	Here $\CF{M}{\Tang{M}}{K}{\infty}$ denotes the smooth vector fields $X$ with $\supp{X} \sub K$,
	and is endowed with the topology of uniform smooth convergence with respect to charts,
	see \cite[Def. F.14  and Def. F.7 \& La. F.9]{arxiv-math-0408008v1} for details.

	By \cite[Prop. F.19]{arxiv-math-0408008v1}, for a locally finite atlas $\atlasA = \sset{\kappa : \widetilde{U}_\kappa \to U_\kappa}$
	such that each $\widetilde{U}_\kappa$ is relatively compact, the map
	\[
		\DCcInfM{M}
		\to
		\bigoplus_{\kappa \in \atlasA}
		\ConDiff{\widetilde{U}_\kappa}{\Tang{\widetilde{U}_\kappa}}{\infty}
		:
		X \mapsto \famiAtl{\VectFLok{X}{\kappa}}{\atlasA}
	\]
	is an embedding. The sum is endowed with the box topology,
	see \cite[6.1-6.7 and Def. F.7 \& La. F9]{arxiv-math-0408008v1} for the definition of the sum
	respectively the topology of the summands;
	we will use these seminorms.

\end{bem}
For an easier argument, we relate the sums $\bigoplus_{i \in I} \ConDiff{U_i}{Y_i}{\ell}$
and $\bigoplus_{i \in I} \BC{V_i}{Y_i}{\ell}$, provided that $V_i \sub U_i$ is relatively compact.
\begin{lem}\label{lem:Abb_SummeCnachSummeBC_stetig}
	Let $I$ a nonempty set and $\ell \in \cl{\N}$. For each $i \in I$, let $U_i, V_i$ be open nonempty subsets
	of the locally convex space $X_i$ such that $\cl{V_i} \sub U_i$ and each $V_i$ is relatively compact,
	and $Y_i$ a normed space.
	Then for each $i \in I$, the map
	\[
		\ConDiff{U_i}{Y_i}{\ell} \to \BC{V_i}{Y_i}{\ell}
		:
		\gamma \mapsto \rest{\gamma}{V_i}
	\]
	is defined and continuous, where each $\ConDiff{U_i}{Y_i}{\ell}$ is endowed with the compact open $\ConDiff{}{}{\ell}$ topology.
	Consequently, the map
	\[
		\bigoplus_{i \in I} \ConDiff{U_i}{Y_i}{\ell}
		\to
		\bigoplus_{i \in I} \BC{V_i}{Y_i}{\ell}
		:
		\fami{\gamma_i}{i}{I} \mapsto \fami{\rest{\gamma_i}{V_i}}{i}{I}
	\]
	is also defined and continuous.
\end{lem}
\begin{proof}
	According to \cite[Rem. 6.7]{arxiv-math-0408008v1},
	the spaces $\bigoplus_{i \in I} \ConDiff{U_i}{Y_i}{\ell}$ are the direct sum in the category of locally convex spaces,
	hence the second assertion follows if the first one is proved.
	Since we assumed that each $V_i$ is locally compact, each restricted map (and any of its derivatives) is bounded,
	and we see using standard compactness arguments that the restriction is continuous.
\end{proof}
We show that any locally bounded function induces continuous seminorms
on the sum $\bigoplus_{\kappa \in \atlasA} \BC{U_\kappa}{\R^d}{\infty}$.
\begin{lem}\label{lem:ContSN_sumBC}
	Let $d\in \N$, $M$ be $d$-dimensional manifold, $f : M \to \R$ locally bounded, $\ell \in \N$
	and $\atlasA = \sset{\kappa : \widetilde{U}_\kappa \to U_\kappa}$ a locally finite atlas such that
	each $\widetilde{U}_\kappa$ is relatively compact.
	Then $\hnM{\cdot}{f}{\ell}{\atlasA}$ is a continuous seminorm on
	$\bigoplus_{\kappa \in \atlasA} \BC{U_\kappa}{\R^d}{\infty}$.
\end{lem}
\begin{proof}
	Since $f$ is locally bounded, it is bounded on each compact set,
	and in consequence on each $\widetilde{U}_\kappa$, which can be proved with a standard compactness argument.
	So for $\kappa \in \atlasA$ and $\gamma \in \BC{\widetilde{U}_\kappa}{\R^d}{\infty}$,
	we have that
	\[
		\hn{\gamma}{\weightLok{f}{\kappa}}{\ell} \leq \noma{\weightLok{f}{\kappa}} \hn{\gamma}{1_{U_\kappa}}{\ell}.
	\]
	Hence $\hnM{\cdot}{f}{\ell}{\atlasA}$ is continuous since it is so on each summand.
\end{proof}
\paragraph{Inclusion of compactly supported diffeomorphisms}
We are ready to prove the criterion.
\begin{prop}\label{prop:Inclusion_of_compactly_supported_diffeomorphisms}
	Let $d \in \N^*$, $(M, g)$ a $d$-dimensional Riemannian manifold
	and $\atlasA = \sset{\kappa : \widetilde{U}_\kappa \to U_\kappa}$ a locally finite atlas for $M$
	such that there exists a point in $M$ that is contained in only one $\widetilde{U}_\kappa$
	and that is adapted to some $(\famiAtl{r_\kappa}{\atlasA}, \famiAtl{\eps_\kappa}{\atlasA}, R)$
	with $\inf_{\kappa \in \atlasA} \eps_\kappa, \inf_{\kappa \in \atlasA} r_\kappa > 0$.
	Further, let $\GewFunk \sub \R^M$ with $1_M \in \GewFunk$ such that each $f \in \GewFunk$
	is bounded on all compact subsets of $M$.
	Then $\DiffcM_0 \sub \DiffWMg$ for all $\atlasB$ and $\omega$ as in \refer{satz:Existenz_Liegruppe_Diffeos_MFK}.%
\end{prop}
\begin{proof}
	For relatively compact $V_\kappa$ such that $\cl{V_\kappa} \sub U_\kappa$ and
	$\clBall{0}{r_\kappa + R} \sub V_\kappa$,
	the map
	\[
		\DCcInfM{M} \to \bigoplus_{\kappa \in \atlasA} \ConDiff{V_\kappa}{\R^d}{\infty}
		: X \mapsto \famiAtl{\VectFLok{X}{\kappa}}{\atlasA}
	\]
	is an embedding, see \refer{bem:topology_CcM}.
	Since $\GewFunk^e$ is locally $\GewFunk$-bounded and each weight in $\GewFunk$ is locally bounded,
	each $f \in \GewFunk^e$ is also locally bounded.
	Hence we can use \refer{lem:ContSN_sumBC} and \refer{lem:Abb_SummeCnachSummeBC_stetig}
	to see that for $f \in \GewFunk^e$ and $\ell \in \N$,
	$\hnM{\cdot}{f}{\ell}{\atlasB}$ is defined and continuous
	on $\bigoplus_{\kappa \in \atlasA} \ConDiff{V_\kappa}{\R^d}{\infty}$
	and hence on $\DCcInfM{M}$.
	This, together with \refer{cor:gewVF-verschiedeneAtlanten-Chart_Changes_Weights},
	implies that $\DCcInfM{M} \sub \CFM{M}{\GewFunk^e}{\infty}{\atlasA}$,
	and that for each $\alpha > 0$,
	\[
		\set{ X \in \DCcInfM{M} }%
						{ \hnM{X}{\omega}{0}{\atlasB},
									\hnM{X}{1_M}{1}{\atlasB} < \alpha }
	\]
	is open in $\DCcInfM{M}$. We know from \refer{satz:Existenz_Liegruppe_Diffeos_MFK} that $\DiffWMg$ is modelled on $\CFM{M}{\GewFunk^e}{\infty}{\atlasA}$,
	and for some $\alpha > 0$, it contains the set
	\[
		\set{\Rexp{g} \circ\, X }%
							{X \in \CFM{M}{\GewFunk^e}{\infty}{\atlasA}, \hnM{X}{\omega}{0}{\atlasB},
										\hnM{X}{1_M}{1}{\atlasB} < \alpha }
		.
	\]
	Hence $\DiffWMg$ contains an open identity neighborhood of $\DiffcM$,
	and thus $\DiffcM_0$.
\end{proof}

\subsubsection{Comparison with the vector space case}

We show that the connected components of the Lie groups $\DiffWvr{\R^d}{\SkaPrd{\cdot}{\cdot}}$
that were constructed in \refer{satz:Existenz_Liegruppe_Diffeos_MFK},
and of $\Diff{\R^d}{}{\GewFunk}$ as constructed in \cite[Thm.~4.2.17]{MR2952176} coincide,
if $\atlasA$ consists of identity maps.
\begin{prop}\label{prop:Comparison_DiffLieGroups_VectorS}
	Let $d \in \N^*$ and $\GewFunk \sub \cl{\R}^{\R^d}$ with $1_{\R^d} \in \GewFunk$.
	Then $\Diff{\R^d}{}{\GewFunk}_0 = \DiffWvr{\R^d}{\SkaPrd{\cdot}{\cdot}}_0$,
	where
	\[
		\atlasA \ndef \set{\id{\Ball{x}{r1}}}{x \in \Z^d}
		\text{  and  }
		\atlasB \ndef \set{\id{\Ball{x}{r_2}}}{x \in \Z^d},
	\]
	with $1 \geq r_1 > r_2 > \tfrac{1}{2}$, and $\R^d$ is endowed with the supremum norm $\noma{\cdot}$.
\end{prop}
\begin{proof}
	Obviously, $\atlasA$ is a locally finite atlas since $\Ball{x}{1 - r_1}$ has nonempty intersection with at most
	$2^d$ chart domains, for all $x \in \R^d$.
	Further, if we set $R \ndef \tfrac{1}{2} (r_1 - r_2)$ and choose $\eps \in ]0, \tfrac{1}{2} \tfrac{r_1 - r_2}{r_1 + r_2}[$,
	$\atlasA$ is adapted to $(r_2, \eps, R)$.

	We have that $\domExpMax{\SkaPrd{\cdot}{\cdot}} = \domLogMax{\SkaPrd{\cdot}{\cdot}} = \R^{2d}$,
	and further that $\Rexp{\SkaPrd{\cdot}{\cdot}}(x, y) = x + y$ and $\Rlog{\SkaPrd{\cdot}{\cdot}}(x, y) = y - x$.
	Hence $\FAbl[2]{\Rexp{\SkaPrd{\cdot}{\cdot}}} = 0$, and for $x \in \Z^d$ and $\sigma \in ]0, 1[$,
	\[
		\grenzExp[\SkaPrd{\cdot}{\cdot}]{\Ball{x}{r_2}}{\Ball{x}{r_1}}
		= \RadExpFibInv[\SkaPrd{\cdot}{\cdot}]{\Ball{x}{r_2}}{\sigma}
		= r_1 - r_2
	\]
	and $\ExpMinusLok{\Ball{x}{r_2}}{\delta}{\SkaPrd{\cdot}{\cdot}} = \pi_2$
	for all $\delta \in ]0, r_1 - r_2[$.
	For $\grenzLog[\SkaPrd{\cdot}{\cdot}]{\Ball{x}{r_2}}{\Ball{x}{r_1}}$ we have
	\[
		\grenzLog[\SkaPrd{\cdot}{\cdot}]{\Ball{x}{r_2}}{\Ball{x}{r_1}}
		= \tfrac{1}{\sqrt{d}} (r_1 - r_2)
	\]
	and
	$\LogPlusVR{\Ball{x}{r_2}}{\delta}{\SkaPrd{\cdot}{\cdot}} = \pi_2$
	for all $\delta \in ]0, \tfrac{1}{\sqrt{d}} (r_1 - r_2)[$.
	For $\kappa, \phi \in \atlasA$ with $(\kappa, \phi) \in \atProd{\atlasA}{\atlasA}$,
	\[
		\FAbl{(\kappa \circ \phi^{-1})} = \idco .
	\]
	We easily deduce that $\GewFunk$ is already saturated,
	and $1_M$ is adjusted if we choose the same
	$\delta < \tfrac{1}{\sqrt{d}} (r_1 - r_2) =
	\QuotNorm[\SkaPrd{\cdot}{\cdot}]{\Ball{x}{r_2}} \RadExpFibInv[\SkaPrd{\cdot}{\cdot}]{\Ball{x}{r_2}}{\sigma}$
	for all charts.
	Further, for all $f \in \GewFunk$, $\ell \in \N$ and $X \in \VecFields{\R^d}$,
	\[
		\hn{\pi_2 \circ X}{f}{\ell} = \hnM{X}{f}{\ell}{\atlasB},
	\]
	and hence $\CcFM{\R^d}{\infty}{\atlasB} \iso \CcF{\R^d}{\R^d}{\infty}$.
	Since the parameterization maps are also compatible, we see that $\DiffWvr{\R^d}{\SkaPrd{\cdot}{\cdot}}$
	contains an open subset of $\Diff{\R^d}{}{\GewFunk}$, and vice versa.
	Hence the assertion holds.
\end{proof}

\printbibliography

\end{document}